\documentclass[11pt]{amsart}

\usepackage{amsfonts,amssymb,enumerate,color,bm,hhline,makecell,array}
\usepackage{array}
\usepackage{mathrsfs}
\usepackage{comment}
\usepackage[all,ps,cmtip]{xy} 
\usepackage[normalem]{ulem}
\usepackage[colorlinks=true,citecolor=blue, urlcolor=blue, linkcolor=blue, pagebackref]{hyperref}
\usepackage{tikz-cd}
\usepackage{amscd}
\usepackage[shortlabels]{enumitem}
\usepackage{tensor}
\usepackage{tikz}
\usetikzlibrary{matrix,arrows}
\usepackage{extarrows}
\usepackage{subcaption}
\usepackage{mathtools}
\usepackage{comment}
\usepackage[noabbrev,capitalise]{cleveref}
\usepackage[textsize=tiny]{todonotes}
\usepackage{stmaryrd}

\usetikzlibrary{3d, calc}

\DeclareFontFamily{U}{wncy}{}
\DeclareFontShape{U}{wncy}{m}{n}{<->wncyr10}{}
\DeclareSymbolFont{mcy}{U}{wncy}{m}{n}
\DeclareMathSymbol{\Sha}{\mathord}{mcy}{"58}

\def\C{\ensuremath{\mathbb{C}}}

\def\G{\ensuremath{\mathbb{G}}}

\def\L{\ensuremath{\mathbb{L}}}

\def\P{\ensuremath{\mathbb{P}}}
\def\Q{\ensuremath{\mathbb{Q}}}
\def\R{\ensuremath{\mathbb{R}}}
\def\Z{\ensuremath{\mathbb{Z}}}

\def\sfq{\ensuremath{\mathsf q}}
\def\sfL{\ensuremath{\mathsf L}}
\def\sfR{\ensuremath{\mathsf R}}

\def\cA{\ensuremath{\mathcal A}}

\def\cC{\ensuremath{\mathcal C}}

\def\cD{\ensuremath{\mathcal D}}
\def\cV{\ensuremath{\mathcal V}}
\def\cE{\ensuremath{\mathcal E}}
\def\cF{\ensuremath{\mathcal F}}

\def\cH{\ensuremath{\mathcal H}}

\def\cK{\ensuremath{\mathcal K}}
\def\cL{\ensuremath{\mathcal L}}
\def\cM{\ensuremath{\mathcal M}}
\def\cN{\ensuremath{\mathcal N}}

\def\cP{\ensuremath{\mathcal P}}

\def\cS{\ensuremath{\mathcal S}}

\def\cU{\ensuremath{\mathcal U}}
\def\cV{\ensuremath{\mathcal V}}
\def\cX{\ensuremath{\mathcal X}}

\def\vv{\ensuremath{\mathbf v}}

\def\fg{\mathfrak g}

\def\sC{\ensuremath{\mathscr C}}
\def\sG{\ensuremath{\mathscr G}}
\def\sH{\ensuremath{\mathscr H}}
\def\sL{\ensuremath{\mathscr L}}
\def\sF{\ensuremath{\mathscr F}}
\def\sE{\ensuremath{\mathscr E}}
\def\sM{\ensuremath{\mathscr M}}

\def\sT{\ensuremath{\mathscr T}}

\def\sP{\ensuremath{\mathscr P}}
\def\sS{\ensuremath{\mathscr S}}
\def\sO{\ensuremath{\mathscr O}}

\def\sQ{\ensuremath{\mathscr Q}}

\def\phi{{\varphi}}

\DeclareMathOperator{\Aut}{Aut}
\DeclareMathOperator{\ch}{ch}

\DeclareMathOperator{\Def}{Def}
\DeclareMathOperator{\Db}{D\textsuperscript{\rm b}}
\DeclareMathOperator{\Dperf}{D_{\mathrm perf}}

\DeclareMathOperator{\End}{End}
\DeclareMathOperator{\Ext}{Ext}

\DeclareMathOperator{\Ann}{Ann}

\DeclareMathOperator{\Hom}{Hom}

\DeclareMathOperator{\Ker}{Ker}

\DeclareMathOperator{\NS}{NS}
\DeclareMathOperator{\HT}{HT}
\DeclareMathOperator{\Pic}{Pic}

\DeclareMathOperator{\rk}{rk}
\DeclareMathOperator{\Supp}{Supp}

\DeclareMathOperator{\Amp}{Amp}
\DeclareMathOperator{\Span}{Span}

\DeclareMathOperator{\td}{td}
\DeclareMathOperator{\Br}{Br}

\DeclareMathOperator{\Prym}{Prym}

\def\Sym{\mathrm{Sym}}
\DeclareMathOperator{\bwed}{\bigwedge\nolimits}
\def\At{\mathrm{At}}

\def\Coh{\mathop{\mathrm{Coh}}\nolimits}

\def\ext{\mathop{\mathrm{ext}}\nolimits}

\def\id{\mathop{\mathrm{id}}\nolimits}

\def\Ker{\mathop{\mathrm{Ker}}\nolimits}
\def\NS{\mathop{\mathrm{NS}}\nolimits}

\def\rk{\mathop{\mathrm{rk}}}

\def\Tr{\mathop{\mathrm{Tr}}\nolimits}

\newcommand{\cEnd}{\mathcal{E}\!{\it nd}}

\newcommand{\cExt}{\mathcal{E}\!{\it xt}}
\def\cHom{\mathop{\mathcal Hom}\nolimits}

\newcommand{\srPic}[2]{\mathcal{P}\mathrm{ic}^{=}_{#1/#2}}
\newcommand{\srPico}[2]{\mathcal{P}\mathrm{ic}^{=0}_{#1/#2}}

\newcommand{\wt}{\widetilde}
\newcommand{\RcHom}{\mathrm{R}\mathcal{H}\!{\it om}}
\newcommand{\HK}{hyper-K\"ahler }

\DeclareFontFamily{U}{wncy}{}
\DeclareFontShape{U}{wncy}{m}{n}{<->wncyr10}{}
\DeclareSymbolFont{mcy}{U}{wncy}{m}{n}
\DeclareMathSymbol{\Sha}{\mathord}{mcy}{"58}

\DeclareMathOperator{\oP}{\overline{P}}

\newcommand{\mor}[1][]{\xrightarrow{#1}}
\newcommand{\isomor}{\mor[\sim]}

\newtheorem{Thm}{Theorem}[section]
\newtheorem{Prop}[Thm]{Proposition}

\newtheorem{Lem}[Thm]{Lemma}

\newtheorem{Cor}[Thm]{Corollary}

\newtheorem*{Ass*}{Assumption}

\newtheorem*{Con*}{Conjecture}
\theoremstyle{definition}
\newtheorem{Defi}[Thm]{Definition}

\newtheorem{Rem}[Thm]{Remark}

\allowdisplaybreaks

\usepackage{enumitem}
\setlist[itemize]{noitemsep,nolistsep}
\setlist[enumerate]{noitemsep,nolistsep}
\usepackage{colortbl}
\usepackage{enumerate}

\begin{document}

\definecolor{xdxdff}{rgb}{0.49019607843137253,0.49019607843137253,1}
\definecolor{uuuuuu}{rgb}{0.26666666666666666,0.26666666666666666,0.26666666666666666}
\definecolor{ffqqqq}{rgb}{1,0,0}

\title{O'Grady's tenfolds from stable bundles on hyper-K\"ahler fourfolds}

\author[Alessio Bottini]{Alessio Bottini}

\address{Max Planck Institute for Mathematics, Vivatsgasse 7, 53111 Bonn, Germany}\email{bottini@mpim-bonn.mpg.de}

\makeatletter
\@namedef{subjclassname@2020}{%
  \textup{2020} Mathematics Subject Classification}
\makeatother
\keywords{Hyper-K\"ahler manifolds, stable vector bundles, moduli spaces.}
\subjclass[2020]{14J42, 14F08, 14J60, 53D30.}

\begin{abstract}
We provide a modular construction of the Laza--Saccà--Voisin compactification of the intermediate Jacobian fibration of a cubic fourfold.
Additionally, we construct infinitely many $20$-dimensional families of polarized  hyper-K\"ahler manifolds of type OG10, realized as moduli spaces of stable bundles on hyper-K\"ahler manifolds of type $\mathrm{K3}^{[2]}$. 
\end{abstract}

\maketitle
\setcounter{tocdepth}{1}
\tableofcontents

\section{Introduction}
Almost 30 years ago, Donagi and Markman \cite[Theorem 8.1]{donagi96} showed that the structure sheaf of a smooth Lagrangian submanifold in a \HK manifold $Z \subset X$ has a smooth symplectic deformation space $\Def(\sO_Z)$.
They also asked whether the symplectic structure extends to the closure of the relative Picard scheme inside the moduli space of stable sheaves.
At least at the birational level this question can be addressed, even more generally, via the methods developed in \cite{sacca24}. 
In the particular case where $X$ is the variety of lines on a smooth cubic fourfold $Y$, and $Z$ is the surface of lines on a hyperplane section, these techniques yield the \HK compactification $J_Y \to \P^5$ already constructed in \cite{lsv17}.

If $\lambda$ is any polarization on $X=F(Y)$, the relative Picard is naturally included in the moduli space $M_{\vv_0}(X,\lambda)$ parametrizing $\lambda$-semistable sheaves with Mukai vector $\vv_0 = v(\sO_Z)$. 
For dimensional reasons, its closure is an irreducible component $M^{\circ}_{\vv_0}(X,\lambda)$.
In this paper we prove the following result, which gives a modular description of the LSV compactification $J_Y$. 

\begin{Thm}[\cref{cor:ModularLSV} and \cref{rem:ChangingPolarization}]\label{thm:IntroLagrangians}
    If the cubic $Y$ is general and $\lambda$ is any polarization on $X$, then $M^{\circ}_{\vv_0}(X,\lambda)$ is a connected component of $M_{\vv_0}(X,\lambda)$. 
    Moreover, it is a smooth HK manifold of type OG10 equipped with a Lagrangian fibration 
    \[
            f: M^{\circ}_{\vv_0}(X,\lambda) \to \P^5, \ G \mapsto \Supp(G),
    \]
    and there is an isomorphism $M^{\circ}_{\vv_0}(X,\lambda) \cong J_Y$ compatible with the Lagrangian fibrations.
\end{Thm}

If the cubic fourfold $Y$ is special---in the sense that its variety of lines $X$ lies in a certain Noether--Lefschetz divisor $\cN(d) \subset K^2_6$---we can find a derived equivalence 
\[
    \Db(X) \isomor \Db(X',\theta_d),\text{ with } \theta_d \in \Br(X'),
\]
sending every sheaf in $M^{\circ}_{\vv_0}(X,h)$ into a $\theta_d$-twisted, slope stable, vector bundle on another HK manifold $X'$. 
If $\vv_d$ denotes their twisted Mukai vector (in the sense of \cite{huybrechts2005}), we obtain an isomorphism 
\[
    M^{\circ}_{\vv_0}(X,h) \cong M^{\circ}_{\vv_d}(X',h')
\]
with (a connected component of) a moduli space of $\mu_{h'}$-stable twisted locally free sheaves. 
As images of atomic sheaves via a derived equivalence, these bundles are also atomic, hence projectively hyperholomorphic.
Thus they can be deformed along twistor paths to any HK manifold deformation equivalent to $X'$.

If we perform such deformations until we reach a HK manifold where the Mukai vector $\vv_d$ becomes algebraic, the deformed sheaves will be untwisted. 
The atomicity condition implies, after possibly making the $c_1$ positive, the Mukai vector $\vv_d$ stays algebraic along a $20$-dimensional moduli space $\cK_d$ of polarized HK manifolds (defined precisely in \cref{sec:CompactificationSecond}).

\begin{Thm}[\cref{thm:ModuliOfVB} and \cref{cor:FamiliesOfOG10}]\label{thm:IntroBundles}
    For infinitely many non-negative integers $d$, there is a quasi-projective variety $\cS_d$ with a quasi-finite morphism $\cS_d \to \cK_d$, and a smooth projective morphism $\cM_d \to \cS_d$ whose fibers are HK manifolds of type OG10. 
    If $s \in \cS_d$ is general, and $(W,h_d)$ is its image in $\cK_d$, the fiber $\cM_d|_s$ is a connected component in a moduli space of $h_d$-semistable sheaves on $W$, and it parametrizes only $\mu_{h_d}$-stable bundles of rank $5d^2$. 
\end{Thm}

The above families are constructed by taking the Stein factorization of the relative moduli space over $\cK_d$.
The reason for this is that we do not know if the moduli spaces of sheaves on high-dimensional HK manifolds are connected. 
A priori, it could happen that given $(W,h_d) \in \cK_d$ there are many different smooth components in $M_{\vv_d}(W,h_d)$, and which one we get depends on the chosen twistor path. 

\subsection{Symplectic forms}
The tangent space to a moduli space of semistable sheaves at a stable point $E$ is identified with $\Ext^1(E,E)$.
On a K3 surface, as first observed in \cite{mukai84}, it can be equipped with the pairing
\[
\Ext^1(E,E) \times \Ext^1(E,E) \to \C, \ (a,b) \mapsto \Tr(a \circ b). 
\]
This is alternating because the trace kills graded brackets, and it is non-degenerate by Serre duality. 

Now let $X$ be a projective HK manifold of dimension $2n$ and let $i: Z \hookrightarrow X$ be a smooth Lagrangian subvariety. 
In \cite[Section 8]{donagi96} the authors construct a symplectic structure on the relative Picard over the Hilbert scheme of deformations of $Z$, depending on the choice of a class in $H^2(X,\C)$.
In certain cases, the local-to-global spectral sequence degenerates, giving an algebra isomorphism 
\[
    \Ext^*(i_*L,i_*L) \cong H^*(Z,\C),
\]
where $L \in \Pic(Z)$.
This happens, for example, if $Z$ is the zero locus of a section of a globally generated vector bundle or if $X$ has dimension 4 by \cite{mladenov19}.
If this happens, the symplectic form associated to a K\"ahler class $\omega \in H^2(X,\R)$ is described as:
\[
    H^1(Z,\C) \times H^1(Z,\C) \to \C,\ (a,b) \mapsto \int_Z{\omega^{n-1}|_Z \cup a \cup b}.
\]
It is non-degenerate by a combination of Poincaré duality and the Hard Lefschetz Theorem for $\omega|_Z$. 

A similar construction works if $E$ is a projectively hyperholomorphic bundle. 
In that case, the class $\overline{\sigma} \in H^2(X,\sO_X)$ plays the role of the K\"ahler class. 
We denote by $L_{\overline{\sigma}} = \overline{\sigma} \otimes \id_E \in \Ext^2(E,E)$, so we can consider the pairing
\[
    \Ext^1(E,E) \times \Ext^1(E,E) \to \C, \ (a,b) \mapsto \Tr(L^{n-1}_{\overline{\sigma}} \circ a \circ b),
\]
where $\circ$ denotes the Yoneda pairing on $\Ext^*(E,E)$. 
It is skew-symmetric by properties of the trace, and non-degenerate by a combination of Serre duality and the holomorphic Hard Lefschetz for hyperholomorphic sheaves proved by Verbitsky, see \cref{thm:HardLefschetzVerbitsky}.
In this case, if we identify $H^2(X,\sO_X) \cong \C$, the symplectic form above is just $\Tr(a \circ b)$, since $\overline{\sigma}$ acts as multiplication by a scalar and can be taken out of the trace.
It is important to observe that we do not need to assume that $\Def(E)$ is smooth.

These two constructions can be unified via the obstruction map, as explained in \cite[Definition 7.3]{bottini22}.
For any $E \in \Db(X)$, Markman \cite{markman21} defines the obstruction map 
\begin{equation*}
    \chi_{E} : \HT^2(X) \to \Ext^2(E,E), \  \eta  \mapsto \eta \lrcorner \exp(\At_E). 
\end{equation*}
by contraction with the exponential of the Atiyah class.
Then for every $\eta \in \HT^2(X)$ we define a pairing
\begin{equation*}
    \Ext^1(E,E) \times \Ext^1(E,E) \to \C, \ (a,b) \mapsto \Tr( \chi_E(\eta)^{n-1} a \circ b). 
\end{equation*}
If one chooses $\eta$ to be respectively a K\"ahler class or $\overline{\sigma}$ one recovers precisely the two cases above. 
At the moment, there are no other examples where it is known if this is non-degenerate.

If $W \in \cK_d$ and $M^{\circ}$ is a smooth component as in \cref{thm:IntroBundles} we can also give an explicit global description of the symplectic form. 

\begin{Thm}[\cref{thm:SmoothnessOurCase}]\label{thm:IntroExtSheaves}
    Let $\pi : W \times M^{\circ} \to  M^{\circ}$ be the second projection, and let $\sE \to W \times  M^{\circ}$ be a twisted universal family. 
    The Yoneda pairing induces an isomorphism 
    \[
        \bwed^2 \cExt^1_{\pi}(\sE,\sE) \cong \cExt_{\pi}^2(\sE,\sE)
    \] 
    Moreover, via this isomorphism the trace map $\Tr: \cExt^2_{\pi}(\sE,\sE) \to \sO_{M^{\circ}}$ gives the symplectic form. 
\end{Thm}

This exhibits an interesting phenomenon that does not occur for K3 surfaces.
Namely, since $\cExt^1_{\pi}(\sE,\sE) = T_{M^{\circ}}$, the sheaf $\cExt^2_{\pi}(\sE,\sE)$ controls the symplectic forms on the moduli space.  
This opens the way to the possibility of directly showing that a symplectic moduli space has a unique symplectic form, without resorting to a deformation argument. 
In this paper, we rely on \cite{lsv17} for this and to determine the deformation type. 

\subsection{Atomic, modular, and hyperholomorphic bundles}
Hyperholomorphic bundles were introduced by Verbitsky in \cite{verbitsky96b} as those bundles which remain holomorphic on every complex structure in a twistor line.  
If $\omega$ is a K\"ahler class and $E$ is hyperholomorphic with respect to $\omega$, then it is $\mu_{\omega}$-polystable.
Conversely, by \cite[Theorem 3.19]{verbitsky99} a $\mu_{\omega}$-polystable bundle $E$ is hyperholomorphic if and only if $c_1(E)$ and $c_2(E)$ remain of Hodge type along the twistor line induced by $\omega$.
In practice, it is useful to extend this definition and consider also \textit{projectively hyperholomorphic} bundles.
A (possibly twisted) bundle $E$ is projectively hyperholomorphic if $\cEnd(E)$ is hyperholomorphic (or equivalently, if $\P(E)$ deforms along a twistor line).

More recently, O'Grady \cite[Definition 1.1]{ogrady22} introduced modular sheaves. 
We will not need the precise definition, but recall that if, for a torsion-free sheaf $E$, the discriminant 
\[
\Delta(E) = -2r(E)\ch_2(E) + c_1(E)^2
\]
remains of Hodge type over all K\"ahler deformations of $X$\footnote{Such sheaves are referred to as very modular or strongly modular in the literature.}, then $E$ is modular.
In \cref{sec:TransformsOfLagrangians}, we will use the theory of slope stability for modular sheaves, developed in \cite{ogrady22}. 

Markman \cite{markman21} and Beckmann \cite{beckmann22} introduced independently the notion of atomic objects. 
An object $E \in \Db(X)$ is atomic if 
\[
\Ann(v(E)) = \Ann(\tilde{v}(E)) \subset \fg(X),
\]
for some 
\[
0 \neq \tilde{v}(E) \in \wt{H}(X,\Q) \coloneqq \Q\alpha \oplus H^2(X,\Q) \oplus \Q\beta.
\]
Here $\fg(X) \subset \End(H^*(X,\Q))$ is the Looijenga--Lunts--Verbitsky (LLV) Lie algebra, which acts on $ \wt{H}(X,\Q)$ via the isomorphism $\fg(X) \cong  \mathfrak{so}(\wt{H}(X,\Q))$ of  \cite{verbitsky96a,looijengaLunts97} (see also \cite[Theorem 2.7]{green22}).
The class $\tilde{v}(E)$ is called the extended Mukai vector of $E$ and, thanks to \cite[Theorem A]{taelman21}, is preserved under derived equivalences (in particular atomicity is preserved as well). 

This is the strongest of the three conditions: by \cite[Lemma 5.4]{beckmann22}, an atomic torsion-free sheaf is modular, and if it is a $\mu$-polystable vector bundle it is also projectively hyperholomorphic. 
On the other hand, atomicity is well-defined for torsion sheaves as well. 
Most notably, a topologically trivial line bundle supported on smooth Lagrangian subvariety is atomic if and only if a certain geometric condition is satisfied \cite[Theorem 1.8]{beckmann22}. 
This gives a way of producing projectively hyperholomorphic bundles by starting with such a line bundle and transforming it into a $\mu$-polystable locally free sheaf via a derived equivalence. 

\subsection{Strategy of the proof}
The proofs of \cref{thm:IntroLagrangians} and \cref{thm:IntroBundles} are intertwined. 
Indeed, as it turns out, it is easier to understand the singularities of a moduli space if it parametrizes only projectively hyperholomorphic vector bundles than if it parametrizes sheaves supported on Lagrangian subvarieties. 
Thus, we study the irreducible component $M^{\circ}_{\vv_0}(X,h)$ by explicit geometric considerations, but, to understand the extra structure that the moduli space may have along this component, we use the isomorphism with the moduli space $M^{\circ}_{\vv_d}(X',h')$. 
Being induced by a derived equivalence, it sends smooth points to smooth points.
So, if we prove that $M^{\circ}_{\vv_d}(X',h')$ is a smooth connected component, the same is true for $M^{\circ}_{\vv_0}(X,h)$.   

\subsubsection{The irreducible component}
Let $\cF \subset X \times \P^5$ be the universal surface of lines on hyperplane sections. 
Taken with its reduced structure, the component $M^{\circ}_{\vv_0}(X,h)$ is isomorphic to the closure of line bundles $\overline{\Pic^0}(\cF/\P^5)$ in the relative moduli space of torsion-free sheaves. 
In \cref{sec:compactifiedPic}, we study this space under the assumption that $Y$ is general, so that, for every hyperplane section $Y_H$, the corresponding surface of lines is integral. 
We prove the following two facts:
\begin{enumerate}
    \item[(i)] There is an isomorphism $\overline{\Pic^0}(\cF/\P^5) \isomor J_Y$ extending the natural birational map. 
    \item[(ii)] Every sheaf in $\overline{\Pic^0}(\cF/\P^5)$ is Cohen--Macaulay. 
\end{enumerate}
In particular, we deduce from (i) that $M^{\circ}_{\vv_0}(X,h)$ is smooth with the reduced structure. 
If we only wanted to prove this, we would not need the descent result of \cite{lsv17} but only the smoothness of the relative compactified Prym variety $\oP_{\cF_0}$. 
In turn, this relies on the smoothness criterion \cite[Corollary B.2]{fantechi1999} for the relative compactified Jacobian of a family of planar curves, together with the transversality arguments of \cite{lsv17}.

\subsubsection{The derived equivalence}
In \cref{sec:TwistedPoincaré} we prove the following result, which generalizes \cite[Proposition 2.1]{addington16} to every projective Tate--Shafarevich twist. 

\begin{Thm}[{{\cref{thm:TwistedPoincaré}}}]
    Let $(S,H)$ be a polarized K3 such that every curve in $|H|$ is integral. 
    Let $X' \to \P^n$ be the relative compactified Jacobian of the linear system $|H|$. 
    Then, for every projective Tate--Shafarevich twist $X \to \P^n$ given by an element of $\Sha^0$, there exists a Brauer class $\theta \in \Br(X')$ and a twisted Cohen--Macaulay sheaf 
    \[
    \sP \in \Coh\left(X \times X',p_2^*\theta\right),
    \]
    inducing a derived equivalence $ \Phi_{ \sP }: \Db(X) \isomor \Db(X',\theta)$.
\end{Thm}

If $X$ is a very general element inside certain (infinitely many) Noether--Lefschetz loci $\cN(d) \subset K^2_6$, it is a Tate--Shafarevich twist of a Beauville--Mukai system $X' \to \P^2$ as in the statement. 
So we find a Brauer class $\theta_d \in \Br(X')$ and a derived equivalence 
\[
    \Phi_{\sP}: \Db(X) \isomor \Db(X',\theta_d).
\]
It sends a sheaf $G \in \Coh(X)$ to a twisted vector bundle on $X'$, provided that $G$ is Cohen--Macaulay and its support is finite over $\P^2$. 
If the sheaf $G$ is atomic, the resulting bundle will be a twisted atomic bundle---we study these in \cref{sec:TwistedAtomicSheaves}---therefore we can apply O'Grady's results \cite{ogrady22} to analyze its slope stability. 
It turns out that for \textit{suitable} polarizations $h' \in \Amp(X')$ (this terminology is introduced in \cite{ogrady22}), the images of the sheaves in $M^{\circ}_{\vv_0}(X,h)$ under $\Phi_{\sP}$ are $\mu_{h'}$-stable $\theta_d$-twisted bundles. 
The key reason is that all the sheaves in $M^{\circ}_{\vv_0}(X,h)$ have integral support.

\subsubsection{Smoothness of the moduli space}
To conclude the proof of \cref{thm:IntroLagrangians} it remains to exclude that on $M^{\circ}_{\vv_0}(X,h)$ meets other components or has embedded points. 
We prove that this is true on the ``vector bundle side" $M^{\circ}_{(X',h')}(\vv_d)$, using that it parametrizes only stable projectively hyperholomorphic bundles. 
Roughly speaking, we exploit the fact that the symplectic form on $\Ext^1(E,E)$ exists for $E \in M^{\circ}_{\vv_d}(X',h')$, even if (a priori) $E$ is not a smooth point. 

Our argument in \cref{sec:ModuliSpacesGeneral} is very elementary, and it requires no deformation theory. 
We directly prove that the sheaf $\cExt^3_{\pi}(\sE,\sE)$ is locally free, by analyzing the spectral sequence coming from Grothendieck--Verdier duality \eqref{eq:VerdierSpectralSequence}.
Here $\sE \to X' \times M^{\circ}_{\vv_d}(X',h')$ is a twisted universal family, so that $\cExt^3_{\pi}(\sE,\sE)$ is the cotangent sheaf to the moduli space.  
The degeneration of the same spectral sequence also gives \cref{thm:IntroExtSheaves}.

As the argument is so straightforward, the assumptions needed to run it play a major role. 
The crucial one, which seems difficult to generalize to other cases, is that the singular locus is of codimension $3$ or higher.
Verifying this condition takes up \cref{sec:NodalCubics}.
We do it by showing, on ``Lagrangian side", that sheaves whose support is at worst nodal are smooth points for $M^{\circ}_{\vv_0}(X,h)$.
Interestingly, even in this case, the symplectic form helps avoid what would otherwise be a difficult geometrical analysis. 
Indeed, as the tangent space at any point is symplectic (because of the derived equivalence), it suffices to show that its dimension is at most $11$.
This can be done with relatively simple geometric considerations, but proving that it is $10$ is much more difficult.

With a slight variation of our arguments, we can prove the following result, which seems more suited to generalizations or to interaction with deformation-theoretic arguments (such as an analysis of the Kuranishi map).

\begin{Thm}[{{\cref{thm:SmoothnessGeneral}}}]\label{thm:IntroSmoothnessGeneral} 
    Let $M$ be a moduli space of (possibly twisted) stable projectively hyperholomorphic bundles on a HK manifold. 
    If $M$ is normal and has lci singularities, then it is smooth. 
\end{Thm}

\subsection{Related works}
This work is the natural continuation of \cite{bottini22}. 
There, we degenerated $Y$ to the determinantal cubic, and the Lagrangian surfaces degenerated to a reducible surface. 
The deformations of this reducible surface have even more components, making it significantly more difficult to analyze stability.
Instead, in this paper we degenerate to certain special smooth cubics and, at the minor cost of dealing with twisted sheaves, we get a moduli space with a projective component parametrizing only slope stable vector bundles. 

\cref{thm:IntroLagrangians} does not supersede \cite{lsv17}, although, as mentioned before, the smoothness part of our argument relies on it in a limited way.
Specifically, we do not need the descent, but only the transversality arguments combined with \cite[Corollary B.2]{fantechi1999}. 
However, once smoothness is established we still have to show that $M^{\circ}_{\vv_0}(X,h)$ is HK of type OG10.
Since $M^{\circ}_{\vv_0}(X,h)$ has a natural symplectic form, this amounts to a deformation argument.
It is likely that this can be done by degenerating to the chordal cubic, but, as the moduli space can degenerate in a bad way, the argument would probably end up needing \cite{klsv18}. 
Future developments may bring a fully modular description of an explicit degeneration to O'Grady's example.
However, until then, \cite{lsv17} remains necessary to establish that $M^{\circ}_{\vv_0}(X,h)$ is HK of type OG10. 

In \cite{ogrady24}, other (irreducible components of) moduli spaces on $\mathrm{K3}^{[2]}$-type HK manifolds are constructed.
These parametrize modular sheaves, which in general are not atomic. 
The only case in which they are atomic, one recovers a construction by Markman \cite[Section 11]{markman21}.
The biggest difference from our work is that it is expected that, if smooth, these components are again of type $\mathrm{K3}^{[n]}$.
Other examples of non-rigid modular (and in general not atomic) sheaves are given in \cite{fatighenti24,fatighentiOnorati24}.

In \cite{guo24} a similar strategy is applied to many other Lagrangian subvarieties to produce projectively hyperholomorphic bundles. 
Understanding the compactifications of the corresponding moduli spaces still appears to be out of reach. 
For one of these cases, first constructed by Iliev--Manivel \cite{iliev08}, a non-modular compactification has been constructed in \cite{xu24} using techniques developed by Saccà \cite{sacca24}.

\subsection{Structure of the paper}
In \cref{sec:TwistedAtomicSheaves}, following Beckmann's and Markman's work, we extend the notion of atomic objects to twisted sheaves. 
We reprove in this setting all the relevant results that hold for untwisted objects.
Namely: they are invariant under (certain) derived equivalences, a twisted atomic bundle is modular, and, if slope polystable, it is also projectively hyperholomorphic. 

In \cref{sec:TwistedPoincaré} we construct the twisted Poincaré sheaf and prove \cref{thm:TwistedPoincaré}.
We describe the cohomological action of the induced equivalence, and show that it sends atomic objects to twisted atomic objects. 

In \cref{sec:compactifiedPic} we study the relative compactified picard $\overline{\Pic^0}(\cF/\P^5)$, for a general cubic $Y$.
We prove that, with its reduced structure, it is isomorphic to $J_Y$, and that it parametrizes only Cohen--Macaulay sheaves. 

In \cref{sec:NodalCubics} we study the compactified Picard of surfaces of lines on nodal cubics, and we show that it behaves in a similar way to that of nodal curves.
We focus especially in the computation of the tangent space to $M_{\vv_0}(X,h)$ at a point in the component $M^{\circ}_{\vv_0}(X,h)$.
Many of the results are proven, more generally, for singular varieties which have only normal crossings singularities in codimension one. 

In \cref{sec:TransformsOfLagrangians} we study under which assumptions a sheaf is sent to a twisted stable bundle via the equivalence of \cref{thm:TwistedPoincaré}, generalizing \cite[Proposition 5.2]{bottini22}. 
For this we extend the notion of modular sheaves to twisted bundles, and generalize some of O'Grady's work to this setting. 

In \cref{sec:ModuliSpacesGeneral} we study the geometry of moduli spaces of twisted hyperholomorphic bundles.
We describe the symplectic form in the twisted case, and prove two results on the smoothness of moduli spaces of stable projectively hyperholomorphic bundles. 
One is \cref{thm:IntroSmoothnessGeneral}, and the other \cref{thm:SmoothnessOurCase} is a more particular criterion which we apply to our situation. 
We also essentially prove \cref{thm:IntroExtSheaves}. 

In \cref{sec:CompactificationSecond} we conclude the proofs of our main results.
We verify that the assumptions of  \cref{thm:SmoothnessOurCase} to show smoothness of $M^{\circ}_{\vv_d}(X',h')$ and deduce it for $M^{\circ}_{\vv_0}(X,h)$. 
We also deform $M^{\circ}_{\vv_d}(X',h')$ along twistor lines to construct smooth components of moduli spaces of untwisted bundles.

\subsection{Notations and conventions}
\begin{itemize}
    \item Moduli spaces of sheaves are considered with respect to Gieseker (semi)stablity. 
    We will always highlight explicitly if the sheaves are also slope stable. 
    \item Irreducible components of moduli spaces are taken with their reduced scheme structure.  
    \item General means in a Zariski open, very general means in the complement of countably many closed subsets.  
    \item If a functor is derived, it is explicitly denoted by $\sf{R}-$ or $\sf{L}-$, with the exception of the derived dual which is denoted $-^{\vee}$. 
    To distinguish between the derived dual and the standard one, we denote the latter by $-^*$. 
    This will be relevant mostly in \cref{sec:ModuliSpacesGeneral}.
    \item When talking about Cohen--Macaulay sheaves we always work over Gorenstein schemes. 
    In that case, $F$ Cohen--Macaulay if and only if its derived dual $F^{\vee}$ is a sheaf (appropriately shifted by the codimension of its support).
    A twisted sheaf is Cohen--Macaulay if locally is Cohen--Macaulay. 
    \item The terms  ``locally free sheaf" and  ``vector bundle" are used interchangeably throughout. 
\end{itemize}

\noindent
\textbf{Acknowledgments.}
During this work, I greatly benefited from many useful discussion from many people. 
I wish to thank: Federico Bongiorno, Riccardo Carini, Franco Giovenzana, Daniel Huybrechts, Eyal Markman, Emanuele Macrì, Luigi Martinelli, Kieran O'Grady, Antonio Rapagnetta, Giulia Saccà, Richard Thomas. 
I have been partially supported by the European Research Council (ERC) under the European Union’s Horizon 2020 research and innovation programme (ERC-2020-SyG-854361-HyperK).

\section{Twisted atomic objects}\label{sec:TwistedAtomicSheaves}

\subsection{Review: Brauer group and twisted sheaves}\label{sec:ReviewBrauer}
If $X$ is an analytic space, given an element $\alpha \in H^2(X,\sO^*_X)$ represented by a \v{C}ech cocycle $\{ \alpha_{ijk} \}$, there is the notion of $\alpha$-twisted sheaves, introduced in \cite{caldararu00}.
Roughly speaking, an $\alpha$-twisted sheaf $F$ is a collection of sheaves $F_i$ over an open cover $\{ U_i \}$ plus transition functions $\phi_{ij}: F_i|_{U_{ij}} \cong F_{j}|_{U_{ij}}$ which satisfy the cocycle condition up to $\alpha_{ijk}$. 
Recall that the Brauer group is the torsion part $\Br(X) =  H^2(X,\sO^*_X)_{\rm{tor}}$. 

Later we will use the following equivalent description of Brauer classes and twisted sheaves, we refer to \cite[Section 1.1]{caldararu00} for details. 
To an element $\alpha \in H^2(X,\sO^*_X)$ one associates an open cover $\{ U _i\}$ of $X$, and a collection of line bundles $\{L_{ij}\}$ over the intersection $U_{ij} = U_i \cap U_j$, satisfying a specific cocycle condition.
Part of this cocycle condition is that, over the triple intersection $U_{ijk}$, the bundle
\[
L_{ijk} \coloneqq L_{ij} \otimes L_{jk} \otimes L_{ki}
\]
is non-canonically trivial. 
To recover the first description, one refines the cover to trivialize the $L_{ij}$. 
Over this refined cover, the usual cocycle $\alpha_{ijk}$ is obtained by choosing trivializations for $L_{ijk}$. 

With this interpretation an $\alpha$-twisted sheaf as collection of sheaves $\{ U_i, F_i \}$, together with isomorphisms
\[
F_{i}|_{U_{ij}} \cong F_{j}|_{U_{ij}} \otimes L_{ij},
\]
compatible with the cocycle condition of the $L_{ij}$.

\textbf{Twisted Chern character.}
There are various ways in the literature to define a Chern character for a twisted sheaf \cite{huybrechts2005,lieblich07}, and they all depend on more data than just the Brauer class.  
We follow \cite{huybrechts2005} for the close analogy to K3 surfaces.

Let $X$ be a smooth projective variety over $\C$, and assume that $\alpha \in \Br(X)$ is topologically trvial, i.e. the image in $H^3(X,\Z)_{\mathrm{tor}}$ vanishes. 
This is always the case for $\mathrm{K3}^{[n]}$-type HK manifolds.
Then, we can choose a B-field $B \in H^2(Y,\Q)$ lifting $\alpha$, which means
\[
\exp(B^{0,2}) = \alpha \in H^2(Y,\sO^*_Y). 
\]
To construct a twisted Mukai vector, we will assume one of two additional conditions, depending on the context. 

One is, as in \cite{huybrechts2005}, that $\exp(B_{ijk}^{0,2}) = \alpha_{ijk}$ at the level of cocycles.
In this case, if we write
\begin{equation}\label{eq:TrivializationBField}
    B_{ijk} = a_{ij}  + a_{jk} - a_{ik}, \ \{ a_{ij}\} \in \Gamma^1(X,C^{\infty}),
\end{equation}
we get an $\alpha_{ijk}$-twisted topological line bundle $L$ given by the transition functions $\exp(a_{ij})$. 

The other is, instead, that $\alpha^d_{ijk} = 1$ at the level of cocycles, or equivalently, to have chosen a lift of $\alpha$ to $H^2(X,\mu_d)$.
In this case, following \cite[Lemma 5.5]{dejong_perry22}, we can take any $\alpha$-twisted topological line bundle $L$ on $X$. 
Then $L^d$ is an untwisted topological line bundle, and the B-field is given by $B = \frac{c_1(L^d)}{d}$.

Having fixed a B-field $B$, and the corresponding twisted topological line bundle $L$, we define a Chern character $\ch^B(-)$ for $\alpha$-twisted coherent sheaves in the following way.
First notice that tensor product with an $\alpha^{-1}$-twisted vector bundle shows that every $F \in \Coh(X,\alpha)$ admits a finite resolution by $\alpha$-twisted vector bundles \cite[Proposition 2.1.8]{caldararu00}.
If $E^{\bullet} \to F$ is any such resolution, we define
\begin{equation}\label{eq:DefTwistedChernChar}
\ch^B(F) \coloneqq \sum_i (-1)^i \ch(E^i \otimes L^{-1}),
\end{equation}
where the $E^i \otimes L^{-1}$ are untwisted topological vector bundles.
We extend it to bounded complexes by additivity in the usual way. 
The twisted Mukai vector is defined to be $v^B(F) \coloneqq \ch^B(F) \td_X^{1/2}$. 

\begin{Rem}\label{rem:TwistedChernIndependentFromRes}
    This definition is independent of the chosen resolution. 
    More generally, we can use any finite resolution of $F \otimes L^{-1}$ in the category of $C^{\infty}$-modules. 
    In other words, the Chern character depends only on the isomorphism class of $F \otimes L^{-1}$ as a $C^{\infty}$-module.
    This is because, since $C^{\infty}$ is a soft sheaf, an isomorphism in the derived category between any two resolutions can be lifted to an isomorphism up to homotopy, see \cite[Exposé II, Proposition 2.3.2(c)]{sga6}.
    A diagram chase then gives the desired independence. 
    This will be important in the proof of \cref{prop:CohomologicalAction}.
\end{Rem}

Following \cite[Remark 1.3]{huybrechts2005} we define a Hodge structure on $H^*(X,\Q)$ by 
\begin{equation}\label{eq:TwistedHodgeStructure}
    H(X,B)^{p,q} \coloneqq \exp(-B)H^{p,q}(Y).
\end{equation}
As shown in \cite[Proposition 1.2]{huybrechts2005}, for every $F \in \Coh(X,\alpha)$ one has 
\begin{equation}\label{eq:TwistedChern}
\ch^B(F) \in \bigoplus_{p} H(X,B)^{p,p}.
\end{equation}
Given a Fourier--Mukai equivalence $\Phi_{\sE} : \Db(X_1,\alpha_1) \isomor \Db(X_2,\alpha_2)$, one gets a Hodge isometry
\[
\Phi^H_{\sE} \coloneqq p_{2,*} \left(v^{-B_1 \boxplus B_2 }(\sE) \cup p^*_1(-)\right) : H^*(X_1,B_1,\Q) \cong H^*(X_2,B_2,\Q).
\]
Where $H^*(X,B,\Q)$ is the rational cohomology endowed with the Mukai pairing and the weight-zero Hodge structure 
\[
H^*(X,B,\C)^{-k,k} \coloneqq \bigoplus_{p - q = -k}  H(X,B)^{p,q}.
\]
The isometry $\Phi^H_{\sE}$ depends on the choice of the B-fields, but we will suppress it from the notation. 
  
To get an easy generalization of the theory of atomic objects to the twisted case, we make the following assumption.
\begin{Ass*}[$\dagger$]\label{assumptionTwisted}
    The Lie algebra morphism 
    \[
    \End(H^*(X,\Q)) \isomor \End(H^*(Y,\Q)), \ f \mapsto \Phi_{\sE}^H \circ f \circ (\Phi_{\sE}^H)^{-1}
    \]
    restrict to an isomorphism $\Phi_{\sE}^{\fg} : \fg(X) \isomor \fg(Y)$ between the LLV algebras.
\end{Ass*}
In the untwisted case, this is shown in \cite{taelman21} via an interpretation of the LLV algebra in terms of Hochschild cohomology.
A general statement for the twisted case is still unavailable in the literature.
In our case in \cref{sec:CohomologicalAction}, we will prove assumption $(\dagger)$ explicitly via a deformation argument. 

\subsection{Twisted extended Mukai lattice}
Let $X$ be a projective HK manifold of dimension $2n$, $\alpha \in \Br(X)$ a topologically trivial Brauer class, and $B \in H^2(X,\Q)$ a B-field for $\alpha$.
The isometry of $H^*(X,\Q)$ given by cup product with $\exp(-B)$ is the exponential of $- \cup(-B) \in \fg(X)$.
Since $\fg(X) \cong \mathfrak{so}(\wt{H}(X,\Q))$ by \cite{verbitsky96a,looijengaLunts97} and \cite[Theorem 2.7]{green22}, we get an isometry 
\[
\exp(-B) \in \mathrm{SO}(\wt{H}(X,\Q))
\]
of the rational extended Mukai lattice. 
We introduce a weight-two Hodge structure on $\wt{H}(X,\Q)$, defined by
\[
\wt{H}(X,B)^{p,q} \coloneqq \exp(-B) \left( \wt{H}(X,\C)^{p,q} \right).
\]
We write $\wt{H}(X,B,\Q)$ to denote the extended Mukai lattice equipped with this twisted Hodge structure.  

The following is an extension to the twisted case of \cite[Theorem 4.8 and 4.9]{taelman21}.
We ignore the relation to the Verbitsky component, since it is not necessary for our purposes, but it generalizes as well.  

\begin{Prop}\label{prop:TwistedEquivalences}
    Let $X_1$ and $X_2$ be deformation equivalent \HK manifolds of dimension $2n$, and assume that either $n$ or $b_2(X)$ is odd. 
    Let $\alpha_1 \in \Br(X_1)$ and $\alpha_2 \in \Br(X_2)$ be topologically trivial Brauer classes, with fixed B-fields $B_1$ and $B_2$.
    Let $\Phi : \Db(X_1,\alpha_1) \isomor \Db(X_2,\alpha_2)$ be a derived equivalence which satisfies assumption $(\dagger)$.
    Then there exists a Hodge isometry
    \[
    \Phi^{\wt{H}} : \wt{H}(X_1,B_1,\Q) \isomor \wt{H}(X_2,B_2,\Q),
    \]
    which is equivariant with respect to $\Phi^{\fg} : \fg(X_1) \isomor \fg(X_2)$. 
\end{Prop}

\begin{proof}
    The construction follows from the representation theoretic arguments in \cite[Section 4.1]{taelman21}, thanks to assumption $(\dagger)$.
    To prove that it preserves the Hodge structure, we argue in as \cite[Proposition 4.7]{taelman21}, using that $\Phi^H$ respects the Hodge structure. 
\end{proof}

\begin{Defi}
    Let $X$ be a HK manifold, $\alpha \in \Br(X)$ a Brauer class.
    An object $E \in \Db(X,\alpha)$ is atomic if there exists a B-field $B$ lifting $\alpha$ and a non-zero $\tilde{v}^B(E) \in \wt{H}(X,\Q)$ such that
    \begin{equation}\label{eq:DefTwistedAtomic}
    \Ann(v^B(E)) = \Ann(\tilde{v}^B(E))
    \end{equation}
    as sub-Lie algebras of $\fg(X)$. 
\end{Defi}

\begin{Rem}\label{rem:IndependenceOnBandAlgebraicity}
    \begin{enumerate}
        \item This definition is independent of the B-field. 
        Indeed, if $B'$ is another B-field for $\alpha$, we have 
        \[
        v^{B'}(E) = \exp(-(B'-B))v^B(E).
        \]
        Therefore, if $\tilde{v}^B(E)$ satisfies \eqref{eq:DefTwistedAtomic} for $v^B(E)$, then $\exp(-(B'-B))\tilde{v}^B(E)$ satisfies it for $v^{B'}(E)$.
        \item If $C$ denotes the Weil operator for the usual Hodge structure on $H^*(X,\Q)$, we have 
         \[
         C \circ \exp(B) \in \Ann(v^B(E)) 
         \]
        because of \eqref{eq:TwistedChern}.
        If $E$ is atomic, this implies that $C \circ \exp(B) \in \Ann(\tilde{v}^B(E)) $.
        In other words, $\tilde{v}^B(E) \in \wt{H}(Y,B,\Q)$ is of Hodge type for the twisted Hodge structure.
    \end{enumerate}
\end{Rem}

\begin{Prop}\label{prop:InvarianceOfTwistedAtomicity}
    In the setting of \cref{prop:TwistedEquivalences}, if $E \in \Db(X_1, \alpha_1)$ it atomic, then its image $\Phi(E) \in \Db(X_2,\alpha_2)$ is atomic. 
\end{Prop}

\begin{proof}
    Choose B-fields $B_1,B_2$ for $\alpha_1,\alpha_2$. 
    By definition there is an extended Mukai vector $\tilde{v}^{B_1}(E)$ for $E$, which satisfies 
    \[
    \Ann(v^{B_1}(E)) = \Ann(\tilde{v}^{B_1}(E)) \subset \fg(X_1).
    \]
    By assumption $(\dagger)$ and \cref{prop:TwistedEquivalences}, we have 
    \[
     \Ann(v^{B_2}(\Phi(E))) = \Phi^{\fg}( \Ann(v^{B_1}(E))) =  \Phi^{\fg}(\Ann(\tilde{v}^{B_1}(E))) =  \Ann(\Phi^{\wt{H}} (\tilde{v}^{B_1}(E))).
    \]
    Therefore, $\Phi(E)$ is atomic and its twisted Mukai vector (with respect to $B_2$) is $ \tilde{v}^{B_2}(\Phi(E)) = \Phi^{\wt{H}} (\tilde{v}^{B_1}(E))$.
\end{proof}
Ultimately for us, atomicity is just a way to produce projectively hyperholomorphic twisted vector bundles, which have symplectic moduli spaces, as observed in \cite{verbitsky99} and recalled in \cref{sec:SymplecticForm}.
That atomic torsion-free sheaves are modular has been proved independently by Markman \cite[Theorem 1.2]{markman21}, and Beckmann \cite[Proposition 1.5]{beckmann22}.
The following is an adaptation of Beckmann's argument to the twisted case.

\begin{Prop}\label{prop:TwistedAtomicHyperholomorphic}
    Let $X$ be a projective HK manifold, $\alpha \in \Br(X)$ a Brauer class and $E$ be an $\alpha$-twisted vector bundle on $X$.
    If $E$ is atomic then $\cEnd(E)$ is modular. 
    If furthermore it is also $\mu_{\omega}$-polystable, then it is $\omega$-projectively hyperholomorphic.
\end{Prop}

\begin{proof}
    The argument of \cite[Lemma 5.4]{beckmann22} works also in the twisted case, so we quickly go over it. 
    Let $B \in H^2(X,\Q)$ a B-field for $\alpha$, and consider the class 
    \begin{equation*}
        k(E) \coloneqq \exp(-c^B_1(E)/r(E))\ch^B(E).
    \end{equation*}
    It does not depend on the B-field, and it satisfies 
    \begin{equation*}
        \Delta(E) = -2r(E)k_2(E) \in H^4(X,\Q).
    \end{equation*}
    As in \cite[Lemma 5.4]{beckmann22} we prove that $ k(E)$ stays of Hodge type for all K\"ahler deformations of $X$.
    The key observation is that atomicity implies
    \begin{equation}\label{eq:Atomicityfork}
    \Ann\left( \exp(-c^B_1(E)/r(E))v^B(E)\right) =  \Ann\left(\exp(-c^B_1(E)/r(E))\tilde{v}^B(E)\right).
    \end{equation}
    The proof of \cite[Proposition 3.8]{beckmann22} and \cite[Theorem 6.13(3)]{markman21} follows directly from the definition of atomicity, so it works in the twisted setting as well. 
    We write
    \[
    \tilde{v}^B(E) = r(E) \alpha + c^B_1(E) + s\beta.
    \]
    This implies that
    \[
    \exp(-c^B_1(E)/r(E))\tilde{v}^B(E) = r(E) \alpha + s'\beta,
    \]
    for some $s' \in \Q$. 
    Therefore, $\exp(-c^B_1(E)/r(E))\tilde{v}^B(E)$ is of Hodge type with respect to every complex structure.
    The equality \eqref{eq:Atomicityfork} implies that the same is true for $k(E)$. 
    Since $\Delta(\cEnd(E))$ is a multiple of $c_2(\cEnd(E))$ we conclude that it is modular. 
    If $E$ is $\mu_{\omega}$-polystable (refer to \cref{sec:SlopeStability} for the definition of slope stability for twisted sheaves), then the endomorphism bundle $\cEnd(E)$ is $\mu_{\omega}$-polystable as well.
    Therefore $\cEnd(E)$ is hyperholomorphic by \cite[Theorem 3.19]{verbitsky99}.
\end{proof}

\section{Twisted Poincaré sheaf}\label{sec:TwistedPoincaré}
Let $(S,H)$ be a polarized K3 surface, and assume that every curve in $|H| \cong \P^n$ is integral. 
This holds for a general $(S,H)$. 
In the rest of this section, $\pi: X \coloneqq \overline{\Pic^0}(\cC/|H|) \to \P^n$ denotes the relative compactified Jacobian of the universal curve $\cC \to |H|$.
Our goal is to extend the autoequivalence of \cite[Theorem C]{arinkin13} to an equivalence $\Db(X^t) \isomor \Db(X,\theta_t)$, where $X^t$ is any Tate--Shafarevich twist of $X$.

\subsection{Review: Tate--Shafarevich group}
The Tate--Shafarevich group of the Lagrangian fibration $X \to \P^n$ was first introduced in \cite[Section 7]{markman14} as the natural generalization of the Tate--Shafarevich group of an elliptic fibration.
It was further studied in \cite{abasheva21}, and with a slightly different point of view in \cite{huybrechts_mattei24}.
It is defined as
\[
\Sha \coloneqq H^1(\P^n, \underline{\Aut}^0_{X/\P^n}),
\]
where $\underline{\Aut}^0_{X/\P^n}$ is the image of the exponential map 
\[
\pi_* T_{X/\P^n} \to \underline{\Aut}_{X/\P^n}.
\]
The group $H^1(\P^n,R^1\pi_*T_{X/\P^n}) \cong \C$ is denoted by $\wt{\Sha}$, and its image via the exponential map by $\Sha^0 \subset \Sha$.
An element $t \in H^1(\P^n,\underline{\Aut}^0_{X/\P^n})$ can  be represented by a \v{C}ech cocycle of relative automorphisms
\[
t_{ij} \in \Aut^0(\pi^{-1}(U_{ij})/U_{ij}),
\]
with respect to an open cover $\P^n = \bigcup_{i} U_{i}$.
If $\sigma: U_{ij} \to X$ denotes the zero section, any such automorphism is given by translation by the section $t_{ij} \circ \sigma$. 
By the cocycle condition, the $t_{ij}$ can be interpreted as transition functions and used to re-glue the manifold $X$ to a new one $X^{t}$, equipped with a fibration $\pi^{t} : X^{t} \to  \P^n$. 
In general, the $X^t$ satisfies all the properties for being \HK besides being K\"ahler.
In our setting, the manifold $X^t$ is K\"ahler by \cite[Proposition 7.7]{markman14} if $t \in \Sha^0$, and projective if $t \in \Sha^0$ is torsion by \cite[Theorem 5.19]{abasheva21}.

From now on, we will often abuse the notation and denote by the same letter $t$ both an element in $\wt{\Sha}$ and its image under the exponential map in $\Sha$.
All the twisted fibrations coming from elements in $\Sha^0$ can be put in a family over $\wt{\Sha}$.

\begin{Prop}[{{\cite[Section 7.2]{markman14}, \cite[Proposition 3.3]{abasheva21}}}]\label{prop:RelativeTSTwist}
    There exists a holomorphic family $\cX \to \wt{\Sha}$,
    such that the fiber over $t \in  \wt{\Sha}$ is $X^t$. 
    Moreover, there exists a holomorphic fibration
    \begin{equation}\label{eq:GlobalFibration}
    \cX \to \P^n \times \wt{\Sha},
    \end{equation}
    which restricts to $\pi_t : X^t \to \P^n$ for every $t \in \wt{\Sha}$.
\end{Prop}

\subsection{Construction of the sheaf}
Let $\pi: X \to |H|$ be a Beauville--Mukai system as above.
In \cite{arinkin13} it is shown that there exists a relative Poincaré sheaf $\sP^0$ on $X \times_{\P^n} X$, which induces a derived equivalence 
\[
\Phi_{\sP^0}: \Db(X) \isomor \Db(X).
\]
In \cite{addington16} this was generalized to the case of relative compactified Picard scheme of any degree, and now we generalize it to any Tate--Shafarevich twist. 
The following is an application of the Theorem of the square to a $C^{\infty}$-section, which we will use multiple times in the rest of the section. 

\begin{Lem}\label{lem:SquareTheorem}
    Let $U \subset |H|$ be an (analytic) open subset, and denote by $\pi : X_U \to U$ the restriction of the Lagrangian fibration.   
    Let $\tau: U \to X_U$ be a $C^{\infty}$-section, and let $\rho: X_U \isomor X_U$ be the corresponding relative translation. 
    Define the topological line bundle  $L \coloneqq (\tau \times \id )^*(\sP^0)$ on $X$.
    Then, there is an isomorphism
    \[
        (\rho \times \id)^*(\sP^0) \cong \sP^0 \otimes p_2^* L
    \]
    of $C^{\infty}$-modules on $X_U \times X_U$. 
\end{Lem}

\begin{proof}
    By \cite[Lemma 6.5]{arinkin13}, given the diagram 
    \[\begin{tikzcd}
	{X_U^{\circ} \times_U X_U} & {X_U^{\circ} \times_U X_U \times_U X_U} & {X_U \times_U X_U} \\
	& {X_U \times_U X_U}
	\arrow["{p_{13}}"', from=1-2, to=1-1]
	\arrow["{p_{23}}", from=1-2, to=1-3]
	\arrow["{\mu \times \id}", from=1-2, to=2-2]
    \end{tikzcd}\]
    we have
    \[
        (\mu \times \id)^*(\sP^0) \cong p_{13}^*(\sP^0|_{X_U^{\circ} \times_U X_U}) \otimes p_{23}^*\sP^0.
    \]
    Here $X_U^{\circ}$ denotes the smooth locus of the Lagrangian fibration $\pi$, and $\mu : X_U^{\circ} \times X_U \to X_U$ its action by translation.  
    Pulling back the isomorphism along the $C^{\infty}$-embedding  
    \[
        \tau \times \id \times \id: X_U \times_U X_U \hookrightarrow X_U^{\circ} \times_U X_U \times_U X_U
    \]
    gives the result. 
\end{proof}

\begin{Thm}\label{thm:TwistedPoincaré}
    For any $t \in \Sha^0$, there exist $\theta_t \in H^2(X,\sO^*_X)$ and a twisted Cohen--Macaulay sheaf 
    \[
    \sP^t \in \Coh\left(X^t \times X,p_2^*\theta_t\right).
    \]
    If $X^t$ is algebraic it induces a derived equivalence 
    \[
    \Phi_{ \sP^t }: \Db(X^t) \isomor \Db(X,\theta_t).
    \]
    Moreover, the mapping $t \mapsto \theta_t$ gives a group homomorphism $\Sha^0 \to H^2(X,\sO^*_X)$.
\end{Thm}

\begin{proof}
Choose an analytic cover $\{ U_i \}$ of $|H| \cong \P^n$ and a cocycle $\{ t_{ij} \}$ representing $t \in \Sha^0$.
Over each open $U_i$ there is a biholomorphism $\rho_i : X^t_{U_i} \cong X_{U_i}$,
and the 
\[
t_{ij} = \rho_j|_{U_{ji}} \circ \rho^{-1}_i|_{U_{ij}}: X_{U_{ij}} \cong X_{U_{ji}} 
\]
are the gluing functions for $X^t$.
Let $\sP_i \coloneqq \sP^0|_{U_i} \in \Coh(X_{U_{i}} \times_{{U_{i}}} X_{U_{i}})$. 
By \cref{lem:SquareTheorem} there is a unique (in this case holomorphic) line bundle $L_{ij}$ on $X_{U_{ij}}$ such that
\begin{equation}\label{eq:TwistedPoincaré}
 (t_{ij} \times \id)^*(\sP_j|_{U_{ij}}) \cong \sP_i|_{U_{ij}} \otimes p_2^*(L_{ij}).
\end{equation}
Using the definition of $L_{ij}$ in \cref{lem:SquareTheorem} one checks that the cocycle condition for the $t_{ij}$ implies that the collection $\{L_{ij}\}$ represents a class $\theta_t \in H^2(X,\sO^*_X)$, as explained in \cref{sec:ReviewBrauer}.
By definition, \eqref{eq:TwistedPoincaré} means that the pullbacks $(\rho_i \times \id)^*\sP_i \in \Coh(X^t_{U_{i}} \times_{U_{i}} X_{U_{i}})$ glue to a $p_2^*\theta_t$-twisted sheaf on $X^t \times X$. 

If we take $t,s \in \Sha(X/\P^n)$, represent them by cocycles $\{ t_{ij}\} $ and $\{ s_{ij}\}$ with respect to the same cover of $\P^n$, their product is represented by $ t_{ij} \circ s_{ij}$.
If $\{L_{ij}\}$ and $\{ M_{ij}\}$ are the collection of line bundles associated to $\{ t_{ij}\}$ and $\{ s_{ij} \}$, it follows from \eqref{eq:TwistedPoincaré} that $\{ L_{ij} \otimes M_{ij} \}$ is the one associated to $\{ t_{ij} \circ s_{ij} \}$. 
Therefore, the association $t \mapsto \theta_t$ is compatible with the group structures. 

It remains to see that $\Phi_{\sP^t}$ is a derived equivalence, for which we exhibit an inverse.
Let $\sQ^t \coloneqq {\sP^t}^\vee[2n] \in \Db(X^t \times X,-p_2^*\theta)$, where ${\sP^t}^{\vee}$ denotes the derived dual. 
Since $\sP^t$ is a Cohen--Macaulay sheaf supported in codimension $n$, $\sQ^t$ is simply a sheaf shifted (to the left) by $n$.
The convolution $\sQ^t \circ \sP^t$ is an untwisted sheaf, and the natural map from the identity (coming from the adjunction) is an isomorphism over the opens $U_i$ by \cite[Theorem C]{arinkin13}.
Thus it is globally an isomorphism. 
This shows that $\sP^t$ is fully faithful, hence an equivalence, since $X$ has trivial canonical bundle. 
\end{proof}

\begin{Rem}\label{rem:TwistingInTheProjectiveCase}
    By \cite[Theorem 5.19]{abasheva21} the manifold $X^t$ is projective if and only if $t \in \Sha^0$ is torsion. 
    In this case, the obstruction $\theta_t$ is torsion as well, therefore it lies in the Brauer group $\Br(X) = H^2(X,\sO^*_X)_{\mathrm{tor}}$. 
    If $X^{\circ} = \Pic^0(X^t/|H|) \subset X$ denotes the smooth locus of the Lagrangian fibration, then $\theta_t|_{X^{\circ}} \in \Br(X^{\circ})$ is the obstruction to the existence of a relative universal sheaf for the Picard scheme. 
    Since $\Br(X) \cong  \Br(X^{\circ})$, this characterizes $\theta_t$ in this case. 
\end{Rem}

\begin{Rem}\label{rem:Untwisting}
    As usual, when dealing with the universal sheaf for the Picard scheme, one has to be careful about untwisting. 
    Notice that 
    \[
    \Supp \sP^t = X^t \times_{\P^n} X.
    \]
    Since $\sP^t$ has rank $1$ on its support, we deduce that $p_2^*(\theta_t)|_{X^t \times_{\P^n} X}$ is trivial in the Brauer group. 
    This does not mean that $\sP^t$ is an untwisted sheaf.
    It means that, after choosing a fine enough cover, we can represent the class $p_2^*\theta_t|_{X^t \times_{\P^n} X}$ by a boundary, which in turn allows to reglue $\sP^t$ to an honest sheaf.
    This cover will not respect the Lagrangian fibration, and therefore the untwisted sheaf will not be universal, hence not induce a derived equivalence.
    The same issue was discussed in \cite[Footnote 4]{addington16}.
\end{Rem}

\begin{Rem}\label{rem:ComputationOfTheObstruction}
    The obstruction $\theta$ can be described explicitly using the relation of the group $\Sha^0$ with the group $\Sha(S,H)$ introduced in \cite{huybrechts_mattei24}. 
    Indeed, by  \cite{huybrechts_mattei24}, the twist $X^t$ is a moduli space of twisted sheaves on $S$ with respect to some Brauer class $\alpha \in \Br(S)$. 
    There is a natural isomorphism\footnote{Because $X$ is a fine moduli space on $S$.} $\Br(X) \isomor \Br(S)$ induced by the Mukai homomorphism, and the obstruction $\theta$ is mapped to $\alpha$ under it. 
    This follows from the discussion in \cite[Section 2.3]{bottiniHuybrechts25}.
\end{Rem}

\subsubsection{The relative case}
The argument for  \cref{thm:TwistedPoincaré} can similarly be applied in the relative case.
More precisely, let $\cX \to \wt{\Sha}$ be the family of \cref{prop:RelativeTSTwist}, and denote by $\wt{X} \to \wt{\Sha}$ the trivial family $ \wt{X} \coloneqq X \times \wt{\Sha} $. 
By the construction \cite[Section 7.2]{markman14}, $\cX$ is obtained by regluing $\wt{X}$ along relative automorphisms.
Regluing along these automorphisms, as in the proof of \cref{thm:TwistedPoincaré}, we obtain the following.

\begin{Prop}\label{prop:TwistedPoincaréHochschildCohomology}
    There exists a class $\tilde{\theta} \in H^2(\wt{X},\sO^*_{\wt{X}})$ and a twisted sheaf
    \begin{equation}\label{eq:GlobalRelativePoincaré}
        \wt{\sP} \in \Coh(\cX \times_{\wt{\Sha}} \wt{X},p_2^*(\tilde{\theta})),
    \end{equation}
    such that for every $t \in \wt{\Sha}$ we have $\wt{\theta}|_t = \theta_t$ and $\wt{\sP}|_t = \sP^t$.
\end{Prop}

\begin{Rem}
This result is consistent with the deformation theory developed in \cite{toda09}.
It is saying that the autoequivalence of $\Db(X)$ given by the usual Poincaré sheaf deforms, provided we deform the source along the commutative deformation given by the Tate--Shafarevich line, and the target along the gerby deformation given by $H^2(X,\sO_X) \subset \HT^2(X)$.
Indeed, by \cite[Theorem 1.2]{toda09}, at least for first-order deformations, this is equivalent to showing that the induced isomorphism
\[
\Phi^{\HT}_{\sP^0} : \HT^2(X) \cong \HT^2(X)
\]
sends $f \in H^1(X,\Omega^1_X)$ (interpreted as an element of $H^1(X,T_X)$) to $H^2(X,\sO_X)$.
One can check this using the action of $\HT^2(X)$ on $\wt{H}(X,\Q)$, and the fact that $\Phi^{\wt{H}}(f) = \beta$, see \cite[Proposition 10.4]{beckmann22}. 
\end{Rem}

\subsection{Cohomological action}\label{sec:CohomologicalAction}
Our next goal is to check that the equivalence 
\[
   \Phi_{ \sP^t }: \Db(X^t) \isomor \Db(X,\theta_t)
\]
satisfies assumption $(\dagger)$.
Consider the deformation over $\wt{\Sha}$ of \cref{prop:TwistedPoincaréHochschildCohomology}. 
Consider the parallel transport operators
\[
f^H_t : H^*(X^t,\Q) \isomor H^*(X,\Q), \ f^{\fg}_t : \fg(X^t) \isomor \fg(X),
\]
from the fiber $X^t$ over $t \in \wt{\Sha}$ to the central fiber $X$.
Clearly $f^{\fg}_t$ is a Lie algebra isomorphism, and $f^H_t$ is equivariant with respect to $ f^{\fg}_t$, namely
\begin{equation}\label{eq:ParallelTransportIsEquivariant}
    f^{\fg}_t(x).f^H_t(\omega) = f^H_t(x.\omega), \ \textrm{ for every } x \in \fg(X^t), \ \omega \in H^*(X^t,\Q).
\end{equation}
Define
\[
\Phi^{\fg}_{\sP^t} \coloneqq \Phi^{\fg}_{\sP^0} \circ f^{\fg}_t : \fg(X^t) \isomor \fg(X),
\]
where $\sP^0 \in \Coh(X \times X)$ is the untwisted Poincaré sheaf. 

The next result relates the cohomological action of $\Phi_{\sP^t}$ with the one of $\Phi_{\sP^0}$, and  formalizes the intuition that $\sP^t$ is a flat deformation of $\sP^0$. 
The difficulty lies in the fact that, along this flat deformation, the class $\theta_t \in H^2(X^t,\sO^*_{X_t})$ is not torsion for infinitely many $t$.
In particular, within our framework, the Mukai vector of $\sP^t$ is not well defined for those $t$.
Fortunately over the torsion points we have the following result.

\begin{Prop}\label{prop:CohomologicalAction}
    Let $t \in \Sha^0$ be a torsion element.
    Then there exists a B-field $B_t \in H^2(X,\Q)$ lifting $\theta_t$, such that
    \[
     \Phi^H_{\sP^t} = \Phi^H_{\sP^0} \circ f^H_t : H^*(X^t,\Q) \isomor H^*(X,B_t,\Q).
    \]
    In particular, the Hodge isometry $ \Phi^H_{\sP^t}$ is equivariant with respect to $\Phi^{\fg}_{\sP^t}$.
\end{Prop}

\begin{proof}
    It suffices to find a B-field $B_t$ such that
    \begin{equation}
       f^H_t( \ch^{B_t}(\sP^t)) = \ch(\sP^0) \in H^*(X,\Q).
    \end{equation}
    By \cref{prop:RelativeTSTwist} and the discussion in \cite[Section 7.2]{markman14}, we can think of $X$ and $X^t$ as the same underlying $C^{\infty}$ manifold $X$, equipped with two different complex structures $J$ and $J^t$.
    The map $X \to \P^n$ is holomorphic with respect to both, and the two complex structures agree on each fiber. 
    The natural identification $H^*(X^t,\Q) = H^*(X,\Q)$ is parallel transport along the family of \cref{prop:RelativeTSTwist}.
    With this point of view, the local biholomorphisms $ \rho_i : X^t_{U_i} \cong X_{U_i}$ can be interpreted as diffeomorphisms $X_{U_i} \cong X_{U_i}$, given by translation with respect to a $C^{\infty}$-section. 
    
    Adopting the notation of \cref{thm:TwistedPoincaré} $\sP_i = \sP^0|_{U_i}$, we set $\sP^t_i \coloneqq \sP^{t}|_{U_i} = (\rho_i \times \id)^*\sP_i$. 
    By \cref{lem:SquareTheorem}, there is a unique topological line bundle $L_{i}$ on $X_{U_i}$ such that
     \begin{equation}\label{eq:DefOfLineBundles}
        \sP^t_i \cong \sP_i \otimes p_2^*L_{i},
    \end{equation}
    as $C^{\infty}$-modules.
    By construction, the composition 
    \[
     \sP^t_i|_{U_{ij}} \isomor \sP^t_j|_{U_{ij}} \otimes p_2^*(L^{-1}_j \circ L_i)
    \]
    is the isomorphism \eqref{eq:TwistedPoincaré}. 
    In particular there are canonical isomorphisms
    \[
    L^{-1}_{j} \otimes L_i \cong L_{ij},
    \]
    where $\{ L_{ij} \}$ represents $\theta_t$.  
    These isomorphisms show that the collection $\{U_i,L_i\}$ glues as a $\theta_t$-twisted topological line bundle $L$ on $X$.
    And the isomorphisms \eqref{eq:DefOfLineBundles} glue to an isomorphism
    \begin{equation*}
         \sP^t \otimes p_2^*(L^{-1}) \cong \sP^0
    \end{equation*}
    of $C^{\infty}$-modules on $X \times X$.
    
    We claim that we can find a B-field $B_t \in H^2(X,\Q)$ such that its associated topological line bundle (cf. equation \eqref{eq:TrivializationBField}) is $L$. 
    Indeed, take any B-field $B'$ lifting $\theta_t$ at the level of cocycles.
    Writing $B'_{ijk} = a'_{ij} + a'_{jk} - a'_{ik}$ we obtain an $(\theta_{t})_{ijk}$-twisted topological line bundle $M$ with transition functions $\{ \exp(a'_{ij}) \}$. 
    Then $L \otimes M^{-1}$ is untwisted on the nose, and
    \[
        B_t \coloneqq B' + c_1(L \otimes M^{-1}) \in H^2(X,\Q)
    \]
    is the desired B-field.
    With this choice, by \cref{rem:TwistedChernIndependentFromRes} we deduce 
    \[
    \ch^{B_t}(\sP^t ) = \ch(\sP^t \otimes L^{-1}) = \ch(\sP^0) \in H^*(X,\Q),
    \]
    where we are identifying $H^*(X^t,\Q) = H^*(X,\Q)$ via parallel transport.
    If we expand this out, we can write 
    \[
    f^H_t(v^{B_t}(\sP^t)) = v(\sP^0). 
    \]
    Then, a computation shows that
    \[
    \Phi^H_{\sP^t} = \Phi^H_{\sP^0} \circ f^H_t,
    \]
    which is equivariant with respect to $\Phi^{\fg}_{\sP^t}$.
\end{proof}

\begin{Cor}\label{cor:PoincaréSendsAtomicToAtomic}
    The equivalence $\Phi_{\sP^t} : \Db(X^t) \isomor \Db(X,\theta_t)$ sends atomic objects to atomic twisted objects. 
\end{Cor}

\begin{proof}
    Since $\Phi^H_{\sP^t}$ is equivariant with respect to $\Phi^{\fg}_{\sP^t}$, its action on $\End(H^*(X,\Q))$ given by conjugation respects the LLV algebra. 
    That is, assumption $(\dagger)$ holds for $\Phi^H_{\sP^t}$. 
    Therefore, we conclude by \cref{prop:InvarianceOfTwistedAtomicity}.
\end{proof}

\section{Compactified Picard scheme}\label{sec:compactifiedPic}
Let $Y \subset \P^5$ be a general cubic fourfold, and let $|\sO_Y(1)| \cong \P^5$ the space of hyperplane sections.
Let
\[
\cV \coloneqq \{ (x,H) \mid x \in Y_H, \ H \in \P^5 \} \subset Y \times \P^5
\]
be the universal hyperplane section.
The second projection $\cV \to \P^5$ is flat, and the fibers are the hyperplane sections of $Y$. 
The first projection $\cV \to Y$ is a $\P^4$-bundle, hence $\cV$ is regular of dimension $8$. 
Analogously, we consider the relative variety of lines 
\[
\cF \coloneqq \{ (l,H) \mid l \in F(Y_H), \  H \in \P^5 \}  \subset F(Y) \times  \P^5.
\]
The second projection $\cF \to \P^5$ is flat, and the fiber over $H$ is the surface of lines in $Y_H$. 
Every fiber is integral with lci singularities by \cite[Lemma 2.5]{lsv17}.
The first projection $\cF \to F(Y)$ is a $\P^3$-bundle, thus $\cF$ is regular of dimension $7$. 

Following Altman--Kleiman \cite{altman80,altman79} we consider the scheme $\Pic^{=}(\cF/\P^5)$ parameterizing torsion-free sheaves of rank 1 on the fibers of $\cF \to \P^5$. 
We denote by $\Pic(\cF/\P^5)^-$ the open subscheme parameterizing Cohen--Macaulay (CM) sheaves, and, as usual, by $\Pic(\cF/\P^5)$ the open subscheme parameterizing locally free sheaves.
Let 
\[
\Pic^0(\cF /\P^5)  \subset \Pic^{-0}(\cF/\P^5) \subset \Pic^{=0}(\cF/\P^5)
\]
be the connected components containing the trivial line bundle.
As shown in \cite{altman79}, the scheme $\Pic^{=0}(\cF/\P^5)$ is projective over $\P^5$, and $\Pic^{-0}(\cF/\P^5)$ is quasi-projective.
Denote by $\overline{\Pic^0}(\cF /\P^5)$ the closure of the locally free locus in $ \Pic^{=0}(\cF/\P^5)$, taken with its {\textit{reduced}} scheme structure. 
Since the Picard scheme of a smooth cubic threefold is $5$-dimensional, we deduce that $\overline{\Pic^0}(\cF /\P^5)$ has dimension $10$.
The goal of this section is to show that $\overline{\Pic^0}(\cF /\P^5)$ is smooth and contained in  $\Pic^{-0}(\cF/\P^5)$.

\subsection{Lines on cubic threefolds}
Now, let $V = Y_H \subset \P^4$ be a hyperplane section of a general cubic $Y$. 
It is an integral cubic threefold with only ADE singularities by \cite[Corollary 3.7]{lsv17}.
If $l \subset V$ is a line, the projection to a complementary plane defines a rational map 
\[
 V \dashrightarrow \P^2. 
\]
It is resolved by blowing up the line, and we obtain a conic bundle 
\begin{equation}
   \pi_l : V_l \to \P^2.
\end{equation}

\begin{Lem}
    For any line $l \subset V$ the discriminant $D_l \subset \P^2$ is a plane quintic. 
\end{Lem}

\begin{proof}
    If $D_l$ is a curve, then the fact that it is a quintic is classical. 
    So assume by contradiction that $D_l = \P^2$, so that every fiber of $\pi_l$ is singular. 
    By Stein factorization of $\pi_l$ we obtain a surface $S_l$, finite over $\P^2$, parametrizing lines in $V$ which all meet $l$. 
    By \cite[Lemma 2.5]{lsv17} this surface must coincide with $F(V)$, so every line in $V$ meets $l$. 
    Consider the map 
    \[
        F(V) \setminus \{l\} \to l, l' \mapsto l \cap l'.
    \]
    If it is constant, then $V$ is a cone, which does not happen if $Y$ is general.
    If it is surjective, for every point in $l$ there is a curve in $V$ of lines through that point. 
    This does not happen if $Y$ is general, by \cite[Claim 2.15]{lsv17}.
\end{proof}

Define the curve 
\[
C_l \coloneqq F(\pi^{-1}(D_l)/D_l)
\]
as the relative scheme of lines in the fibers of $\pi_l^{-1}(D_l)$.
Since $\pi_l$ is a conic bundle, the morphism $C_l \to D_l$ is flat of degree $2$.
Moreover, it is étale if all the fibers of $\pi_l$ are reduced.

At least the line $l$ is general (we will later show that this holds for all lines), \cite[Lemma 1.26]{huybrechts23} says that $C_l$ is the closure of the locus of lines meeting $l$
\begin{equation}
    C_l = \overline{\{ l' \in F(V) \mid l \neq l' \text{ and } l' \cap l \neq \emptyset \}}.
\end{equation}
The involution $\iota : C_l \to C_l$ associating to a line $l'$ its residual line 
\begin{equation}\label{eq:InvolutionResidualLine}
    l' \mapsto l'', \text{ where } \mathrm{Span}(l,l') \cap V = l \cup l' \cup l''
\end{equation}
gives an identification $C_l/\langle \iota \rangle \cong D_l$.

Following \cite[Definition 2.9]{lsv17}, we define a line $l \subset V$ to be very good if every fiber of $\pi_l$ is reduced and $C_l$ is integral. 
By \cite[Proposition 2.10]{lsv17}, every hyperplane section of a general cubic fourfold contains a very good line.

\begin{Rem}\label{rem:VeryGoodImpliesSmooth}
    If $p \in V$ is a singular point, the locus of lines through $p$
    \[
    C_p = \{ l \in F(V) \mid p \in l \}
    \]
    is a curve, and the singular locus of $F(V)$ is the union of these curves (see \cite[Lemma 1.5(ii)]{altman77}).
    It follows that a very good line $l$ is a smooth point of $F(V)$, as otherwise, $C_l$ would be reducible.  
\end{Rem}

\subsection{Relative setting}

Let $\cH \subset \P^5 \times \P^5 \xrightarrow{p_2} \P^5$ be the universal hyperplane section and let $\L_{\cV} \subset \cV \times_{\P^5} \cF$ be the universal line.
Pick a family of $2$-planes
\[
\Pi \subset \cH \times_{\P^5} \cF 
\]
disjoint from $\L_{\cV}$. 
Blowing-up the universal line we get the projection
\[
\pi : \wt{\cV \times_{\P^5} \cF} \to \Pi.
\]
By the above discussion, the discriminant locus $\cD \subset \Pi$ is a hypersurface, which restricts to the plane quintic $D_l \subset \P^2$ for every $(l,H) \in \cF$.
Define 
\[
\cC \coloneqq F(\pi^{-1}(\cD)/\cD) \to \cD
\]
to be the relative scheme of lines in the fibers of $ \pi^{-1}(\cD) \to \cD$.

\begin{Lem}\label{lem:RelativeGrassmannianEmbedsInFano}
    There is a closed embedding of schemes over $\cF$
    \[
    \cC \hookrightarrow \cF \times_{\P^5} \cF.
    \]
    The set-theoretical image is the closure of the set of intersecting lines. 
    In particular, $\cC$ is irreducible. 
\end{Lem}

\begin{proof}
    The functor of points of $\cC$, as a scheme over $\cF$, is the subfunctor $\mathrm{Hilb}_{\pi^{-1}(\cD)/\cF}$ parameterizing lines in the fibers of $\pi^{-1}(\cD) \to \cD$.
    We get a chain of inclusions of functors
    \begin{equation}\label{eq:ChainOfInclusionsHilbert}
    \cC \subset \mathrm{Hilb}_{\pi^{-1}(\cD)/\cF} \subset \mathrm{Hilb}_{\wt{\cV \times_{\P^5} \cF}/\cF}.
    \end{equation}
    
    The blow-down
    \[
    \wt{\cV \times_{\P^5} \cF} \to \cV \times_{\P^5} \cF
    \]
    does not contract any fiber of $\pi$, as this is true already if we blow-up the whole $\P^5$. 
    Therefore, the blow-down composed with the composition \eqref{eq:ChainOfInclusionsHilbert} gives a monomorphism of schemes over $\cF$
    \[
    \cC \hookrightarrow \cF \times_{\P^5} \cF.
    \]
    which is a closed embedding by the properness of $\cC \to \cF$.  
    
    A fiber of $\pi_l$ contains $l' \neq l$ if and only if $l'$ meets $l$, and contains $l$ if and only if it lies on a plane tangent to $V$.  
    It follows that set-theoretically the image is 
    \[
    \overline{\{ (l,l',H) \mid l \neq l',\  \mathrm{ and} \  l \cap l' \neq \emptyset \}},
    \]
    which is irreducible by \cite[Lemma 2.13]{lsv17}.
\end{proof}


\begin{Rem}
    In particular, it follows that $\cD$ is irreducible.
    Since $\cD$ is also Cohen--Macaulay and generically reduced, it is in fact integral. 
\end{Rem}

\begin{Lem}\label{lem:FlatnessofC}
    The natural map $\cC \to \cD$ is flat, and $\cC$ is Cohen--Macaulay. 
    In particular, $\cC$ is integral, and the morphism $\cC \to \cF$ is flat and projective. 
\end{Lem}

\begin{proof}
    Every fiber of 
    \[
    \pi^{-1}(\cD) \to \cD
    \]
    is a singular conic.
    Since the Hilbert scheme is compatible with base change, the fiber $\cC_d = F(\pi^{-1}(d))$ over a point $d \in \cD$ is the scheme of lines in the singular conic $\pi^{-1}(d)$. 
    It consists of two reduced points if the conic is reduced, and a double point if the conic is a double line. 
    So, the map $\cC \to \cD$ is finite with fibers of constant length $2$.
    Thus it is flat, as $\cD$ is integral. 
    
    Since any $0$-dimensional length $2$ scheme is Cohen--Macaulay, the morphism $\cC \to \cD$ is Cohen--Macaulay by \cite[Lemma 37.22.2]{stacks}.
    Therefore, $\cC$ is Cohen--Macaulay by \cite[Lemma 37.22.4(1)]{stacks} and the fact that $\cD$ is Cohen--Macaulay, being a hypersurface in a regular scheme.
    We conclude because $\cD \to \cF$ is flat and projective, and a Cohen--Macaulay scheme does not have embedded points, so $\cC$ is integral and it agrees with the locus of intersecting lines by \cref{lem:RelativeGrassmannianEmbedsInFano}.
\end{proof}

\subsection{Abel--Jacobi and restriction}
We denote by $\cF_0 \subset \cF$ the locus of very good lines. 
The projection $f: \cF_0 \to \P^5$ is smooth and surjective. 
Therefore, the base change $\cF \times_{\P^5} \cF_0$ is smooth. 
Consider the ``Abel--Jacobi'' morphism
\begin{equation}\label{eq:Albanese}
    \alpha: \cF \times_{\P^5} \cF_0 \to \Pic^{-0}(\cF \times_{\P^5} \cF_0 / \cF_0), \quad (l,l',H) \mapsto I_{C_l/F(Y_H)} \otimes \sO_{F(Y_H)}(C_{l'}).
\end{equation}
It is well defined, because for every $(l,H) \in \cF$ the curve $C_l \coloneqq \cC_{(l,H)} \subset F(Y_H)$ is a Cohen--Macaulay scheme by \cref{lem:FlatnessofC}. 
Since $F(Y_H)$ has Gorenstein singularities, the ideal sheaf $I_{C_l/F(Y_H)}$ is Cohen--Macaulay on $F(Y_H)$. 

Let $\cC_0 \xhookrightarrow{i} \cF \times_{\P^5} \cF_0$ be the intersection of $\cC$ with the open locus $\cF \times_{\P^5} \cF_0$. 
By definition of very good lines, $\cC_0 \to \cF_0$ has integral fibers with planar singularities.
It is a Cartier divisor, because the ambient space is smooth. 
Since $\cC_0 \to \cF_0$ is flat by \cref{lem:FlatnessofC}, it is a Cartier divisor fiberwise (a so-called relative effective Cartier divisor, see \cite[Section 31.18]{stacks}).
Therefore, by \cref{prop:PullbackOfCM}, we have a well-defined restriction map
\begin{equation}\label{eq:Restriction}
    i^*: \Pic^{-0}(\cF \times_{\P^5} \cF_0/\cF_0) \to \overline{\Pic^0}(\cC_0 / \cF_0), \quad G \mapsto G|_{C_l}.
\end{equation}
Recall that for a family of integral curves with planar singularities, the relative compactified Picard parametrizes Cohen--Macaulay sheaves. 

\begin{Lem}
    The composition 
    \[
    \gamma \coloneqq i^* \circ \alpha : \cF \times_{\P^5} \cF_0 \to \overline{\Pic^0}(\cC_0 / \cF_0)
    \]
    is a local complete intersection morphism. 
    In particular, the pullback
    \begin{equation}\label{eq:Pullback}
    \gamma^* : \overline{\Pic^0}( \overline{\Pic^0}(\cC_0 / \cF_0)) \to \Pic^{-0}(\cF \times_{\P^5} \cF_0/\cF_0)
    \end{equation}
    is well defined. 
\end{Lem}

\begin{proof}
    See \cref{sec:functorialityPicard} for the definition of a lci morphism. 
    Consider the diagram 
\[\begin{tikzcd}
	{\cF \times_{\P^5} \cF_0} & {\overline{\Pic^0}(\cC_0 / \cF_0)} \\
	{\cF_0}
	\arrow[from=1-1, to=2-1]
	\arrow[from=1-2, to=2-1]
	\arrow["\gamma", from=1-1, to=1-2]
\end{tikzcd}\]
    Since ${\overline{\Pic^0}(\cC_0 / \cF_0)}$ is smooth by \cite[Corollary B.2]{fantechi1999} and \cite[Corollary 3.9]{lsv17}, by \cite[Lemma 37.62.13]{stacks} we get that $\gamma$ is a local complete intersection morphism. 
    We conclude by \cref{lem:LciOnFibers} and \cref{prop:PullbackOfCM} that the pullback is well-defined at the level of Cohen--Macaulay sheaves. 
    We conclude because, by \cite[Theorems A and B]{arinkin13}, every sheaf in $\overline{\Pic^0}( \overline{\Pic^0}(\cC_0 / \cF_0))$ is Cohen--Macaulay.
\end{proof}

\begin{Rem}\label{rem:BaseChangeToVG}
    We have the following base-change property:
    \begin{equation}\label{eq:baseChange}
       \overline{\Pic^0}(\cF/\P^5) \times_{\P^5} \cF_0 = \overline{\Pic^0}(\cF \times_{\P^5} \cF_0/\cF_0).
    \end{equation}
    Indeed, since $\Pic^{=0}(\cF/\P^5)$ commutes with base change we have 
    \[
     \overline{\Pic^0}(\cF/\P^5) \times_{\P^5} \cF_0 \subset \Pic^{=0}(\cF/\P^5) \times_{\P^5} \cF_0 = \Pic^{=0}(\cF \times_{\P^5} \cF_0/\cF_0).
    \]
    Furthermore, $\cF_0 \to \P^5$ is smooth with integral fibers, so $\overline{\Pic^0}(\cF/\P^5) \times_{\P^5} \cF_0 $ is integral.   
    It contains the relative Picard, so it must coincide with its closure.
\end{Rem}

\begin{Prop}\label{prop:CompactifiedPicCM}
    The variety $\overline{\Pic^0}(\cF/\P^5)$ is contained in the Cohen--Macaulay locus.
    Moreover, the inclusion
    \[
    \Pic^0(\cF/\P^5) \subset \overline{\Pic^0}(\cF/\P^5)
    \]
    has complement of codimension 2. 
\end{Prop}

\begin{proof}
    Since $ \overline{\Pic^0}( \overline{\Pic^0}(\cC_0 / \cF_0))$ is proper over $\cF_0$, the schematic image of $\gamma^*$ is closed and proper over $\cF_0$. 
    Clearly it contains the line bundles, and it is integral because so is $\overline{\Pic^0}( \overline{\Pic^0}(\cC_0 / \cF_0))$. 
    Therefore, it must be equal to $\overline{\Pic^0}(\cF \times_{\P^5} \cF_0/\cF_0)$, which a fortiori must be contained in $\Pic^{-0}(\cF \times_{\P^5} \cF_0/\cF_0)$. 
    By \cref{rem:BaseChangeToVG} this means that 
    \[
     \overline{\Pic^0}(\cF/\P^5) \times_B \cF_0 \subseteq \Pic^{-0}(\cF/\P^5) \times_B \cF_0,
    \]
    which, projecting on the first factor, gives the first statement. 

    As for the second part, notice that the fibers of $\overline{\Pic^0}(\overline{\Pic^0}(\cC_0/\cF_0))$ are irreducible by \cite[Theorem B]{arinkin13}.
    Since $\gamma^*$ is surjective onto $\overline{\Pic^0}(\cF \times_{\P^5} \cF_0/\cF_0)$, it follows that the same holds for the fibers of $\overline{\Pic^0}(\cF/\P^5)$.
    The complement of the locally free locus lives over the divisor in $\P^5$ which parameterizes singular cubics.
    Over each singular cubic it is a divisor in the fiber, therefore it has codimension 2 globally. 
\end{proof}

\begin{Prop}\label{thm:CompactifiedPicSmooth}
     Restriction induces an isomorphism
     \[
     i^* : \overline{\Pic^0}(\cF \times_{\P^5} \cF_0 /\cF_0) \cong \overline{P}_{\cF_0}, 
     \]
     which descends to an isomorphism 
     \[
        \overline{\Pic^0}(\cF/\P^5) \cong J_Y. 
     \]
     In particular, $\overline{\Pic^0}(\cF /\P^5)$ is a smooth HK manifold of type OG10, and the natural map to $\P^5$ is a Lagrangian fibration. 
\end{Prop}
   
\begin{proof}
    Over the open locus of smooth hyperplane sections, the restriction \eqref{eq:Restriction} factors through the relative compactified Prym by a result of Mumford, as explained in \cite[Proposition 5.3.10]{huybrechts23}. 
    Taking closures, we get 
     \begin{equation}
            i^* :  \overline{\Pic^0}(\cF \times_{\P^5} \cF_0/\cF_0) \to \oP_{\cF_0}.
    \end{equation}
    It is an isomorphism over the dense open of smooth hyperplane sections, and so by the properness of the source it is surjective. 
    
    We now show that it is an isomorphism everywhere. 
    On the fiber over $(l,H) \in \cF_0$, we have a morphism of algebraic groups
    \[
    i_{(l,H)}^*: \Pic^0(F(Y_H)) \to \Prym(C_l/D_l).
    \]
    Both are connected and 5-dimensional, and $i^*$ is dominant.
    General theory of algebraic groups implies that it is surjective with finite kernel. 
    It follows that 
    \[
        i^* : \Pic^0(\cF \times_{\P^5} \cF_0/\cF_0 ) \to \Prym(\cC_0/\cD_0)
    \]
    is surjective, birational, and quasi-finite. 
    Since the target is smooth, it is an isomorphism by Zariski's main theorem. 
    Therefore the pullback $i^*$ between the compactifications is an isomorphism over the locus of locally free sheaves, which has complement of codimension $2$ in the source.
    Since the target $\oP_{\cF_0}$ is smooth, it is an isomorphism (see \cite[1.40]{debarre01}).

    To see that $ i^*$ descends to an isomorphism
    \[
     \overline{Pic^0}(\cF/\P^5) \cong J_Y
    \]
    it is enough to see that it descends over an open subset. 
    Indeed, then the closure of the graph of the induced birational map will give an isomorphism, because it is an isomorphism after a smooth base change. 
    Over the locus of smooth hyperplane sections $U \subset \P^5$, the composition
    \[
     \Pic^0(\cF \times_{U} \cF_{0,U}/\cF_{0,U} ) \isomor \oP_{\cF_0}|_U =   J_Y \times_{U} \cF_{0,U}
    \]
    reduces to the ``Abel--Jacobi'' map induced by the universal line $\L_U \subset \cV_U \times_U \cF_U$, see \cite[Proposition 5.3.10]{huybrechts23} and \cite[Section 5]{lsv17}.
    This descends over $U$ by construction.
\end{proof}

\begin{Rem}\label{rem:FibersOfRelativePic}
In particular this says that $\overline{\Pic^0}(\cF/\P^5)$ commutes with taking fibers, i.e.
\[
\overline{\Pic^0}(\cF/\P^5)|_H = \overline{\Pic^0}(F(Y_H)).
\]
This is not true in general for a flat family $\mathcal{X} \to S$ of integral varieties, because taking the closure of the relative Picard could add components in the fiber. 
This happens for example for a degeneration of smooth curves to a singular curve with non-planar singularities. 
\end{Rem}

\begin{Rem}\label{rem:SmoothnessWithoutLsv}
Since the morphism $\cF_0 \to \P^5$ is smooth, surjective, with integral fibers, it follows that $\overline{\Pic^0}(\cF/\P^5)$ is smooth and flat over $\P^5$, even without using the descent results of \cite{lsv17}. 
We need these results to see that it is a HK manifold of type OG10.
\end{Rem}

\section{Nodal cubic threefolds}\label{sec:NodalCubics}
In the previous section we studied the structure of the reduced variety $\overline{\Pic^0}(\cF/\P^5)$. 
In this section we build the foundation for understanding its natural scheme structure, at least along the locus of nodal cubics. 

Let $V \subset \P^4$ be a cubic threefold with exactly one ordinary double point $p$.
Projection from the node gives a rational map $V \dashrightarrow \P^3$, which is resolved by blowing-up the node.
The geometry is summarized in the following diagram.
\[\begin{tikzcd}
	& Q & {\wt{V}} & E \\
	{\{p\}} & V && {\P^3} & D
	\arrow["b"', from=1-3, to=2-2]
	\arrow["\pi", from=1-3, to=2-4]
	\arrow[hook, from=2-1, to=2-2]
	\arrow[from=1-2, to=2-1]
	\arrow[hook, from=1-2, to=1-3]
	\arrow[hook', from=2-5, to=2-4]
	\arrow[from=1-4, to=2-5]
	\arrow[hook', from=1-4, to=1-3]
	\arrow[dashed, from=2-2, to=2-4]
\end{tikzcd}\]
The exceptional divisor $Q$ of the blow-up $b : \wt{V} \to V$ is a non-singular quadric.
The resolution of the projection $\pi : \wt{V} \to \P^3$ is the blow-up of a smooth genus 4 curve $D \subset \P^3$. 
This curve parameterizes lines in $V$ through the node, and it is the intersection of a cubic and a quadric. 
Details can be found in \cite{clemens72}.

As recalled in \cref{rem:VeryGoodImpliesSmooth}, the surface $F(V)$ is singular along $D$.
It is resolved by the normalization, which is given by 
\begin{equation}\label{eq:NormalizationOfFano}
    n : D^{(2)} \to F(V), \quad l + l' \mapsto l'',
\end{equation}
where $l''$ is the residual line in the intersection $\Span(l,l') \cap  V$. 
There are two copies of $D$ embedded in $D^{(2)}$, obtained via the two rulings on the non-singular quadric $Q$. 
More specifically, for each point $x \in D$, there are two lines in $Q$ that contain $x$, one for each ruling. 
Since $D$ is the intersection of $Q$ and a cubic, each of these lines meets $D$ in two more points, giving a point in $D^{(2)}$. 
The normalization glues together the images $D_1$ and $D_2$ of these two embeddings. 

\begin{Lem}[{{\cite[Corollary 4.1, Inters 4.4]{geer10}}}]\label{lem:DegreeOfNormals}
    The divisors $D_1$ and $D_2$ are algebraically equivalent in $D^{(2)}$ but not linearly equivalent, and their intersection number is $(D_1.D_2) = 0$.
\end{Lem}

\subsection{Cohen--Macaulay sheaves on $F(V)$}
The compactified Picard of the surfaces of lines on nodal cubic threefolds was first studied in \cite[Section 8]{geer10}, where it is shown to be geometrically similar to the compactified Picard of nodal curves.
For a review on the compactified Picard of integral planar curves we refer to \cite{cook98,lsv17,rego80}.
The essential fact that we want to generalize is the following:
If $G$ is a torsion-free sheaf of rank 1 on a locally planar integral curve $C$, there is a line bundle $L$ on its normalization $n: \wt{C} \to C$ such that $n_*L = G$.

\begin{Defi}\label{def:PlanarInCod1}
    Let $Z$ be a reduced but possibly reducible complex variety $Z$.
    We say that it has \textit{planar singularities in codimension} $1$ if it is singular along a smooth divisor $D \subset Z$, and étale locally near every $d \in D$ there is a pointed curve $(C,0)$ with only a planar singularity in $0$, and an isomorphism 
    \[
       Z \cong_{\text{loc}} D \times C \ \text{ such that }  D \mapsto D \times \{ 0\}.
    \]
    We say that $Z$ has \textit{nodal singularities in codimension} $1$ if moreover $C$ has a simple node in $0$.
\end{Defi}

If $V$ is a cubic threefold with a single node, then $F(V)$ has nodal singularities in codimension 1 by \cite[Section 7]{clemens72}.
If $V$ has a unique singular point of type ADE, then its surface of lines is singular along the smooth curve $D$ of lines through the singular point by \cite[Corollary 1.11]{altman77}.

\begin{Rem}\label{rem:NormalizationOfPlanarSingularities}
    If $Z$ has planar singularities in codimension $1$, then it is resolved by its normalization $n: \wt{Z} \to Z$, which étale locally is given by $D \times \wt{C} \to D \times C$.
    The inverse image $\wt{D} = n^{-1}(D)$ may have some non-reduced connected components, but the restriction of the normalization to its reduction $n|_{\wt{D}_{\rm{red}}}: \wt{D}_{\rm{red}} \to D$ is étale.
    In particular, $\wt{D}_{\rm{red}}$ is smooth. 
    If the singularities are nodal, then $\wt{D}$ is also reduced.
\end{Rem} 

\begin{Prop}\label{prop:CMisPushOfLineBundleGeneral}
    Let $Z$ be a complex variety with planar singularities in codimension $1$.
    Then, for every sheaf $G$ on $Z$ which is $S_2$ of rank $1$ but not locally free, there exists a line bundle $L \in \Pic(\wt{Z})$ such that $n_*L = G$.
    In particular, $G$ is Cohen--Macaulay. 
\end{Prop}

\begin{proof}
    This amounts to gluing the family version of the usual statement for locally planar curves. 
    Given $G$, we define the line bundle as 
    \[
     L \coloneqq (n^*G)^{\vee \vee} \in \Pic(\wt{Z}).
    \]
    It is locally free, because it is reflexive of rank $1$ on a smooth variety.
    There is a natural map $n^*G \to L$, which by adjunction gives 
    \begin{equation}\label{eq:MorphismToPushforward}
         \varphi: G \to n_*L.
    \end{equation}
    First we prove that $n_*L$ is a Cohen--Macaulay sheaf. 
    Away from $D$ it is a line bundle on a smooth variety; so assume $Z = D \times C$ and choose a point $d \in D$.
    Up to taking a smaller neighborhood of $d$ in the étale topology, we may assume that $L = p_2^*(L')$ for some $L' \in \Pic(\wt{C})$. 
    Then, by flat base change, we have 
    \[
    n_*L = n_*(p_2^*L') = p_2^*(n_*L') \in \Coh(D \times C)
    \]
    This is Cohen--Macaulay as it is the pullback of the Cohen--Macaulay sheaf $n_*L'$ via the flat map $p_2$.
    
    Since both $G$ and $n_*L$ are $S_2$, it suffices to see that $\varphi$ is an isomorphism away from codimension $2$.
    Away from $D$ it is an isomorphism because $n$ is, and $G$ is $S_2$.
    Hence we can assume that $Z = D \times C$ as in the statement. 
    By generic flatness, there is an open in $D$ over which $G$ and $n_*L$ are flat. 
    Since the closed complement of this subset has codimension $2$ in $Z$, we may assume that $G$ and $n_*L$ are flat over $D$.
    We may also assume that $n^*G \to (n^*G)^{\vee \vee}$ is surjective, as its cokernel is supported in codimension $2$.  
    
    For every $d \in D$, the restriction $G|_d$ is a torsion-free sheaf on $C$ by flatness.
    Since $(n_*L)|_d = n_*(L|_d)$, because $n$ is finite, we conclude if we show there is a canonical identification 
    \begin{equation}\label{eq:PushforwardforFibers}
         (n^*G|_d)^{\vee \vee} \cong L|_d.
    \end{equation}
    Indeed, $G|_d$ is a torsion-free sheaf on a planar curve, for which it is true that $\varphi|_d$ is an isomorphism $G|_d \cong n_*((n^*G|_d)^{\vee \vee})$. 
    
    To prove \eqref{eq:PushforwardforFibers} notice that there is a natural map
    \[
    (n^*G|_d)^{\vee \vee} \to (n^*G)^{\vee \vee}|_d = L|_d,
    \]
    obtained by restriction from $n^*G \to (n^*G)^{\vee \vee}$ and factoring through the double dual. 
    This map is always surjective, because it is a factorization of a surjective map, and generically is an isomorphism. 
    Since the source is torsion-free, it must be an isomorphism. 
\end{proof}

\begin{Lem}\label{lem:IndependenceOnL}
    Let $Z$ be an irreducible projective variety, and let $n: \wt{Z} \to Z$ be a finite morphism from an irreducible projective variety $\wt{Z}$. 
    Then for any line bundle $L \in \Pic(\wt{Z})$, there is an isomorphism
    \[
     \RcHom(n_*L,n_*L) \cong \RcHom(n_*\sO_{\wt{Z}},n_*\sO_{\wt{Z}}).
    \]
    In particular $\Ext^i(n_*L,n_*L) \cong \Ext^i(n_*\sO_{\wt{Z}},n_*\sO_{\wt{Z}})$ for every $i \in \Z$.
\end{Lem}

\begin{proof}
    Let $\sL$ be a universal line bundle on $\wt{Z} \times \Pic(\wt{Z})$, and consider $\sG \coloneqq (n \times \id)_*(\sL) \in \Coh(Z \times \Pic(\wt{Z}) )$.
    It is flat with respect to the second projection, and commutes with base change because $n$ is finite, so $\sG|_{L} = n_*L$ for every $L \in \Pic(\wt{Z})$.
    The complex $ \RcHom (\sG,\sG)$ is compatible with (derived) restriction. 
    So, for every $L \in \Pic(\wt{Z})$ we have 
    \[
     \RcHom (\sG,\sG)|_{\wt{Z} \times \{ L \}} = \RcHom(n_*L,n_*L).
    \]
    Let $m_{L}$ be the automorphism of $\Pic(\wt{Z})$ given by tensor product with $L$. 
    We have 
    \[
    (\id \times m_{L})_*\sG=  (n \times \id)_*((\id \times m_{L})_*\sL) = (n \times \id)_*(\sL \times p_2^*\rho(L)) = \sG \otimes p_2^*\rho(L),
    \]
    for some line bundle $\rho(L)$ on $\Pic(\wt{Z})$.
    This is because two universal line bundles are unique up to pullback of a line bundle on $\Pic(\wt{Z})$. 
    Therefore we have 
    \[
      (\id \times m_{L})_*\RcHom (\sG,\sG) = \RcHom ( (\id \times m_{L})_*\sG, (\id \times m_{L})_*\sG) = \RcHom (\sG,\sG).
    \]
    Restricting to $\sO_{\wt{Z}}$ gives 
    \[
    \RcHom(n_*L,n_*L) =  (\id \times m_{L})_*\RcHom (\sG,\sG)|_{\sO_{\wt{Z}}} \cong  \RcHom (\sG,\sG)|_{\sO_{\wt{Z}}} = \RcHom(n_*\sO_{\wt{Z}},n_*\sO_{\wt{Z}}).
    \]
    Taking (derived) global sections and cohomology gives the second result. 
\end{proof}

We will apply the above result to normalizations, closed embeddings, and compositions of the two. 

\subsection{Towards smoothness}
Our goal is to compute the dimension of the space of deformations $\Ext^1_X(i_*G,i_*G)$, where $Z$ is the surface of lines on a nodal cubic threefold, $i:Z \hookrightarrow X$ its embedding inside the variety of lines of a cubic fourfold, and $G$ is a Cohen--Macaulay sheaf of rank $1$ on $Z$.
In general, there exists a spectral sequence 
\begin{equation}\label{eq:SpectralSequenceNormal}
   E^{p,q}_2 = \Ext^p(G, G \otimes \bwed^q N_{Z/X}) \implies \Ext^{p+q}(i_*G,i_*G),
\end{equation}
because $i$ is a regular embedding.
In this case the normal bundle is the dual of the tautological bundle $\sS^{\vee}$.

The surface of lines on a nodal cubic threefold has nodal singularities in codimension $1$.
If $G$ is not locally free, we can write $G = n_*L$ for $L \in \Pic(\wt{Z})$ by \cref{prop:CMisPushOfLineBundleGeneral}. 
By \cref{lem:IndependenceOnL} it suffices to consider the case $L = \sO_{\wt{Z}}$.
The first terms of the spectral sequence \eqref{eq:SpectralSequenceNormal} are 
\begin{equation}\label{eq:FirstTermsNormal}
    0 \to \Ext^1(n_*\sO_{\wt{Z}},n_*\sO_{\wt{Z}}) \to \Ext^1_X(i_*n_*\sO_{\wt{Z}},i_*n_*\sO_{\wt{Z}}) \to \Hom(n_*\sO_{\wt{Z}},n_*\sO_{\wt{Z}} \otimes \sS^{\vee}).
\end{equation}
Our goal is to show that the middle space has dimension $10$. 
We will conclude this in \cref{lem:SmoothnessCod3}, for now we bound its dimension by $11$.

We start by computing the leftmost space. 
Here we use the local-to-global spectral sequence to get 
\begin{equation}\label{eq:LocalToGlobalFirstTerms}
 0 \to H^1(Z,\cEnd(n_*\sO_{\wt{Z}})) \to \Ext^1(n_*\sO_{\wt{Z}},n_*\sO_{\wt{Z}}) \to H^0(Z,\cExt^1(n_*\sO_{\wt{Z}}n_*\sO_{\wt{Z}})).
\end{equation}
We start with the space on the right.

\begin{Lem}\label{lem:ExtOfPush}
    Let $Z$ be a variety with nodal singularities in codimension $1$, and let $n: \wt{Z} \to Z$ be its normalization.
    For any line bundle $L \in \Pic(\wt{Z})$, there are isomorphisms of rank $2$ bundles on $D$
    \[
     \cExt^1( n_*L, n_*L) \cong N_{D/Z} \cong n_*\sO_{\wt{D}}(\wt{D}),
    \]
    where $\wt{D} \coloneqq n^{-1}(D)$. 
\end{Lem}

\begin{proof}
    Thanks to \cref{lem:IndependenceOnL} we may assume $L = \sO_{\wt{D}}$.
    Locally analytically, the normalization identifies two isomorphic varieties along two reduced copies of the divisor $D$. 
    Therefore, $n|_{n^{-1}(D)} : n^{-1}(D) \to D$ is étale of degree $2$ over $D$, in particular $\wt{D}$ is smooth. 
    A local computation shows that $n_*\sO_{\wt{Z}} \cong I_{D}^{\vee}$, because this holds for nodal curves. 
    Dualizing this, we get $ n_*\sO_{\wt{Z}}^{\vee} = I_D$, since $n_*\sO_{\wt{Z}}$ is Cohen--Macaulay by \cref{prop:CMisPushOfLineBundleGeneral}, and $I_D$ is reflexive. 
    Therefore, we get 
    \[
    \RcHom(n_*\sO_{\wt{Z}},n_*\sO_{\wt{Z}}) \cong \RcHom(I_D,I_D).
    \]
    In particular $\cExt^1(n_*\sO_{\wt{Z}},n_*\sO_{\wt{Z}}) \cong \cExt^1(I_D,I_D)$.
    The $\cHom$ long exact sequence of the ideal short exact sequence is 
    \begin{equation}\label{eq:LongExactHomSequence}
    \cHom(I_D,I_D) \to \cHom(I_D,\sO_Z) \to \cHom(I_D,\sO_D) \to \cExt^1(I_D,I_D) \to 0,
    \end{equation}
    where the zero on the right is because $I_D$ is a CM sheaf. 
    Étale locally we can assume $Z = D \times C$ singular along $D \times \{ 0\}$, with $C = \rm{Spec}(\C[x,y]/(xy))$ and $D$ smooth.
    An explicit computation at the ring level shows that the first map in \eqref{eq:LongExactHomSequence} is surjective.
    Therefore the last map is an isomorphism
    \[
    N_{D/Z} = \cHom(I_D,\sO_D) \cong \cExt^1(I_D,I_D).
    \]
    It remains to compute the normal sheaf $N_{D/Z}$.
    By definition of $\wt{D}$ there is a morphism $n^*(\sC_{D/Z}) \to \sC_{\wt{D}/\wt{Z}}$, induced by the codifferential of $n$.
    By adjunction we get 
    \begin{equation}\label{eq:Codifferential}
       \sC_{D/Z} \to n_*\sC_{\wt{D}/\wt{Z}}  = n_*\sO_{\wt{D}}(-\wt{D}).
    \end{equation}
    Again, we assume that $Z = D \times C$ as above.
    The conormal sheaf is simply 
    \[
    \sC_{D/Z} = T^{\vee}_{C,0} \otimes \sO_{D}.
    \]
    If $0_{1}$ and $0_2$ denote the two preimages of $0$ under the normalization $\wt{C} \to C$, we have 
    \[
    n_*\sC_{\wt{D}/\wt{Z}} = \left(T_{\wt{C},0_1} \oplus T_{\wt{C},0_2}\right)^{\vee} \otimes \sO_{D}.
    \]
    We conclude that the map \eqref{eq:Codifferential} is an isomorphism, because for the nodal curve $T_{\wt{C},0_1} \oplus T_{\wt{C},0_2} \to T_{C,0}$ is an isomorphism. 
    Since $n|_{\wt{D}}: \wt{D} \to D $ is étale, $n_*$ commutes with dualization.
    Thus we get the desired isomorphism for the normal sheaf as well. 
\end{proof}

\begin{Cor}\label{cor:ExtOnNodal}
    Let $Z$ be the surface of lines on a cubic threefold with one node.
    Then, for every $G$ Cohen--Macaulay sheaf of rank $1$ on $Z$ we have
    \[\ext^1(G,G) = 
    \begin{cases}
        5 & \text{ if } G \text{ is locally free,} \\
        6 & \text{ otherwise.} 
    \end{cases}
    \]
\end{Cor}

\begin{proof}
    Recall from \cite{altman77} that $h^1(Z,\sO_Z) = 5$, so if $G$ is locally free then $\ext^1(G,G) = 5$.
    If $G$ is not locally free, we write $G = n_*L$ as in \cref{prop:CMisPushOfLineBundleGeneral}. 
    As in \cref{lem:IndependenceOnL} assume that $L = \sO_{\wt{Z}}$, and consider the first terms of the local-to-global spectral sequence \eqref{eq:LocalToGlobalFirstTerms}:
    \[
     0 \to H^1(Z,\cEnd(n_*\sO_{\wt{Z}})) \to \Ext^1(n_*\sO_{\wt{Z}},n_*\sO_{\wt{Z}}) \to H^0(Z,\cExt^1(n_*\sO_{\wt{Z}}n_*\sO_{\wt{Z}})).
    \]
    
    By \cref{lem:ExtOfPush} we get  
    \[
    H^0(Z,\cExt^1(n_*\sO_{\wt{Z}},n_*\sO_{\wt{Z}})) = H^0(D,\sO_{D_1}(D_1) \oplus \sO_{D_2}(D_2)).
    \]
    where $D_1$ and $D_2$ are the two copies of $D$ lying over it.  
    By \cref{lem:DegreeOfNormals} the normal bundles $\sO_{D_i}(D_i)$ have degree $0$.
    Therefore, this space at dimension at most $2$, and exactly $2$ if and only if they are both trivial. 

    The leftmost space is equal to $H^1(\wt{Z},\sO_{\wt{Z}}) = H^1(D^{(2)},\sO_{D^{(2)}})$, which has dimension $4$. 
    Indeed, there is a natural morphism
    \begin{equation}\label{eq:ComputationOfCurlyEnd}
        n_* \sO_{\wt{Z}} \to \cEnd(n_*\sO_{\wt{Z}}).
    \end{equation}
    Locally the normalization $n$ corresponds to the embedding of an integral domain inside its integral closure $A \hookrightarrow B$.
    The morphism \eqref{eq:ComputationOfCurlyEnd} becomes 
    \[
    B \to \End_A(B), \quad  b \mapsto b \cdot -, 
    \]
    which is an isomorphism because $B$ lives in the quotient field of $A$. 

    Combining everything we get $\ext^1(n_*\sO_{\wt{Z}},n_*\sO_{\wt{Z}}) \leq 4 + 2 = 6$.
    Now assume that it is $5$.
    Every sheaf in the boundary of $\overline{\Pic^0}(Z)$ is Cohen--Macaulay, so the compactified Picard would be smooth by \cref{lem:IndependenceOnL}.
    The isomorphism of \cref{thm:CompactifiedPicSmooth} that it is not, so we get a contradiction. 
\end{proof}

We have computed the dimension of the space on the left in \eqref{eq:FirstTermsNormal}, now we move to the rightmost space.
Locally analytically, the embedding $i: Z \hookrightarrow X$ looks like two smooth Lagrangian surfaces meeting along a smooth curve. 
We can compute the pullback of the normal bundle $n^*\sS^{\vee}$ to the normalization $\wt{Z}$ via the following result. 

\begin{Lem}\label{lem:pullbackOfNormal}
    Let $(X,\sigma)$ be a (not necessarily compact) complex symplectic manifold, and let $Z_1, Z_2 \subset X$ be two Lagrangian submanifolds meeting along a smooth divisor $D = Z_1 \cap Z_2$. Let $Z = Z_1 \cup Z_2$, $n: \tilde{Z} \to Z$ be the normalization, and $D_1 \cup D_2 \subset \tilde{Z}$ the two preimages of the divisor $D$. 
    Then there is an isomorphism
    \[
    n^*(\sC_{Z/X}) \cong T_{\tilde{Z}}(-\log(D_1 + D_2))
    \]
    induced by the symplectic form.
\end{Lem}

\begin{proof}
    Consider the composition 
    \[
     n^*(\sC_{Z/X}) \to n^*(\Omega_X) \isomor n^*T_X.
    \]
    Away from $D_1,D_2$ the normalization map $n$ is an isomorphism onto a Lagrangian submanifold, hence this composition maps isomorphically onto $T_{\tilde{Z}}$.
    Therefore, to conclude it suffices to see that, restricted to $D_i$, it factors as 
    \begin{equation}\label{eq:factorizationNormal}
    n^*(\sC_{Z/X})|_{D_i} \twoheadrightarrow T_{D_i} \hookrightarrow n^*T_X|_{D_i}.
    \end{equation}
    
    Say that $X$ has dimension $2n$. 
    We can choose local coordinates for $X$ such that $Z_1 = \{ x_1 = \dots = x_n = 0 \}$ and $Z_2 = \{x_2 = \dots x_{n+1} = 0 \}$.
    Then $Z = \{ x_1x_{n+1} = x_2 = \dots x_n = 0\}$, and $D = \{  x_2 = \dots x_n = 0 \}$.
    Via this local description, one sees that the natural map $ \sC_{Z/X}|_D \to \Omega_X|_D$ maps surjectively onto the subbundle $\Ann(T_{Z_1}|_D + T_{Z_2}|_D)$.
    Under the isomorphism $\sigma:  \Omega_X|_D \cong T_X|_D$, this corresponds to the symplectic orthogonal 
    \[
    (T_{Z_1}|_D + T_{Z_2}|_D)^{\perp} = T_D.
    \]
    Therefore the composition factors through the subbundle $T_D \subset T_X|_D$, and pulling back we get \eqref{eq:factorizationNormal}.
\end{proof}

\begin{Cor}\label{cor:PullbackOfTautological}
    Let $V$ be a nodal cubic threefold, and let $\sS$ be the tautological bundle on $F(V)$.
    We have 
    \[
    n^*(\sS^{\vee}) \cong \Omega_{C^{(2)}}(\mathrm{log}(D_1 + D_2)).
    \]
    In particular, $h^0(D^{(2)},n^*(\sS^{\vee})) = 5$.
\end{Cor}

\begin{proof}
    Choose a smooth cubic fourfold $Y$ such that $V \subset Y$ is a hyperplane section.
    The tautological bundle $\sS$ is the conormal bundle of the inclusion $F(V) \subset F(Y)$. 
    Locally analytically, near the singular locus of $F(Y_H)$, this embedding looks like two smooth Lagrangian surfaces meeting on a smooth curve.  
    Therefore, we can apply \cref{lem:pullbackOfNormal} to conclude that $n^*(\sS^{\vee}) \cong \Omega_{D^{(2)}}(\log(D_1 + D_2))$.

    To compute global sections, recall the following exact sequence 
    \begin{equation}\label{eq:ShortExactSequenceForLogCotangent}
    0 \to \Omega_{D^{(2)}} \to \Omega_{D^{(2)}}(\mathrm{log}(D_1 + D_2)) \to \sO_{D_1} + \sO_{D_2} \to 0.
    \end{equation}
    The connecting homomorphism $H^0(D_1,\sO_{D_1}) \oplus H^0(D_2,\sO_{D_2}) \to H^1(D^{(2)},\Omega_{D^{(2)}})$ sends $(1,0)$ and $(0,1)$ to the fundamental class of $D_1$ and $D_2$ respectively. 
    The two curves $D_1$ and $D_2$ are algebraically equivalent by \cref{lem:DegreeOfNormals}, so the connecting homomorphism has a one-dimensional kernel. 
    Since the space $H^0(D^{(2)},\Omega_{D^{(2)}}) $ has dimension $4$, we deduce that $h^0(D^{(2)}, n^*(\sS^{\vee})) = 5$. 
\end{proof}

\begin{Cor}\label{cor:DimensionForNodal}
    Let $X$ be the variety of lines of a smooth cubic fourfold, and $Z$ be the surface of lines of a hyperplane section with one node. 
    Let $i: Z \hookrightarrow X$ be the embedding and $G$ a Cohen--Macaulay sheaf on $Z$.
    Then $\ext^1(i_*G,i_*G) \leq 11$. 
\end{Cor}

\begin{proof}
    Consider the exact sequence \eqref{eq:FirstTermsNormal}
    \[
    0 \to \Ext^1(G,G) \to  \Ext^1(i_*G,i_*G) \to \Hom(G,G \otimes \sS^{\vee}).
    \]
    By \cref{cor:ExtOnNodal} the left space has dimension $6$. 
    For the space on the right, we write $G = n_*\sO_{\wt{Z}}$ by \cref{prop:CMisPushOfLineBundleGeneral} and \cref{lem:IndependenceOnL}.
    We have 
    \[
    \Hom(G,G \otimes \sS^{\vee}) = H^0(\wt{Z},n^*\sS^{\vee}),
    \]
    which has dimension $5$ by \cref{cor:PullbackOfTautological}.
    Thus, we conclude that $\ext^1(i_*G,i_*G) \leq 11$. 
\end{proof}

\section{Stable bundles from Lagrangians}\label{sec:TransformsOfLagrangians}
In this section we explain how the Poincaré sheaf gives projectively hyperholomorphic bundles from sheaves supported on Lagrangian subvarieties. 
Let $X' = M_{(0,H,1-g)}(S,H)$ be a Beauville--Mukai system with a section on a polarized K3 of genus $g$.
Assume that every curve in $|H|$ is integral, so that \cite{arinkin13} applies. 

Let $\pi: X \to \P^n$ be a projective Tate--Shafarevich twist of the Lagrangian fibration $\pi': X' \to |H| \cong \P^n$. \cref{thm:TwistedPoincaré} gives a twist $\theta \in \Br(X')$, and a twisted sheaf $\sP \in \Coh(X \times X',p_2^*\theta)$.
For $t \in \P^n$ we denote by $X_t$ and $X'_t$ the fibers of the Lagrangian fibrations $X \to \P^n$ and $X' \to \P^n$.
We denote by $\sP_t$ the restriction to $\Coh(X_t \times X'_t)$ of the Poincaré sheaf $\sP$. 

Note the change in notation compared to \cref{sec:TwistedPoincaré}:  $X'$ denotes the Lagrangian fibration with the section, the Tate--Shafarevich twist is now fixed to be $X$, the twisted Poincaré sheaf is $\sP$, and $t$ denotes a point in the base $\P^n$.

\begin{Lem}\label{lem:PropertiesOfP}
    For every $t \in \P^n$, denote by $i_t : X_t \hookrightarrow X$ and $i'_t : X'_t \hookrightarrow X'$ the embedding of the fibers.
    There are isomorphisms of functors
    \begin{enumerate}
        \item $\Phi_{\sP} \circ i_{t,*} \cong i'_{t,*} \circ \Phi_{\sP_t}$.
        \item $i'^{*}_{t} \circ \Phi_{\sP} \cong  \Phi_{\sP_t} \circ i^*_t$.
    \end{enumerate}
    And similarly for $\Phi^{-1}_{\sP}$.
\end{Lem}

\begin{proof}
    First notice that $\sP$ is flat over $|H|$, as it is flat over both factors. 
    Additionally, since it is supported on the fiber product, the equivalence $\Phi_{\sP}$ is $\P^n$-linear in the following sense. 
    If $F \in \Db(\P^n)$ we have 
    \[
        \Phi_{\sP}(- \otimes \pi^*(F)) =  \Phi_{\sP}(-) \otimes \pi'^*(F).
    \]
    The statement follows from this and the base change Theorem, see \cite[Lemma 5.3]{bottini22} for more details.
\end{proof}

\subsubsection{Local freeness}
Now we give a criterion to prove local freeness of the image of a sheaf under $\Phi$.
The following is simply a twisted version of \cite[Proposition 5.2]{bottini22}.

\begin{Prop}\label{prop:LocallyFreeGeneral}
    Let $Z$ be a subvariety of $X$, and $i: Z \hookrightarrow X$ the embedding. 
    Let $G$ be a rank 1 sheaf on $Z$ such that $i_*G$ is Cohen--Macaulay on $X$.
    Assume that $Z$ is finite over $\P^n$ of degree $r$.
    Then $E \coloneqq \Phi_{\sP}(i_*G)$ is a twisted locally free sheaf of rank $r$, and its restriction to a general fiber $X'_t$ splits 
        \[
            E_t = \bigoplus^{r}_{k = 1} L_{t,k}
        \]
    as a sum of $\theta|_{X'_t}$-twisted line bundles. 
    Moreover, if $i_*G$ is atomic, then $E$ is twisted atomic. 
\end{Prop}

\begin{proof}
    To verify that $E$ is locally free, it suffices to show that
    \begin{equation}\label{eq:HomologicalConditionLocalFreeness}
        \Ext^i_{\Db(X',\theta)}(E,\C(x)) = 0 \quad   \forall i >0,
    \end{equation}
    for all points $x \in X'$.
    As in the untwisted case, this shows that the derived fiber at every point is concentrated in degree $0$, ensuring that $E$ is locally free. 
    We have 
    \[
        \Ext^i_{\Db(X',\theta)}(E,\C(x)) \cong  \Ext^i_{\Db(X)}(i_*G,\Phi^{-1}_{\sP}(\C(x))).
    \]
    By construction the Brauer class $\theta|_{X'_t}$ is trivial, and $\sP_t \in \Coh(X_t \times X'_t, \theta|_{X'_t})$ differs from Arinkin's universal sheaf by a $\theta|_{X'_t}$-twisted line bundle. 
    The same argument as in \cite[Proposition 5.2]{bottini22} allows us to conclude that $E$ is locally free.

    Now let $t \in \P^n$ be general, and let $\{ x_1,\dots,x_{r} \}$ denote the points of intersection of $Z$ with $X_t$.
    By \cref{lem:PropertiesOfP}$(2)$, we have
    \[
    E_t = \Phi_{\sP_t}(i_*G|_{X_t}) = \bigoplus^{r}_{k = 1} \Phi_{\sP_t}(\C(x_k)).
    \]
    By definition of the Poincaré sheaf, if $x \in X_t$ is a smooth point the image of the skyscraper sheaf $\C(x)$ under $\Phi_{\sP}$ is a line bundle.
    Setting $L_{t,k} \coloneqq  \Phi_{\sP_t}(\C(x_k))$, the decomposition follows. 
    In particular, $E$ has rank $r$. 
    Finally, if $i_*G$ is atomic, it follows from \cref{prop:InvarianceOfTwistedAtomicity} that $E$ is also aotmic.
\end{proof}

For integral atomic Lagrangians, finiteness over the base is the same as being generically finite, and thus it is a cohomological condition.

\begin{Prop}\label{prop:FinintessOverBaseLagrangians}
    Let $\pi : X \to \P^n$ a Lagrangian fibration, and let $i: Z \hookrightarrow X$ be an integral Lagrangian such that $i_*\sO_Z$ is atomic. 
    If $\pi_*[Z] = r[\P^n] \in H^{0}(\P^n,\Z)$ with $r >0 $, then $\pi|_Z : Z \to \P^n$ is finite of degree $r$.
\end{Prop}

\begin{proof}
    The equality $\pi_*[Z] = r[\P^n]$ means that $Z$ dominates $\P^n$, and thus it is generically finite, since $Z$ is integral of dimension $n$. 
    Let $\nu : \wt{Z} \to Z$ be a resolution of singularities, and let $j = i \circ \nu : \wt{Z} \to X$ the composition. 
    By \cite[Lemma 2.7]{lehn16} we have 
    \begin{equation}\label{eq:RestrictionSingular}
        \Ker \left(j^* :H^2(X,\Q) \to H^2(\wt{Z},\Q) \right) = \Ker \left( H^2(X,\Q) \xrightarrow{[Z] \cup -}  H^{2n+2}(X,\Q)\right).
    \end{equation}
    Write $\tilde{v}(i_*\sO_Z) = \lambda + s\beta$ for some $\lambda \in \NS(X)$ and $s \in \Q$.
    We show that 
    \[
        \langle \lambda \rangle^{\perp} = \Ker \left( H^2(X,\Q) \xrightarrow{[Z] \cup -}  H^{2n+2}(X,\Q)\right).
    \]
    Take $\mu \in H^2(X,\Q)$ be any class, and let $B_{\mu} \in \mathrm{SO}(H^*(X,\Q))$ be the isometry acting as cup product with $\exp(\mu)$.
    Its action on $\wt{H}(X,\Q)$ is given in \cite[Section 4.1]{beckmann21}, and if $q(\mu,\lambda) =0$ we have $B_{\mu}(\tilde{v}(i_*\sO_Z)) = \tilde{v}(i_*\sO_Z)$.
    Therefore $B_{\mu}(v(i_*\sO_Z)) = v(i_*\sO_Z)$, by definition of atomicity. 
    Since 
    \[
        [Z] \cup \mu = B_{\mu}(\ch_{2n+2}(i_*\sO_Z)) - \ch_{2n+2}(i_*\sO_Z) = 0
    \]
    we conclude that $\langle \lambda \rangle^{\perp} \subseteq \Ker([Z] \cup -)$.
    As the latter space cannot be the whole $H^2(X,\Q)$ this must be an equality. 

    The equality \eqref{eq:RestrictionSingular} implies that
    \[
        \rk \left( j^* : \NS(X)_{\Q} \to \NS(\wt{Z})_{\Q} \right) = 1.
    \]
    Since $\nu: \wt{Z} \to Z$ is surjective, the pullback $\nu^* : \NS(Z)_{\Q} \to \NS(\wt{Z})_{\Q}$ is injective.
    Here, if $Z$ is singular, the Neron--Severi group $\NS(Z)$ is the group of line bundles up to numerical equivalence. 
    It follows that the restriction $i^*: \NS(X)_{\Q} \to \NS(Z)_{\Q}$ has also rank $1$. 
    Therefore
    \[
        \pi|_Z^*(\sO_{\P^n}(1)) = i^*\pi^*(\sO_{\P^n}(1)) \in \NS(Z)_{\Q},
    \]
    must be a positive (as $r > 0$) multiple of an ample class, so it is itself ample.
    Since the pullback of the ample class $\sO_{\P^n}(1)$ is ample, the map $\pi|_Z: Z \to \P^n$ is finite. 
    The cohomological assumption forces its degree to be $r$. 
\end{proof}

\subsubsection{Slope stability}\label{sec:SlopeStability}
We follow Yoshioka's approach \cite{yoshioka06} for the construction of moduli spaces of twisted sheaves, as moduli spaces of certain sheaves on a Brauer--Severi variety representing the Brauer class. 
Although his convention for the Chern character of twisted sheaves is slightly different from \cite{huybrechts2005}, it is perfectly equivalent.
Recall that the slope of a twisted sheaf $E$ on a polarized variety $(X,h)$ is defined as 
\begin{equation}\label{eq:DefSlope}
    \mu^B_h(E) \coloneqq 
    \begin{cases}
        \frac{c_1^B(E).h^{\dim X - 1}}{r(E)} &\text{ if } r(F) > 0, \\
        + \infty &\text{ if } r(E) = 0.
    \end{cases}
\end{equation}
Here, $B$ is a fixed B-field for the Brauer class, and $c_1^B$ is the twisted first Chern class.
As usual, a torsion-free sheaf $E$ is $\mu_h$-semistable (or simply $\mu$-semistable) if for all subsheaves of smaller rank $F \subset E$ the inequality $\mu^B_h(F) \leq \mu^B_h(E)$ holds, and $\mu_h$-stable if the strict inequality holds. 
Changing $B$ changes the slope by an additive constant, hence the corresponding stability condition is well-defined and does not depend on $B$. 
For this reason, we will often omit $B$ from the notation and write $\mu_h$ instead of $\mu^B_h$.

Now we need to extend some results proved by O'Grady in \cite{ogrady22} to the twisted setting.
For simplicity, we restrict our discussion to locally free sheaves. 
Recall from \cite[Definition 3.3]{ogrady22} the definition of the constant $a$ for a modular sheaf, used for the wall and chamber decomposition. 

\begin{Defi}
    Let $X$ be a HK manifold, $\theta \in \Br(X)$ a Brauer class, and let $E$ be be a $\theta$-twisted locally free sheaf. 
    We say that $E$ is modular if $\cEnd(E)$ is modular, and we set
    \[
        a(E) \coloneqq \frac{a(\cEnd(E))}{r(E)^4},
    \]
    where $r(E)$ is the rank of $E$.
\end{Defi}

Following \cite[(3.2.1)]{ogrady22}, if $F$ is a $\theta$-twisted torsion-free sheaf and $E$ is a $\theta$-twisted locally free sheaf, we set 
\[
    \lambda_{F,E} \coloneqq c_1(E^{\vee} \otimes F) = r(E)c^B_1(F) - r(F)c^{B}_1(E).
\]
The last equality holds for any B-field $B \in H^2(X,\Q)$ for $\theta$.
With this definition, the main technical results of \cite{ogrady22} extend naturally to the twisted setting by reducing to the untwisted case. 
Indeed, tensorization with $E^{\vee}$ maps $\theta$-twisted vector bundles into untwisted bundles, and we have
\begin{equation}\label{eq:KeyEqualityForTwistedModular}
     \lambda_{F,E} = \frac{1}{r(E)^2 }\lambda_{E^{\vee} \otimes F,E^{\vee} \otimes E} .
\end{equation}

Recall from \cite[Definition 3.5]{ogrady22}, that a class $h \in \Amp(X)_{\R}$ is $a$-suitable 
if, for every hyperplane orthogonal to a class $\lambda \in \NS(X)$ with $-a \leq q(\lambda) < 0$, $h$ lies on the same half space as $f$.
We need the following generalization to twisted bundles of \cite[Proposition 5.11(a)]{ogrady24}.

\begin{Prop}\label{prop:SemistabilityRestriction}
    Let $X \to \P^n$ a Lagrangian fibration, $\theta \in \Br(X)$ a Brauer class and $E$ a $\theta$-twisted modular vector bundle on $X$.
    Let $h$ be an ample class that is $a(E)$-suitable.
    Assume that $E_t$ is $\mu_{h_t}$-semistable for a general $t \in \P^n$.
    If $E$ is not $\mu_h$-stable, then there exists a subsheaf $F \subset E$ with $0 < r(F) < r(E)$ such that $\mu_{h_t}(F_t) =\mu_{h_t}(E_t) $ for a general $t \in \P^n$, and $\mu_h(F) = \mu_h(E)$.
\end{Prop}

\begin{proof}
    We follow the proof of \cite[Proposition 5.11(a)]{ogrady24} and indicate where care is needed due to the twist. 
    Let $\mathsf{S}$ be the set of rational numbers $s$ such that there is a $\theta$-twisted subsheaf $F \subset E$ with smaller rank and
    \[
        q(\lambda_{F,E}, (1-s)h + sf) =0.
    \]
    We show that $S$ is non-empty and finite. 
    Since $E$ is not $\mu_h$-stable, we can find $F \subset E$ with smaller rank and $\mu_h(F) \geq \mu_h(E)$. 
    Then 
    \[
        \mu_h(E^{\vee} \otimes F) \geq \mu_h(E^{\vee} \otimes E),
    \]
    thus \cite[Lemma 3.7]{ogrady22} combined with \eqref{eq:KeyEqualityForTwistedModular} gives $q(\lambda_{F,E},h) \geq 0$. 
    Similarly, since $E_t$ is semistable, $\cEnd(E)_t = \cEnd(E_t)$ is semistable by \cite[Corollary 3.2.10]{huybrechts_lehn_2010}.
    It follows from \cite[Lemma 3.11]{ogrady22} and  \eqref{eq:KeyEqualityForTwistedModular} that $q(\lambda_{F,E},f) \leq 0$.
    From these two inequalities it follows that $\mathsf{S}$ is non-empty.
    At the same time it is finite, because the corresponding set for $\cEnd(E)$ is finite.
    The equality 
    \[
        r(E)r(F)(\mu_{h_t}(F_t) - \mu_{h_t}(E_t)) = n!c_Xq(h,f)^{n-1}q(\lambda_{F,E},f)
    \]
    works in the twisted setting as well.
    Notice that the difference of the slopes does not depend on the B-field.
    So assume $\mu_{h_t}(F_t) - \mu_{h_t}(E_t) < 0$, or equivalently $q(\lambda_{F,E},f) < 0$.
    Using \cite[Proposition 3.10]{ogrady22}, we get 
    \[
        -a(\cEnd(E)) \leq q(\lambda_{E^{\vee} \otimes F,E^{\vee} \otimes E}) \leq 0,
    \]
    thus by definition 
    \[
        -a(E) \leq q(\lambda_{F,E}) \leq 0.
    \]
    From here we conclude in the same way.
\end{proof}

Now take $X \to \P^n$ a Tate--Shafarevich twist of $X' \to \P^n$, as above.
Moreover, let $i: Z \hookrightarrow X$ the embedding of a subvariety, and let $G$ be a rank $1$ sheaf on $X$ such that $i_*G$ is Cohen--Macaulay.
Then by \cref{prop:LocallyFreeGeneral} its image $E \coloneqq \Phi_{\sP}(i_*G)$ is a $\theta$-twisted locally free sheaf.

\begin{Thm}\label{thm:StabilitySecondExample}
    With the notation above, assume that $Z$ is integral and $i_*G$ is atomic. 
    Let $h'$ be an $a(E)$-suitable polarization on $X'$. 
    Then $E$ is a twisted atomic $\mu_{h'}$-stable vector bundle.
    In particular, it is projectively hyperholomorphic. 
\end{Thm}

\begin{proof}
    By \cref{prop:LocallyFreeGeneral} the restriction $E_t$ is the direct sum of (twisted) line bundles $L_{t,k}$ with the same degree, in particular polystable. 
    Assume that $E$ is not $\mu_{h'}$-stable. 
    Then \cref{prop:SemistabilityRestriction} produces a subsheaf $F \subset E$ with smaller rank and fiberwise destabilizing. 
    
    The only way for $F_t$ to destabilize is that if its is a subsum of the line bundles $L_{t,k}$. 
    By \cref{lem:PropertiesOfP} this implies that $\Phi_{\sP}^{-1}(F) \in \Db(X)$ is, generically over $\P^n$, a sheaf whose support is finite over $\P^n$ and contained in the surface $Z$.
    The surface $Z$ is integral, so the support agrees with $Z$, at least generically over the base. 
    Therefore $r(E) = r(F)$ which gives a contradiction.
    So $E$ is $\mu_{h'}$-stable, hence projectively hyperholomorphic by \cref{prop:TwistedAtomicHyperholomorphic}. 
\end{proof}

\subsection{The surfaces of lines}\label{sec:TheSurfacesOfLines}
Now we apply the results above to our case. 
Let $\cK^2_{6}$ be the moduli space of polarized \HK manifolds $(X,h)$ of degree $6$ and divisibility $2$. 
A general point corresponds to the variety of lines $F(Y)$ on a smooth cubic fourfold $Y $, equipped with the Pl\"ucker polarization. 
Fix a non-negative integer $d$, and let $\cN(d) \subset \cK^2_{6}$ be the closure of the locus parametrizing polarized \HK manifolds $(X,h)$ with Neron-Severi group
\begin{equation}\label{eq:NSOfFano}
\NS(X) = \Z h \oplus \Z f, \text{ and BBF form } q(h) =6, \ q(h,f) = 2d, \ q(f) = 0.
\end{equation}
A general $(X,h) \in \cN(d)$ has a unique (up to automorphisms of $\P^2$) Lagrangian fibration $\pi: X \to \P^2$ that satisfies $\pi^*\sO_{\P^2}(1) = \sO_X(f)$.

\begin{Prop}[{{\cite[Theorem 7.11]{markman14},\cite[Proposition B.4]{ogrady22}}}]\label{prop:OGradyTateShafarevich}
    There exists a polarized K3 surface $(S,H)$ of degree 2, such that $X \to \P^2$ is birational to a Tate--Shafarevich twist in $\Sha^0$ of the Beauville-Mukai system $ X' \coloneqq M_{(0,H,-1)}(S,H) \to \P^2$.
    Moreover, if $d \geq  35 $ and $ 3 \nmid d $ then $X$ is isomorphic to this Tate--Shafarevich twist.
\end{Prop} 

As $d \in \Z_{\geq 0}$ varies, the Noether--Lefschetz loci $\cN(d)$ form an infinite collection of divisors in $\cK^2_{6}$, so they cover a dense subset of $\cK^2_{6}$. 
All properties which are Zariski general continue to hold for a general element of $\cN(d)$ for infinitely many $d$.

In particular, we can find infinitely many $d$ such that a very general element in $\cN(d)$ is a variety of lines $X=F(Y)$ with the following properties:
\begin{itemize}
    \item The Neron--Severi group is as in \eqref{eq:NSOfFano}.
    \item For every hyperplane section $Y_H \subset Y$, its surface of lines $F(Y_H)$ is integral.
\end{itemize}
In this case, its dual $X'$ is a Beauville--Mukai system on a polarized K3 $(S,H)$ of Picard rank 1, so that every curve in $|H|$ is integral, and we can apply \cref{thm:TwistedPoincaré}.
In particular we obtain a Brauer class $\theta_d \in \Br(X')$.
The applications of all the results above gives the following. 

\begin{Cor}\label{cor:VBfromSurfaces}
    Let $X$ be as above, let $i: Z \hookrightarrow X$ be the surface of lines of any hyperplane section, and take $G \in \overline{\Pic^0}(Z)$. 
    Let $h'$ be a suitable polarization on $X'$.
    Then $E = \Phi(i_*G)$ is a $\mu_{h'}$-stable $\theta_d$-twisted atomic vector bundle of rank $5d^2$. In particular, it is projectively hyperholomorphic.
\end{Cor}

\begin{proof}
    By suitable polarization here we mean $a(E)$-suitable. 
    Since suitable polarizations depend only on the discriminant of $E$, they are independent of $Z$ and $G$.
    By our assumption on $X$, every such $Z$ is integral. 
    Moreover, by \cref{prop:CompactifiedPicCM} $G$ is Cohen--Macaulay on $Z$.
    Since $Z$ is Gorenstein, $i_*G$ is Cohen--Macaulay on $X$. 
    By \cite[Section 8.1.3]{beckmann22} the structure sheaf $i_*\sO_Z$ is atomic, so the same is true for $i_*G$ as the Mukai vector is the same. 
    By \cref{prop:FinintessOverBaseLagrangians} and \cref{thm:StabilitySecondExample} we deduce that $E$ is a $\theta_d$-twisted, atomic, $\mu_{h'}$-stable, locally free sheaf of rank $r$, where $\pi_*[Z] = r[\P^2]$. 

    It remains only to compute the rank $r$.
     By \cite[Lemma 7.1 and Example 7.4]{markman21}, the class of the surface is
    \[
         [Z] = \frac{5}{8}\left(h^2 - \frac{1}{5}c_2(X)\right) \in H^4(X,\Z).
    \]
    For a HK of type $\mathrm{K3}^{[2]}$, the second Chern class $c_2(X)$ is a multiple of dual of the BBF form $\sfq^{\vee}$, hence $ \int_X{c_2(X) \cup f^2 } = 0$, because $q(f) = 0$. 
    Moreover, by Fujiki's formula we have 
    \[
    \int_X{h^2\cup f^2} = 2q(h,f)^2 = 2\cdot (2d)^2
    \]
    Combining everything, we get
    \begin{equation*}
         \int_X{[Z] \cup f^2}  = \frac{5}{8}\int_X{h^2 \cup f^2} - \frac{1}{8}\int_X{c_2(X) \cup f^2} = 5d^2,
    \end{equation*}
    Projection formula implies that $r = 5d^2$.
\end{proof}

\section{Moduli spaces of hyperholomorphic bundles}\label{sec:ModuliSpacesGeneral}
In this section we study the geometry of moduli spaces of twisted projectively hyperholomorphic vector bundles on HK manifolds. 
We recall the construction of the symplectic forms, and prove various results aimed at establishing smoothness.

\subsection{Trace map for twisted sheaves}\label{sec:Trace}
We aim to extend the construction of the symplectic form in \cite[Section 7]{bottini22} to the twisted case.
For simplicity, we focus on projectively hyperholomorphic vector bundles.  
The first step is to extend the trace map to this setting.
For a good account in the untwisted case see \cite[Section 1.4]{kuznetsovMark09} and \cite[Section 10]{huybrechts_lehn_2010}, or \cite[Section 2.3]{caldararu03} for a more categorical approach.

Let $X$ be a smooth projective variety and $\alpha \in \Br(X)$ a Brauer class. 
Let $E$ be an $\alpha$-twisted locally free sheaf, its endomorphisms bundle $\cA \coloneqq \cEnd(E) $ is an Azumaya algebra representing $\alpha$.
Since $\cEnd(E) = E^{\vee} \otimes E$, there is a natural isomorphisms of $\sO_X$-modules $\cA\cong \cA^{\vee}$.
The trace map $\Tr: \cA \to \sO_X$ is the dualization of the natural section $\sO_X \to \cA$.
On the opens where $E$ is untwisted, this is the usual trace map. 

For an untwisted vector bundle $F$, tensoring induces a trace map $\Tr_F: \cA \otimes F \to F$.
This induces maps in cohomology:
\[
    \Tr_F: \Ext^k(E,E \otimes F) \to H^k(X,F).
\]
If $X$ has dimension $d$ and $\omega_X$ is the canonical bundle, under the Serre duality isomorphism 
\[
    \Ext^0(E,E) \cong \Ext^d(E,E \otimes \omega_X)^{\vee}
\]
the identity corresponds to the composition 
\[
    \Tr_{\omega_X}: \Ext^d(E,E \otimes \omega_X) \to H^d(X,\omega_X) =  \C.
\]
Therefore, the Serre duality pairing can be written as 
\[
    \Ext^k(E,E) \otimes \Ext^{d-k}(E,E \otimes \omega_X) \to H^d(X,\omega_X) = \C \quad (a,b) \mapsto \Tr_{\omega_X}(b \circ a).
\]

As in the untwisted case, it follows from the definition that the trace is linear: for $f \in H^k(X,F)$ and $g \in \Ext^l(E,E)$ we have 
\begin{equation}\label{eq:LinearityOfTrace}
    \Tr_F((f \otimes \id_E) \circ g) = f \circ \Tr_{\sO_X}(g).
\end{equation}

\begin{Rem}\label{rem:TraceKillsBrackets}
    The composition $\cA \times \cA \to \cA$ induces, when taking cohomology, the Yoneda pairing $\circ$ on $\Ext^*(E,E)$. 
    Therefore, as in the untwisted case \cite[Lemma 10.1.3]{huybrechts_lehn_2010}, we have $\Tr([a,b]) = 0$ for $a,b \in \Ext^*(E,E)$, where
    \[
        [a,b] = a \circ b - (-1)^{|a||b|}b \circ a
    \]
    is the graded bracket. 
\end{Rem}

As usual, everything can be extended to complexes in $\Db(X,\alpha)$ taking resolutions, but, for our purposes, the case of a locally free sheaf suffices.

\subsection{Symplectic form}\label{sec:SymplecticForm}
Now we take $X$ to be a projective HK manifold of dimension $2n$, and we choose a generator $\sigma \in H^0(X,\Omega^2_X)$.
In the following we tacitly identify $H^{2k}(X,\sO_X) \cong \C$ via $\overline{\sigma}^{2k}$ for every $k$. 
Consider the pairing 
\begin{equation}\label{eq:SympForm}
        \tau: \Ext^1(E,E) \times \Ext^1(E,E) \to H^2(X,\sO_X) \cong \C, \ (a,b) \mapsto \Tr( a \circ b),
\end{equation}
where we are writing $\Tr$ for $\Tr_{\sO_X}$.
While we cannot expect it to be non-degenerate in general, for projectively hyperholomorphic bundles it is. 

\begin{Lem}\label{lem:SymplecticityOfExt1}
    If $E$ a $\mu$-stable, projectively hyperholomorphic twisted bundle, then $\tau$ is symplectic.
    In particular, $\Ext^1(E,E)$ has even dimension. 
\end{Lem}

\begin{proof}
From \cref{rem:TraceKillsBrackets} it follows that $\tau$ is skew-symmetric.
The linearity of the trace \eqref{eq:LinearityOfTrace} gives 
\begin{equation}\label{eq:LinearitySympForm}
    \overline{\sigma}^{n-1} \circ \Tr(a \circ b) = \Tr(L_{\overline{\sigma}}^{n-1} \circ a \circ b).
\end{equation}
where $L_{\overline{\sigma}} \coloneqq \overline{\sigma} \otimes \id_E \in \Ext^2(E,E)$.
It suffices to see that the latter is non-degenerate, as 
\[
    \overline{\sigma}^{n-1} \circ - : H^{2}(X,\sO_X) \to H^{2n}(X,\sO_X)
\]
is an isomorphism.
Since $\cA= \cEnd(E)$ is $\mu$-polystable and hyperholomorphic, this follows from \cref{thm:HardLefschetzVerbitsky} below (applied to $\cA$) combined with Serre duality.
\end{proof}

\begin{Thm}[{{\cite[Theorem 4.2A]{verbitsky96b}}}]\label{thm:HardLefschetzVerbitsky}
    Let $X$ be a HK manifold of dimension $2n$, and $F$ a hyperholomorphic vector bundle on $X$.
    Then for every $k \leq n$ the map
    \[
    L_{\overline{\sigma}}^{n-k} \circ - : H^k(X,F) \to H^{2n-k}(X,F)
    \]
    is an isomorphism.
\end{Thm}

\begin{Rem}\label{rem:SymplecticFormUntwisted}
    The equality \eqref{eq:LinearitySympForm} is the same as \cite[Lemma 7.4]{bottini22}, using that for $\overline{\sigma}$ only the degree zero term of the exponential of the Atiyah class comes into play.
\end{Rem}

\subsection{Relative ext sheaves}
Let $S$ be a reduced scheme, and $\sE$ a twisted $S$-flat sheaf of positive rank on $X \times S$. 
Let $\pi : X \times S \to S$ be the second projection, and consider the sheaves
\begin{equation}\label{eq:RelativeExt}
    \cExt_{\pi}^k(\sE,\sE) \coloneqq \cH^k(\sfR\pi_{*} \circ \RcHom(\sE,\sE)) \in \Coh(S).
\end{equation}
If $\sE$ is locally free, which is the case we are interested in, this is also equal to $\sfR^k\pi_{*}(\cEnd(\sE))$.

For any closed point $s \in S$, there is a base change map
\begin{equation}\label{eq:BaseChangeMap}
    \phi_k(s) :  \cExt_{\pi}^k(\sE,\sE)|_{s} \to \Ext^k(\sE_s,\sE_s),
\end{equation}
where $\sE_s \coloneqq \sE|_{X \times \{s\}}$, see \cite[Section 12]{hartshorne77} and \cite{lange83}.
Notoriously, it fails to be an isomorphism in general, and bijectivity at the $k$-th level is related to the local freeness of $\cExt_{\pi}^{k-1}(\sE,\sE)$.
This is captured in the cohomology and base change theorem, \cite[Theorem 12.11]{hartshorne77}.
On the direct sum $\bigoplus_k \cExt^k_{\pi}(\sE,\sE)$ there is a natural product which makes it into a coherent $\sO_S$-algebra, which we denote by $ - \circ -$.
This product is compatible with the usual Yoneda product on $\Ext^*(\sE_s,\sE_s)$ via the base change maps. 

The second projection $\pi: X \times S \to S$ is a smooth projective morphism with relative dualizing complex given by $\sO_{X \times S}[2n]$. 
Therefore, we can apply Grothendieck--Verdier duality to $\RcHom(\sE,\sE)$, which gives an isomorphism 
\[
    \sfR\pi_*(\RcHom(\sE,\sE))[2n] \isomor (\sfR\pi_*(\RcHom(\sE,\sE)))^{\vee},
\]
where we have used that $\RcHom(\sE,\sE) \cong \RcHom(\sE,\sE)^{\vee}$. 
Combined with the Grothendieck spectral sequence for the composition of derived functors, we get a spectral sequence
\begin{equation}\label{eq:VerdierSpectralSequence}
    E^{p,q}_2 = \cExt^p_{S}(\cExt_{\pi}^{-q}(\sE,\sE), \sO_{S}) \implies \cExt_{\pi}^{p+q+2n}(\sE,\sE),
\end{equation}
whose degeneration yields morphisms
\begin{equation}\label{eq:RelativeSerreHomVersion}
    \cExt_{\pi}^i(\sE,\sE) \to \cExt^{2n-i}_{\pi}(\sE,\sE)^*.
\end{equation}
By the usual tensor-hom adjuction, they can be thought as a relative Serre duality pairing:
\begin{equation}\label{eq:RelativeSerreTensorPairing}
    \cExt_{\pi}^i(\sE,\sE) \otimes \cExt^{2n-i}_{\pi}(\sE,\sE) \to \sO_S.   
\end{equation}
In particular, for $i = 2n$,  there is a map 
\[
    \cExt^{2n}_{\pi}(\sE,\sE) \to \cExt^{0}_{\pi}(\sE,\sE)^{*}. 
\]
Dualizing the natural inclusion $\sO_S \to \cExt^{0}_{\pi}(\sE,\sE)$, and composing gives the (relative) trace map
\[
    \Tr: \cExt^{2n}_{\pi}(\sE,\sE) \to \sO_S.
\]

\begin{Rem}\label{rem:TraceMapsHK}
If $X$ is HK and $\sE$ is locally free, we can also construct a trace map for any even $k$. 
Consider the trace map $\cEnd(\sE) \to \sO_{X \times S}$, and apply $R^k\pi_*(-)$ to get 
\[
    R^k\pi_*(\cEnd(\sE)) \to R^k\pi_*(\sO_{X \times S}) = H^k(X,\sO_X) \otimes \sO_S \cong \sO_S.
\]
If $k = 2n$, it agrees with the above one up to constants.

\end{Rem}

\begin{Lem}\label{lem:FirstandLastExt}
    Assume that for every closed point $s \in S$, the twisted sheaf $\sE_{s}$ is a simple locally free sheaf. 
    Then, the base change maps
    \[
     \varphi_k(s) : \cExt_{\pi}^k(\sE,\sE)|_{s} \to \Ext^k(\sE_s,\sE_s)
    \]
    are isomorphisms for every $s \in S$, and for $k \in \{ 0,2n-1,2n \}$.   
    Moreover, the morphisms
    \[
    \sO_S \isomor \cExt_{\pi}^0(\sE,\sE) \text{ and }  \Tr: \cExt_{\pi}^{2n}(\sE,\sE) \isomor \sO_S.
    \]
    are isomorphisms.
\end{Lem}

\begin{proof}   
    Under these assumptions, the sheaf $\sE$ itself is locally free. 
    By Serre duality and the simpleness of $\cE_s$ we have
    \[
    \dim \Ext^0(\sE_s,\sE_s) = \dim \Ext^{2n}(\sE_s,\sE_s) = 1
    \]
    Grauert's Theorem \cite[Corollary 12.9]{hartshorne77}, applied to the bundle $\cEnd(\sE)$ and the second projection $\pi : X \times S \to S$, implies that $\cExt_{\pi}^k(\sE,\sE)$ is locally free and the base change map is an isomorphism for $k \in \{ 0,2n \}$.
    Then \cite[Corollary 12.11(b)]{hartshorne77} applied to $\varphi_{2n}(s)$ implies that $\varphi_{2n-1}(s)$ is an isomorphism.
    
    The natural inclusion $\sO_{S} \to \cExt_{\pi}^0(\sE,\sE)$ is compatible with base change.
    Since it is an isomorphism for every fiber, it is an isomorphism at the level of sheaves. 
    For $k = 2n$, we consider the following commutative diagram 
\[\begin{tikzcd}
	{\cExt^{2n}(\sE,\sE)|_s} & {\cExt^{0}(\sE,\sE)^{*}|_s} & {\sO_{S}|_{s}} \\
	{\Ext^{2n}(\sE_s,\sE_s)} & {\Ext^0(\sE_s,\sE_s)^{*}} & \C
	\arrow[from=1-1, to=1-2]
	\arrow["{\phi_{2n}(s)}", from=1-1, to=2-1]
	\arrow["\cong", from=1-2, to=1-3]
	\arrow["\cong", from=1-3, to=2-3]
	\arrow["\cong", from=2-1, to=2-2]
	\arrow["{\phi_0(s)^{*}}"', from=2-2, to=1-2]
	\arrow["\cong", from=2-2, to=2-3]
\end{tikzcd}\]
    which shows that the Trace map is compatible with base change. 
    Again, since it is an isomorphism fiberwise, so it is an isomorphism for the sheaf. 
\end{proof}

\begin{Prop}\label{prop:SmoothLocusSymplectic}
    Let $S$ be an open subset of the smooth locus of a moduli space $M$ of semistable sheaves on $X$. 
    Assume it parametrizes only twisted stable projectively hyperholomorphic vector bundles.
    Then there is an isomorphism $\cExt^1_{\pi}(\sE,\sE) \cong T_S$.
    Furthermore, the pairing 
    \begin{equation}\label{eq:RelativeSymplecticForm}
        \cExt^1_{\pi}(\sE,\sE) \otimes \cExt^1_{\pi}(\sE,\sE) \to H^2(X,\sO_X) \otimes \sO_S \cong \sO_S, \ (a,b) \mapsto \Tr(a \circ b)
    \end{equation}
    is a non-degenerate 2-form on $S$.
\end{Prop}

\begin{proof}
    Under the assumptions, the dimension of the tangent space $\Ext^1(E,E)$ is constant. 
    Therefore, the sheaf $\cExt^1_{\pi}(\sE,\sE)$ is locally free and commutes with base change by Grauert's Theorem.
    By the main result of \cite{lehn98}, the sheaf $\cExt^{2n-1}_{\pi}(\sE,\sE)$ is the cotangent sheaf to the stable locus of the moduli space $M$.
    This is also locally free and commutes with base change for the same reason, therefore Serre duality implies 
    \[
        \cExt^1_{\pi}(\sE,\sE) \cong \cExt^{2n-1}_{\pi}(\sE,\sE)^{*} \cong T_{S}.
    \]
    Since both the Yoneda pairing and the trace map commute with base change, we conclude thanks to \cref{lem:SymplecticityOfExt1}.
\end{proof}

We conclude the section by globalizing the morphism of \cref{thm:HardLefschetzVerbitsky}, for later use.
The Leray spectral sequence relative to the morphism $\pi : X \times S \to S$ for the hypercohomology of the complex $\RcHom$, gives a global-to-local map 
\[
    \Ext^k_{X \times S}(\sE,\sE) \to H^0(S,\cExt_{\pi}^k(\sE,\sE) ).
\]
As in the absolute case, we pull back $\overline{\sigma}$ to get a morphism $ \sO_{X \times S} \to \sO_{X \times S}[2]$.
In turn, this gives an element of $\Ext^2_{X \times S}(\sE,\sE)$, which gives a section  $\cL_{\overline{\sigma}} \in H^0(S,\cExt_{\pi}^2(\sE,\sE) )$.
This is compatible with base change in the sense that $ \phi_2(s)(\cL_{\overline{\sigma}}) = L_{\overline{\sigma}} \in \Ext^2(\sE_s,\sE_s)$, as one easily sees by unraveling the definitions.

\subsection{A particular result}\label{sec:ParticularResult}
Now we put ourselves under some very restrictive assumptions, which, as we later prove, are satisfied in the case we consider. 
Let $X$ be a HK fourfold, and let $M^{\circ}$ be a projective irreducible component of a moduli space $M = M_{\vv}(X,h)$ of possibly twisted semistable sheaves.
We assume that:
\begin{enumerate}
    \item[(A)] Every sheaf in $M^{\circ}$ is a slope stable, projectively hyperholomorphic bundle.
    \item[(B)] The open locus $M^{\circ} \cap M_{\rm{sm}}$ of smooth points of $M$ has complement of codimension at least $3$ in $M^{\circ}$.
    \item[(C)] Taken with its reduced structure, $M^{\circ}$ is smooth. 
    \item[(D)] There is a non-empty open subset $U \subset M^{\circ} \cap M_{\rm{sm}}$ such that for every $E \in U$ the Yoneda pairing \\ $\bwed^2 \Ext^1(E,E) \to \Ext^2(E,E)$ is an isomorphism.
\end{enumerate}
These conditions mean that, along $M^{\circ}$, the moduli space $M$ may have other components or embedded points, but they have to meet $M^{\circ}$ in a closed subset of codimension at least $3$.
We prove that this does not happen, and conclude that $M^{\circ}$ is a smooth connected component of $M$.

\begin{Rem}\label{rem:SymplecticityOfM}
    Under these assumptions, by \cref{prop:SmoothLocusSymplectic} the intersection $M^{\circ} \cap M_{\rm{sm}}$ is symplectic. 
    Since its complement has high codimension in $M^{\circ}$, the symplectic form extends to the whole component.
    In particular, the canonical bundle $\omega_{M^{\circ}}$ is trivial.
\end{Rem}

Endow $M^{\circ}$ of its reduced structure; it is a closed subscheme of the moduli space $M$. 
We denote by $\sE$ the restriction to $X \times M^{\circ}$ of the twisted universal family.
It is $M^{\circ}$-flat, and for every $E \in M^{\circ}$ we have $\sE|_E \cong E$. 
In particular, $\sE$ is a locally free twisted sheaf on $X \times M^{\circ}$.

Denoting $\pi : X \times M^{\circ} \to M^{\circ}$ the second projection, the spectral sequence \eqref{eq:VerdierSpectralSequence} becomes
\begin{equation*}
    E^{p,q}_2 = \cExt^p_{M^{\circ}}(\cExt_{\pi}^{-q}(\sE,\sE), \sO_{M^{\circ}}) \implies \cExt_{\pi}^{p+q+4}(\sE,\sE).
\end{equation*}
Since $\pi$ is smooth of relative dimension $4$, the terms $E^{p,q}_2$ vanish unless $0 \leq -q \leq 4$.
Also clearly $E^{p,q}_2$ vanishes for $p < 0$, so that the spectral sequence is concentrated in the fourth quadrant.
To avoid confusion, in the following we denote by $-^{\vee} = \RcHom(-,\sO)$ the derived dual, and $-^* = \cHom(-,\sO)$ the underived dual. 

\begin{Lem}\label{lem:ComputingCurlyExt1}
    There are isomorphisms
    \[
    \cExt^1_{\pi}(\sE,\sE) \cong \cExt^3_{\pi}(\sE,\sE)^{*} \cong T_{M^{\circ}}.
    \]
    In particular, $E_2^{p,-1} = 0$ for all $p > 0$.
\end{Lem}

\begin{proof}
    Consider the spectral sequence for $p + q = -3$. 
    The terms of interest are:
    \begin{itemize}
         \item $E^{0,-3}_2 =  \cHom(\cExt^{3}_{\pi}(\sE,\sE),\sO_{M^{\circ}}) \to E^{2,-4}_2 = \cExt^2(\cExt^{4}_{\pi}(\sE,\sE),\sO_{M^{\circ}})$,
         \item $E^{1,-4} = \cExt^1(\cExt^{4}_{\pi}(\sE,\sE),\sO_{M^{\circ}})$. 
    \end{itemize}
    The terms $E^{1,-4}_2$ and $E^{2,-4}$ vanish by \cref{lem:FirstandLastExt}.
    So, at this level the spectral sequence degenerates at the second page giving 
    \[
    \cExt^1_{\pi}(\sE,\sE) \cong \cExt^3_{\pi}(\sE,\sE)^*.
    \]
    It follows that $\cExt^1_{\pi}(\sE,\sE)$ is a reflexive sheaf on the smooth variety $M^{\circ}$.
    
    \cref{prop:SmoothLocusSymplectic} gives an isomorphism
    \[
        \cExt^1_{\pi}(\sE,\sE) \cong T_{M^{\circ}}
    \]
    over the open locus $M^{\circ} \cap M_{\rm{sm}}$.
    It extends to $M^{\circ}$, because both sheaves are reflexive on the smooth variety $M^{\circ}$, and $M^{\circ} \cap M_{\rm{sm}}$ has complement of high codimension. 
    Since $M^{\circ}$ is smooth, in particular the sheaf $\cExt^1_{\pi}(\sE,\sE)$ is locally free. Therefore $E_2^{p,-1} = \cExt^p(\cExt^1_{\pi}(\sE,\sE),\sO_{M^{\circ}}) = 0$ for all $p >0 $.
\end{proof}

Now we concentrate on the row $ q = - 3$.
The following crucial result is the reason we ask for codimension $3$ instead of $2$ in assumption (B).

\begin{Lem}\label{lem:TorsionInExt3}
    Let $\sT \coloneqq T(\cExt_{\pi}^3(\sE,\sE))$ be the torsion subsheaf of $\cExt_{\pi}^3(\sE,\sE)$.
    There is an isomorphism
    \[
    \cExt_{\pi}^3(\sE,\sE) \cong \cExt^1_{\pi}(\sE,\sE) \oplus \sT,
    \]
    and $\sT$ is supported away from $M^{\circ} \cap M_{\rm{sm}}$. 
    In particular 
    \begin{equation*}
        E_2^{p,-3} = \cExt^p(\cExt^3_{\pi}(\sE,\sE),\sO_{M^{\circ}}) = \cExt^p(\sT,\sO_{M^{\circ}})
    \end{equation*}
    vanishes for $p = 1,2$.
\end{Lem}

\begin{proof}
    Consider the section $\cL_{\sigma} \in H^0(M^{\circ},\cExt^2_{\pi}(\sE,\sE))$ constructed in section \cref{sec:SymplecticForm}.
    Cup product with it gives a morphism
    \begin{equation}\label{eq:InjectionExt1to3}
            \mu : \cExt^1_{\pi}(\sE,\sE) \to \cExt^3_{\pi}(\sE,\sE).
    \end{equation}
    On the other hand, by relative Serre duality we have a morphism $\cExt^3_{\pi}(\sE,\sE) \to \cExt^1_{\pi}(\sE,\sE)^{*}$.
    The composition $\cExt^1_{\pi}(\sE,\sE) \to \cExt_{\pi}^1(\sE,\sE)^{*}$ is an isomorphism on the smooth locus $M^{\circ} \cap M_{\rm{sm}}$.
    Indeed, by Grauert's Theorem, on $M^{\circ} \cap M_{\rm{sm}}$ the sheaf $\cExt^1_{\pi}(\sE,\sE)$ is locally free and commutes with base change, and on the fibers it is an isomorphism by Serre duality and \cref{thm:HardLefschetzVerbitsky}. 
    Since $M^{\circ} \cap M_{\rm{sm}}$ has complement of high codimension in $M^{\circ}$, this extends to an isomorphism 
    \[
    \cExt^1_{\pi}(\sE,\sE) \isomor \cExt_{\pi}^1(\sE,\sE)^{*}
    \]
    over $M^{\circ}$.
    It follows that $\mu$ is a split injection of a torsion-free sheaf into a coherent sheaf of the same rank, which means that the cokernel is the torsion part $\sT$. 
    This is supported away from $M^{\circ} \cap M_{\rm{sm}}$ as $\cExt^3_{\pi}(\sE,\sE)$ is locally free there.  
    The vanishing follows from \cite[Proposition 1.1.6]{huybrechts_lehn_2010}, as $\sT$ is supported in codimension at least $3$.
\end{proof}

\begin{Lem}\label{lem:ComputingCurlyExt2}
    There is an isomorphism
    \[
    \bwed^2 \cExt^1_{\pi}(\sE,\sE) \isomor \cExt^2_{\pi}(\sE,\sE).
    \]
    In particular $E^{p,-2}_2 = 0$ for all $p > 0$.
\end{Lem}

\begin{proof}
From \cref{lem:FirstandLastExt} and \cref{lem:TorsionInExt3} respectively, we deduce that $E_2^{2,-4} = E^{1,-3}_2 = 0$.
Therefore the only interesting term of the spectral sequence for  $p + q = -2$ is 
\[
    d_2^{0,-2}: E_2^{0,-2} =  \cHom(\cExt^2_{\pi}(\sE,\sE),\sO_{M^{\circ}}) \to E_2^{2,-3} =\cExt^2(\cExt^3_{\pi}(\sE,\sE),\sO_{M^{\circ}}).
\] 
\cref{lem:TorsionInExt3} implies that $E^{2,-3}_2 = 0$.
On page $3$, the differential lands in $E^{3,-4}_{3}$ which vanishes by \cref{lem:FirstandLastExt}. 
So the spectral sequence degenerates, giving an isomorphism
\begin{equation}\label{eq:CurlyExt2SelfDual}
    \cExt^2_{\pi}(\sE,\sE)^* \cong \cExt^2_{\pi}(\sE,\sE).
\end{equation}
In particular, $\cExt^2_{\pi}(\sE,\sE)$ is reflexive. 
On the smooth locus $M^{\circ} \cap M_{\rm{sm}}$, the Yoneda product is skew-symmetric and gives a morphism of bundles
\[
\bwed^2 \cExt^1_{\pi}(\sE,\sE) \to \cExt^2_{\pi}(\sE,\sE).
\]
The source is locally free and the target is reflexive on $M^{\circ}$, therefore this map extends.\footnote{At this point, we are not claiming that this extension agrees with the Yoneda product away from the smooth locus, just that some extension exists, although a posteriori this will be true.}
Over the open $U$ of assumption (D), this map is an isomorphism, so we get a short exact sequence 
\begin{equation}\label{eq:ShortExactSequenceYoneda}
0 \to \bwed^2 \cExt^1_{\pi}(\sE,\sE) \to \cExt^2_{\pi}(\sE,\sE) \to \mathscr{Q} \to 0,
\end{equation}
where $\mathscr{Q}$ is a torsion sheaf.
By reflexivity, the quotient $\mathscr{Q}$ is either zero, or every component of its support has codimension $1$. 
On the other hand, by \cref{lem:ComputingCurlyExt1} and \cref{rem:SymplecticityOfM} we have
\[
\det \cExt^1_{\pi}(\sE,\sE) \cong \det T_{M^{\circ}} \cong \sO_{M^{\circ}}.
\]
Similarly, $\cExt^2_{\pi}(\sE,\sE)$ is reflexive and isomorphic to its dual by \eqref{eq:CurlyExt2SelfDual}, therefore its determinant is $2$-torsion. 
Taking determinants in the sequence \eqref{eq:ShortExactSequenceYoneda}, we get a non-zero section of $\det \cExt^2_{\pi}(\sE,\sE)$, which implies $\det \cExt^2_{\pi}(\sE,\sE) = \sO_{M^{\circ}}$.
We conclude that $\det \mathscr{Q} = \sO_{M^{\circ}}$, which implies $\mathscr{Q} = 0$.
The vanishings $E_2^{p,-2} = 0$ for $p > 0$ follow from the local freeness of $\cExt^2_{\pi}(\sE,\sE)$.
\end{proof}

\begin{Thm}\label{thm:SmoothnessOurCase}
    Assume that the assumptions (A),..., (D) are satisfied.
    Then $M^{\circ}$ is a connected component and it is contained in the smooth locus $M_{\rm{sm}}$.
    If $\sE$ is the restriction to $X \times M^{\circ}$ of a twisted universal family, then the sheaves $\cExt^k_{\pi}(\sE,\sE)$ are locally free for every $k$, and $\cExt^1_{\pi}(\sE,\sE) = T_{M^{\circ}}$.
    There are isomorphisms
    \[
    \cExt^3_{\pi}(\sE,\sE) \cong \cExt^1_{\pi}(\sE,\sE)^{*} \text{ and } \bwed^2 \cExt^1_{\pi}(\sE,\sE) \isomor \cExt^2_{\pi}(\sE,\sE)
    \]
    induced by Serre duality and the Yoneda pairing respectively.
    Moreover, via the last isomorphism the trace map $\Tr : \cExt^2_{\pi}(\sE,\sE) \to \sO_{M^{\circ}}$ is the symplectic form $\tau$ on $M^{\circ}$.   
\end{Thm}

\begin{proof}
    We look at the spectral sequence for $p + q = -1$. 
    By the above computations, the only non-zero term is $E_2^{0,-1}=\cExt^1_{\pi}(\sE,\sE)^*$.
    Its differential in page $2$ vanishes, in page $3$ it might not; nevertheless the spectral sequence gives
    \[
        \cExt^3_{\pi}(\sE,\sE) \subseteq \cExt_{\pi}^1(\sE,\sE)^*.
    \]
    In particular, $\sT = 0$.
    The proof of \cref{lem:TorsionInExt3} implies that both
    \[
        \mu : \cExt^1_{\pi}(\sE,\sE) \to \cExt^3_{\pi}(\sE,\sE) \text{ and } \cExt^3_{\pi}(\sE,\sE) \to \cExt^1_{\pi}(\sE,\sE)^*
    \]
    are isomorphisms. 
    \cref{lem:FirstandLastExt} implies that $ \cExt^3_{\pi}(\sE,\sE)$ commutes with base change, so
    the dimension $\ext^3(E,E)$ is constant on $M^{\circ}$.
    By Serre duality, the same is true for $\ext^1(E,E)$.
    The local description of $M$ (via the Kuranishi map) implies that if $\ext^1(E,E)$ is constant, then $M^{\circ}$ consists of smooth points for $M$, in particular it is a connected component. 
    The rest of the statement follows from \cref{prop:SmoothLocusSymplectic} and the \cref{lem:ComputingCurlyExt2}.
\end{proof}

\subsection{A general result}
The same arguments allow to prove a more general smoothness criterion, which essentially requires the cotangent sheaf to be torsion-free.

\begin{Thm}\label{thm:SmoothnessGeneral} 
    Let $X$ be a projective \HK manifold of dimension $2n$.
    Let $M$ be an open locus of a moduli space of possibly twisted, Gieseker semistable sheaves.
    Assume that $M$ is normal with lci singularities, and that it parametrizes only slope stable, projectively hyperholomorphic bundles.
    Then $M$ is smooth.
\end{Thm}

\begin{proof}
    Let $\sE$ be a twisted universal family over $X \times M$, and let $\pi : X \times  M \to  M$ be the second projection. 
    By \cref{lem:FirstandLastExt} we get that $\cExt_{\pi}^{2n}(\sE,\sE) \cong \sO_{M}$, therefore the spectral sequence \eqref{eq:VerdierSpectralSequence} degenerates for $p + q = -(2n - 1)$.
    So we have
    \[
        \cExt^1_{\pi}(\sE,\sE) \cong \cExt^{2n-1}_{\pi}(\sE,\sE)^*.
    \]
    In particular $ \cExt^1_{\pi}(\sE,\sE)$ is reflexive. 
    Since $M$ is lci and regular in codimension $1$ it is reduced. 
    A reflexive sheaf on a normal scheme is $S_2$, so we can extend morphisms defined in codimension $2$. 

    Similarly to \cref{lem:TorsionInExt3} we let $\mu: \cExt^1_{\pi}(\sE,\sE) \to \cExt^{2n-1}_{\pi}(\sE,\sE)$ be the product with $\cL_{\sigma}^{n-1} \in H^0(M,\cExt^2_{\pi}(\sE,\sE))$.
    The same argument applied in this case gives a decomposition
    \[
        \cExt_{\pi}^{2n - 1}(\sE,\sE) \cong \cExt^1_{\pi}(\sE,\sE) \oplus \sT,
    \]
    where $\sT = T(\cExt_{\pi}^{2n - 1}(\sE,\sE))$.
    The sheaf $\cExt_{\pi}^{2n - 1}(\sE,\sE)$ is the cotangent sheaf to $M$ by \cite{lehn98}, as it parametrizes only stable sheaves.
    Under our assumptions, the cotangent sheaf to $M$ is torsion-free by \cite[Proposition 9.7]{kunz86}. 
    Therefore, $\sT = 0$ and $\mu :\cExt^1_{\pi}(\sE,\sE) \isomor \cExt^{2n-1}_{\pi}(\sE,\sE) $ is an isomorphism. 
    Notice that $\mu$ commutes with base change, since it arises as product with a section of $\cExt^2_{\pi}(\sE,\sE)$.
    Thus for every $E \in M$ we have the following commutative square 
\[\begin{tikzcd}
	{\cExt^1_{\pi}(\sE,\sE)|_E } & {\cExt^{2n-1}_{\pi}(\sE,\sE)|_E } \\
	{\Ext^1(E,E)} & {\Ext^{2n-1}(E,E)}
	\arrow["\cong", from=1-1, to=1-2]
	\arrow["{\varphi_{1}(E)}", from=1-1, to=2-1]
	\arrow["{\varphi_{2n-1}(E)}", from=1-2, to=2-2]
	\arrow["\cong", from=2-1, to=2-2]
\end{tikzcd}\]
    The top horizontal map is an isomorphism by the argument above, while the bottom one by \cref{thm:HardLefschetzVerbitsky}.
    The base change map $\varphi_{2n-1}(E)$ is an isomorphism by \cref{lem:FirstandLastExt}.
    Hence, the base change map $\varphi_{1}(E)$ is also an isomorphism. 
    Therefore, by \cite[Theorem 12.11(b)]{hartshorne77} the sheaf $\cExt^1_{\pi}(\sE,\sE)$ is locally free, because $\phi_0(E)$ is an isomorphism by Grauert's Theorem. 
    This implies that $\cExt^{2n-1}_{\pi}(\sE,\sE) \cong \cExt^{1}_{\pi}(\sE,\sE) $ is locally free, which means that $M$ is smooth. 
\end{proof}

\section{Modular families of O'Grady's tenfolds}\label{sec:CompactificationSecond}
\subsection{Relative Picard, Lagrangians, and vector bundles}\label{sec:PicLagrangiansBundles}

Let $X$ be the variety of lines of a general cubic fourfold, and $h$ be the Pl\"ucker polarization.
In \cref{sec:compactifiedPic} we studied the relative compactified Picard scheme $\overline{\Pic^{0}}({\cF/\P^5})$ of the universal surface $\cF \subset X \times \P^5$.
Now we show that, with its reduced structure, it embeds as an irreducible component of a moduli space of semistable sheaves on $X$. 
Proving that it is in fact a connected component amounts to showing, among other things, that a sheaf supported on a surface of lines $Z$ cannot be a degeneration of sheaves supported on surfaces which are not deformation equivalent to $Z$. 
This is not an easy task to do directly, but a posteriori will be true. 

For the next couple of results, it is useful to keep in mind the stack-theoretic point of view for the construction of the moduli spaces. 
This is because we want to compare a relative moduli space with an absolute moduli space, neither of which is fine a priori, and working with moduli functors becomes unnecessarily convoluted.

\begin{itemize}
    \item If $\vv \in H^*(X,\Q)$ is any class, the moduli stack $\sM_{h}$ consists (as a fibered category over $\mathrm{Sch/\C}$) of pairs $(T,\sF)$, where $T$ is a scheme over $\C$ and $\sF$ is a flat family of semistable sheaves over $T \times X$. 
    Over the locus of stable sheaves it is a $\mathbb{G}_m$-gerbe over its good moduli space $\cM_{h}$.
    It is a disjoint union of the stacks $\sM_{h}(\vv)$ parametrizing sheaves with fixed Mukai vector $\vv$.
    
    \item The relative moduli space of torsion-free sheaves $\srPic{\cF}{\P^5}$ consists of triples $(T,f,G)$ where $T$ is a scheme over $\C$, $f : T \to \P^5$ is a morphism of $\C$-schemes, and $\sG$ is a flat family of torsion-free sheaves of rank $1$ on $T \times_{\P^5} \cF$. 
    Every sheaf is simple, therefore it is a $\G_m$-gerbe over its coarse moduli space $\Pic^{=}({\cF/\P^5})$. 
    It is a disjoint union of connected components and we denote by $\srPico{\cF}{\P^5}$ the connected component containing the structure sheaves.
\end{itemize}

The following result essentially shows the closure of the relative Picard inside the relative moduli space of torsion-free sheaves agrees with the closure inside the moduli space of stable sheaves on $X$. 

\begin{Lem}\label{lem:RelativePicToLagrangians}
    Let $X$ be the Fano variety of a general cubic fourfold, $\lambda$ any polarization, and $\cF \subset X \times \P^5$ the universal surface of lines. 
    Let $M_{\vv_0}(X,\lambda)$ be the moduli space of $h$-semistable sheaves with Mukai vector equal to $v(\sO_Z)$.
    There is a closed embedding 
    \begin{equation}\label{eq:PicInLagrangians}
    \overline{\Pic^0}(\cF/\P^5) \hookrightarrow M_{\vv_0}(X,\lambda), \ G \mapsto i_*G
    \end{equation}
    whose image is an irreducible component $M^{\circ}_{\vv_0}(X,\lambda)$. 
    Considered with its reduced structure, $M^{\circ}_{\vv_0}(X,\lambda)$ is smooth, isomorphic to $J_Y$ and parametrizes only Cohen--Macaulay sheaves on $X$.
\end{Lem}

\begin{proof}
    Consider the morphism of stacks
    \begin{align*}
        i_*: \srPico{\cF}{\P^5} & \to \sM_{\vv_0}(X,\lambda) \\
        (T,f,\sG) &\mapsto (T,i_*\sG),
    \end{align*}
     where $i :T \times_B \cF \hookrightarrow T \times X$ is the closed embedding obtained as the base change of $\cF \subset \P^5 \times X $.
     We want to see that it is well-defined and a monomorphism. 
    
    To see that $i_*\sG$ is flat, it suffices to see that the functor $p^*_T(-) \otimes i_*\sG$ is exact.
    This is easily seen using projection formula and the flatness of $\sG$. 
    Every fiber $\sG_t$ is Gieseker $\lambda$-stable, because it is torsion-free of rank $1$ with integral support. 

    It remains to check that it is a monomorphism, i.e. fully faithful. 
    Recall that a morphism in $\srPic{\cF}{\P^5}$ is a pair $(g,\phi) : (T,f,\sG) \to (T',f',\sG')$ of a morphism $g : T \to T'$ of $\P^5$-schemes, and an isomorphism $\phi : (g \times \id)^*(\sG') \isomor \sG$. 
    A morphism in $\sM_{\lambda}$ is a pair $(k,\psi) : (T,\sH) \to (T',\sH')$ of a morphism $k : T \to T'$ of $\C$-schemes, and an isomorphism $\psi :  (k \times \id)^*(\sH') \isomor \sH$.
    Given a pair $(T,\sH) = i_*(T,f,\sG)$, we recover $f : T \to \P^5$ as the classifying map associated to $\Supp\sH \subset T \times X$, and $\sG$ as $i^*\sH$.
    In particular, it is injective on morphisms because $\phi = i^*(\psi)$. 
    To show surjectivity on morphisms, take $(k,\psi) : (T,\sH) \to (T',\sH')$ a morphism in $\sM_{\lambda}$, then 
    \[
    (k \times \id)^{-1}(\Supp \sH') = \Supp \sH \subset U \times X.
    \]
    Therefore, the corresponding classifying maps satisfy $f' \circ k = f$.
    This means that $k$ is a map of $\P^5$-schemes, so $(k,\psi)$ is in the image of $i_*$. 

    Since these stacks are $\G_m$-gerbes over the corresponding coarse moduli spaces (we are dealing with stable sheaves only), the induced map at the level of moduli spaces $\Pic({\cF/\P^5})^{=0} \to \sM_{\lambda}$ is a monomorphism.
    Since the closure of line bundles is proper, we get a closed embedding 
    \begin{equation*}
    \overline{\Pic^0}({\cF/\P^5}) \hookrightarrow M_{\vv_0}(X,\lambda),
    \end{equation*}
    Line bundles supported on smooth Lagrangians are smooth points in $M_{\vv_0}(X,\lambda)$ by \cite{donagi96} (see also \cref{lem:SmoothnessCod3}).  
    Since the tangent space to $M_{\vv_0}(X,\lambda)$ at any such point is $10$-dimensional, the image of \eqref{eq:PicInLagrangians} is an irreducible component by \cref{thm:CompactifiedPicSmooth}.
    If the cubic is general we can apply \cref{thm:CompactifiedPicSmooth} and \cref{prop:CompactifiedPicCM} to conclude.
\end{proof}

Now we show that the construction of \cref{sec:TransformsOfLagrangians} works in families. 
Assume that $X$ is a very general element in some Noether--Lefschetz divisor $\cN(d)$ as in \cref{sec:TheSurfacesOfLines}.
We maintain the notation from there: $\pi: X \to \P^2$ is the Lagrangian fibration, $\pi':X' \to \P^2$ is the dual fibration, and $\Phi_{\sP} : \Db(X) \isomor \Db(X',\theta_d)$ is the twisted derived equivalence given by the Poincaré sheaf. 
We fix a B-field for $\theta_d$, and let $\vv_d$ be the twisted Mukai vector of the $\theta_d$-twisted bundle $\Phi_{\sP}(\sO_Z)$, where $Z \subset X$ is the surface of lines on a hyperplane section. 
We say that a polarization is suitable if it is so for the discriminant of these bundles.

\begin{Lem}\label{lem:LagrangiansToVb}
    Let $\vv_d$ be as above, and let $h' \in \Amp(X')$ be a suitable polarization.
    There is a closed embedding
    \begin{equation}\label{eq:PicInVectorBundles}
         \overline{\Pic^0}(\cF/\P^5) \hookrightarrow M_{\vv_d}(X',h'), \ i_*G \mapsto \Phi_{\sP}(i_*G)
    \end{equation}
    The image is an irreducible component $M^{\circ}_{\vv_d}(X',h')$ consisting of $\mu_{h'}$-stable projectively hyperholomorphic $\theta_d$-twisted locally free sheaves. 
\end{Lem}

\begin{proof}
    We argue as in \cref{lem:RelativePicToLagrangians}, with the added step of the derived equivalence. 
    Consider the morphism of stacks 
    \begin{align*}
          \srPico{\cF}{\P^5} & \to \sM_{\vv_d}(X',h') \\
        (T,f,\sG) &\mapsto (T, \Phi_{\sP,T}(i_*\sG)),
    \end{align*}
    where $ \Phi_{\sP,T} : \Dperf(T \times X) \isomor  \Dperf(T \times X',\pi_2^*\theta) $ is the $T$-linear equivalence obtained by base change of $\Phi_{\sP}$.
    
    It suffices to see that it is well-defined, then one concludes as in \cref{lem:RelativePicToLagrangians}.
    To verify this, we show that the image $\sE \coloneqq \Phi_{\sP,T}(i_*\sG) \in \Dperf(X',\theta)$ is a flat family of twisted vector bundles on $T \times X'$. 
    Notice that, the equivalence $\Phi_{\sP,T}$ is given by relative Fourier--Mukai with the base change of $\sP$, hence the complex $\sE$ is bounded. 
    For every closed point $t \in T$ the derived restriction to $\{ t\} \times X$ is
    \[
    \Phi_{\sP,T}(i_*\sG)|_{\{ t\} \times X} \cong  \Phi_{\sP}(i_*\sG_t),
    \]
    by $T$-linearity.
    The latter is twisted locally free by \cref{cor:VBfromSurfaces}, therefore flatness follows from \cite[Lemma 3.31]{huybrechts2006}.
    In turn, flatness implies that the whole family is locally free by \cite[Lemma 1.27]{simpson94}.
    The rest of the statement follows as in \cref{lem:RelativePicToLagrangians} using \cref{thm:StabilitySecondExample}.
\end{proof}

\begin{Rem}\label{rem:SmoothLocusIsPreserved}
    As before, $M^{\circ}_{\vv_d}(X',h')$ smooth with its reduced structure, but the moduli space may have other components or embedded points. 
    A priori, these components need not to be related to the other components of the moduli space of sheaves supported on Lagrangians, as the morphism $M^{\circ}_{\vv_0}(X,h) \to M_{\vv_d}(X',h')$ is only defined on an irreducible component.
    However, since $\Phi_{\sP}$ is an equivalence, it sends smooth points to smooth points, because to detect smoothness of a moduli space in a stable point it suffices that $\Ext^1$ has the right dimension. 
\end{Rem}

\subsection{Smoothness}
Now we verify the assumptions (A)--(D) of \cref{sec:ParticularResult} to show that the component $M^{\circ}_{\vv_0}(X,h)$ is a connected component consisting of smooth points.  
When otherwise not specified, $X$ is assumed to be general inside some Noether--Lefschetz locus as in \cref{sec:TransformsOfLagrangians}.
By construction there is a support morphism $f: M^{\circ}_{\vv_0}(X,h) \to \P^5$ sending a sheaf $G$ supported on some $Z \subset X$ to the corresponding $Z$. 
We let 
\[
    M_{\rm{nod}} =\{ G \in M^{\circ}_{\vv_0}(X,h) \mid \text{either } G \text{ is locally free on } Z, \text{ or } Z \text{ is nodal} \},
\]
where by nodal we mean that $Z$ is the surface of lines on a cubic threefold with only a nodal singularity. 
The following is assumption (B). 

\begin{Lem}\label{lem:SmoothnessCod3}
    The locus $M_{\rm{nod}} \subset M^{\circ}_{\vv_0}(X,h)$ is open with complement of codimension $3$, and it is contained in the smooth locus.
\end{Lem}

\begin{proof}
    Let $U_1 \subset \P^5$ be the open locus of hyperplane sections with at worst one node.
    Its complement $U^c_1 \subset \P^5$ has codimension $2$. 
    The complement of $M_{\rm{nod}}$ is contained in $f^{-1}(U^c_1)$, and over every point it is a proper closed subset of the fiber.
    Since $f$ has integral fibers by \cref{rem:FibersOfRelativePic}, we conclude that $M_{\rm{nod}}$ is open and its complement has codimension at least $3$. 
    To prove smoothness it suffices to see that for every $G \in M_{\rm{nod}}$ we have $\ext^1(G,G) = 10$.
    
    \textbf{Locally free sheaves.}
    Assume that $G$ is locally free over its support, which may be singular but it is integral. 
    We may assume that $G = i_*\sO_Z$ by \cref{lem:IndependenceOnL}, where $i: Z \hookrightarrow X$ is the inclusion.
    Since $Z \subset X$ is the zero locus of a regular section of $\sS^{\vee}$, we get
    \[
        \Ext^k(i_*\sO_Z,i_*\sO_Z) =  \bigoplus_{p + q = k} H^p(Z,\sS|_Z^{\vee}),
    \]
    taking the Koszul resolution. 
    Since $\sS|_Z^{\vee} = N_{Z/X}$ this gives
    \[
         \ext^1(i_*\sO_Z,i_*\sO_Z) = h^1(Z,\sO_Z) + h^0(Z,N_{Z/X}) = 10.
    \]
    Both the summands above are $5$, the first one is \cite[Proposition 1.15]{altman77}, and the latter is simply the smoothness (at the point $Z$) of the Hilbert scheme of deformations of $Z$.

    \textbf{Nodal cubics.}
    Now take $G$ any sheaf supported on a nodal $Z$.
    Then by \cref{prop:CompactifiedPicCM} $G$ is the pushforward of a Cohen--Macaulay sheaf on $Z$, and by \cref{cor:ExtOnNodal} we get $\ext^1(G,G) \leq 11$.
    By \cref{cor:PoincaréSendsAtomicToAtomic} and \cref{thm:StabilitySecondExample}, the image $\Phi_{\sP}(G)$ is a slope stable twisted atomic vector bundle. 
    Therefore, \cref{prop:TwistedAtomicHyperholomorphic} and \cref{lem:SymplecticityOfExt1} imply that $\Ext^1(G,G)$ has even dimension, which must be $10$.
\end{proof}

\begin{Cor}\label{cor:ModularLSV}
    Let $X$ be the Fano variety of lines on a general cubic fourfold, $h$ be the Pl\"ucker polarization, and $\vv_0 = v(\sO_Z)$, and let $M^{\circ}_{\vv_0}(X,h) \subset M_{\vv_0}(X,h)$ be as above. 
    Then, $M^{\circ}_{\vv_0}(X,h)$ is a connected component of $M_{\vv_0}(X,h)$.
    Moreover, $M^{\circ}_{\vv_0}(X,h)$ is a smooth HK manifold of type OG10 isomorphic to $J_Y$, and the support morphism $M^{\circ}_{\vv_0}(X,h) \to \P^5$ is a Lagrangian fibration. 
\end{Cor}

\begin{proof}
    First fix a Noether--Lefschetz divisor $\cN(d)$ as in \cref{sec:TheSurfacesOfLines}, and take $X$ very general in it.
    Let $M^{\circ}_{\vv_d}(X',h')$ be the component of \cref{lem:LagrangiansToVb}.
    We first prove that it consists of smooth points, and for this we verify the assumptions (A)--(D) of \cref{sec:ParticularResult}.

    Assumption (A) is satisfied by \cref{thm:StabilitySecondExample}.
    Assumption (B) follows from \cref{lem:SmoothnessCod3} and \cref{rem:SmoothLocusIsPreserved}.
    Assumption (C) follows from the isomorphism $M^{\circ}_{\vv_0}(X,h) \cong M^{\circ}_{\vv_d}(X',h')$ and \cref{thm:CompactifiedPicSmooth}.
    Lastly, for (D) notice that if $L$ is a line bundle on a smooth Lagrangian surface $Z$, then there is an algebra isomorphism
    \[
        \Ext^*(i_*L,i_*L) \cong H^*(Z,\C),
    \]
    by \cite{mladenov19}.
    Since for $Z = F(Y_H)$ the cup product induces an isomorphism $\bwed^2 H^1(Z,\C) \isomor H^2(Z,\C)$, we get (D).
    \cref{thm:SmoothnessOurCase} implies that $M^{\circ}_{\vv_d}(X',h')$ is a smooth connected component of $M_{\vv_d}(X',h')$.
    Therefore by \cref{rem:SmoothLocusIsPreserved}, it follows that $M^{\circ}_{\vv_0}(X,h)$ consists of smooth points, so it is a connected component. 

    It is left to deform to a general cubic.
    Let $\cC$ be the moduli space of smooth cubic fourfolds (it embeds as an open subset in $\cK^2_6$), let $\cX \to \cC$ be the universal variety of lines.
    By abuse of notation denote by $h$ also the relative Pl\"ucker polarization.
    Let $\cM_{\vv_0}(\cX/\cC,h) \to \cC$ be the relative moduli space of $h$-semistable sheaves with Mukai vector $\vv_0 = v(\sO_Z)$, and let $\cM^{\circ}_{\vv_0}(\cX/\cC,h)$ be the closure of the relative Picard.
    The embedding of \cref{lem:RelativePicToLagrangians} works over an open of $\cC$, and the same is true for \cref{rem:FibersOfRelativePic}.
    Thus $\rho: \cM^{\circ}_{\vv_0}(\cX/\cC,h) \to \cC$ has integral fibers.
    
    By the first part of the proof, $\rho$ is dominant, and, being projective, it is also surjective.
    By generic flatness, up to taking an open in $\cC$ we may assume that $\rho$ is flat.
    By the first part of the proof, there is a dense subset (in particular, non-empty) of $\cC$ over which the fibers of $\rho$ are smooth. 
    Therefore, the critical locus $\rm{Crit}(\rho)$ does not dominate $\cC$. 
    Thus over the complement of its image in $\cC$, the family $\rho: \cM^{\circ}_{\vv_0}(\cX/\cC,h) \to \cC$ is smooth and projective.
    If a cubic $Y \in \cC $ is general (so possibly we have to further restrict $\cC$), by \cref{thm:CompactifiedPicSmooth} we have an isomorphism $M^{\circ}_{\vv_0}(X,h) \cong J_Y$.
    The fact that the support morphism is a Lagrangian fibration is obvious. 
\end{proof}

\begin{Rem}\label{rem:ChangingPolarization}
    As explained in \cref{lem:RelativePicToLagrangians}, for a general cubic the polarization plays no role.
    Therefore, we obtain \cref{thm:IntroLagrangians} for any polarization $\lambda$. 
\end{Rem}

\subsubsection{Deformations along twistor lines} 
Before delving into the proof of \cref{thm:IntroBundles}, we recall how to deform twisted hyperholomorphic bundles along a twistor line.
For reference see \cite[Section 6.2]{markman20}, and \cite[Section 4.4]{peregoToma17}.

Let $X$ be a projective HK manifold, and let $\cX \to \P^1$ be the twistor space for $\cX_0 = X$ with respect to the K\"ahler class $\omega_X$. Let $\theta \in \Br(X)$ a Brauer class and let $E$ be a $\theta$-twisted projectively hyperholomorphic bundle on $X$. 
Its endomorphism bundle $\cA = \cEnd(E)$ is an Azumaya algebra representing the twist $\theta$.
In fact, it represents a lift of $\theta$ to $H^2(X,\mu_r)$, where $r$ is the rank of $E$. 

Since $\cA$ is polystable hyperholomorphic it deforms to a bundle $\wt{\cA}$ on $\cX$. 
One can check that the multiplication maps deforms as well, so  $\wt{\cA}_t$ is an Azumaya algebra for every $t \in \P^1$.
The corresponding twist $\theta_t \in H^2(\cX_t,\mu_r)$, is just $\theta$ via the identification $H^2(X,\mu_r) = H^2(\cX_t,\mu_r)$.
Seeing a $\theta$-twisted projectively hyperholomorphic bundle as a module over $\cA$, we can deform it as a $\theta_t$-twisted bundle by deforming the module structure.

We can also keep track of how the twisted Mukai vector changes along this deformation. 
Via the lift to $H^2(X,\mu_r)$, we choose a B-field $B \in H^2(X,\Q)$ such that $rB \in H^2(X,\Z)$. 
Explicitly, if $\theta^r = 1$ at the level of cocycles, we can find a $\theta$-twisted topological line bundle $L$ such that $L^r$ is untwisted on the nose, and $B = \frac{c_1(L^r)}{r}$ is a B-field for $\theta$.
Under parallel transport $H^2(X,\Q) = H^2(\cX_t,\Q)$, we obtain a B-field $B_t$ for $\theta_t$ for every $t \in \P^1$. 
By definition, for every $\theta$-twisted hyperholomorphic bundle, we have
\[
    v^B(E) = v^{B_t}(E_t),
\]
under the identification $H^*(X,\Q) = H^*(\cX_t,\Q)$.

\begin{Rem}\label{rem:MukaiVectorOfUntwist}
    By definition of B-field, if there is a $t \in \P^1$ for which $B_t$ is algebraic, then $\theta_t$ is trivial.
    Equivalently, if we choose a topologically trivial line bundle $L$ on $X$ as above, then it becomes algebraic on $\cX_t$. 
    This means that we can use it to canonically untwist $\theta_t$-twisted bundles on $\cX_t$.
    The untwisted bundle corresponding to $E_t$ is $E_t \otimes L^{-1}$, and, by definition, its Mukai vector is
    \[
         v(E_t \otimes L^{-1}) = v^{B_t}(E_t) \in H^*(\cX_t,\Q).
    \]
    Therefore, over $t \in \P^1$, we obtain a canonically\footnote{Depending only on the choice of a $\theta$-twisted topological line bundle $L$ on $X$.} untwisted bundle $E_t \otimes L^{-1}$ with Mukai vector given by the parallel transport of $v^B(E)$. 
\end{Rem}

\subsubsection{Proof of the \cref{thm:IntroBundles}}
Now take $X$ very general in some Noether--Lefschetz locus $\cN(d)$ as in \cref{sec:TheSurfacesOfLines}. 
Consider the corresponding Brauer class $\theta_d \in \Br(X')$, and choose a B-field $B$ representing $\theta_d$, as in the discussion above.

Up to tensoring every $E \in M^{\circ}_{\vv_d}(X',h')$ with a sufficiently ample line bundle on $X'$, we may assume that $c^{B}_1(E)$ has positive square $q(c^{B}_1(E)) > 0$. 
This operation does not change the $\mu$-stability (if we change the polarization accordingly), nor the atomicity of $E$. 
It does change their Mukai vector, which, by abuse of notation, we still denote by $\vv_d$.

By the twisted version of \cite[Lemma 5.4]{beckmann22} (see also the proof of \cref{prop:TwistedAtomicHyperholomorphic}), the extended Mukai vector of $E \in M^{\circ}_{\vv_d}(X',h')$ can be written as 
\[
    \tilde{v}^{B}(E) = 5d^2 \alpha + m_dc_d + s_d\beta \in \wt{H}(X',\Q)
\]
for some $s_d \in \Q$, where $c^{B}_1(E) = m_dc_d$ with $m_d \in \Z$ and $c_d \in H^2(X,\Z)$ primitive.
Fix a marking $\eta_{X'}: H^2(X',\Z) \cong \Lambda_{\mathrm{K3}^{[n]}}$, and let $\cK_d$ be an irreducible component of the moduli space of polarized HK manifolds of polarization type given by $\eta_{X'}(c_d)$.

\begin{Rem}
For $(W,h_d) \in \cK_d$, if a parallel transport operator $H^*(X',\Q) \cong H^*(W,\Q)$ sends $c_d$ to $h_d$, then the extended Mukai vector $\tilde{v}^{B}(E)$ deforms to 
\[
    \tilde{v}_d = 5d^2 \alpha + m_dh_d + s_d\beta \in \wt{H}(W,\Q).
\]
Therefore, if $T: \Sym^2 \wt{H}(X,\Q) \to H^*(X,\Q)$ denotes the LLV-equivariant projection, the Mukai vector $\vv_d$ is deformed to 
\begin{equation}\label{eq:DefOfMukaiVector}
    \vv_d = \frac{1}{10d^2}T(\tilde{v}_d^{(2)}) = 5d^2 +  m_dh_d + \dots \in H^*(W,\Q),
\end{equation}
see \cite[Corollary 3.10]{bottini22} for the complete form.
\end{Rem}

\begin{Thm}\label{thm:ModuliOfVB}
    If $(W,h_d) \in \cK_d$ is very general, the moduli space $M_{\vv_d}(W,h_d)$ contains a connected component $M^{\circ}_{\vv_d}(W,h_d)$ satisfying the following.
    \begin{enumerate}
        \item $M^{\circ}_{\vv_d}(W,h_d)$ is a smooth projective HK manifold of type OG10. 
        \item $M^{\circ}_{\vv_d}(W,h_d)$ parametrizes only $\mu_{h_d}$-stable atomic bundles.
    \end{enumerate}
\end{Thm}


\begin{proof}
    Start with $(X',h')$ and the component $M^{\circ}_{\vv_d}(X',h')$ parametrizing $\theta_d$-twisted atomic $\mu_{h'}$- bundles. 
    Let $a_d \coloneqq a(\cEnd(E))$ for some $E \in M^{\circ}_{\vv_d}(X',h')$, the constant used in \cref{sec:SlopeStability}, so that $h'$ is $a_d$-suitable. 
    Since the walls in the K\"ahler cone are algebraic, we can find a K\"ahler class $\omega_{X'}$ such that:
    \begin{itemize}
        \item Every $E \in M^{\circ}_{\vv_d}(X',h')$ is $\mu_{\omega_{X'}}$-stable. 
        \item The corresponding twistor line is general.
    \end{itemize}
    Unless $h_d$ is on an $a_d$-wall---a situation we avoid by assuming $(W,h_d)$ very general---we can do the same for $(W,h_d)$.
    Thus we obtain a K\"ahler class $\omega_W$, whose twistor line is general, and such that the $\mu_{\omega_W}$-stability is equivalent to the $\mu_{h_d}$-stability.
    Choose a marking $\eta_{W}$ so that $(X',\eta_{X'})$ and $(W,\eta_{W})$ are in the same connected component of the moduli space of marked HK's, and such that $\eta_{X'}(c_d) = \eta_W(h_d)$. 
    By applying \cite[Theorem 3.2 and 5.2e]{verbitsky96a} to general points in the twistor lines given by $\omega_{X'}$ and $\omega_W$ respectively, we find a general twistor path connecting $(X',\eta_{X'})$ and $(W,\eta_W)$ whose first and last twistor lines are given by $\omega_{X'}$ and $\omega_W$. 

    Now we explain how the component $M^{\circ}_{\vv_d}(X',h')$ deforms along the twistor path. 
    As usual, it is done by induction on the number of twistor lines, so we focus on the first one.
    Let $p: \cX \to \P^1$ be the holomorphic projection from the twistor space, and let $\tilde{\omega}$ be the deformation of $\omega_{X'}$. 
    Choose any bundle $F \in M^{\circ}_{\vv_d}(X',h') $ and deform it as explained above.
    The projectivization of the deformation gives a relative Brauer--Severi variety $\cP \to \cX$, such that for every $t \in \P^1$ we have $\cP_t = \P(F_t)$. 
    Then, the relative moduli space ${\cM}_{\vv_d}({\cX/\P^1},\tilde{\omega}) \to \P^1$ of twisted semistable sheaves on $\cX \to \P^1$ is constructed (as an complex space) via a relative moduli space of stable bundles on $\cP$, see \cite[Sections 4.3.1 and 4.3.2]{peregoToma17} for more details.\footnote{They work over K3 surfaces, but they only use it to prove smoothness.}
    
    If, for a bundle $E \in M^{\circ}_{\vv_d}(X',h')$, we denote by $E_t$ its deformation to $X_t = \cX|_t$, we get a differentiable embedding 
    \begin{equation}\label{eq:DeformationAlongTwistor}
        M^{\circ}_{\vv_d}(X',h') \times \P^1 \hookrightarrow {\cM}_{\vv_d}({\cX/\P^1},\tilde{\omega}), \ (E,t) \to E_t.
    \end{equation}
    The cohomology of hyperholomorphic bundles remains constant along a twistor deformation by \cite[Corollary 8.1]{verbitsky96b}, thus $\ext^1(E,E) = \ext^1(E_t,E_t)$ for every $E \in M^{\circ}_{\vv_d}(X',h')$.
    In particular, for every $t$, we obtain a smooth compact component  $M^{\circ}_{\vv_d}(X_t,\tilde{\omega}|_{X_t}) \subset \cM_{\vv_d}({\cX/\P^1},\tilde{\omega})|_t$. 
    Finally, as every fiber $ M^{\circ}_{\vv_d}(X_t,\tilde{\omega}|_{X_t})$ is a smooth moduli space parametrizing slope stable vector bundles on $\cP_t$, it is equipped with a natural K\"ahler metric thanks to the Kobayashi--Hitchin correspondence and gauge theory \cite[Theorem 7.6.36]{kobayashi87}.
    We conclude that the image of the embedding in \cref{eq:DeformationAlongTwistor} is a smooth proper family $ {\cM}^{\circ}_{\vv_d}({\cX/\P^1},\tilde{\omega}) \to \P^1$ with K\"ahler fibers. 
    Moreover, every fiber is a K\"ahler deformation a HK manifold of type OG10, so it is itself a HK manifold of type OG10.
    
    Continuing along the twistor path, we end up with a component $M^{\circ}_{\vv_d}(W,h_d)$, which is a HK manifold of type OG10.
    By the discussion above, it parametrizes untwisted atomic bundles, with Mukai vector $\vv_d$ by \cref{rem:MukaiVectorOfUntwist}.
    They are $\mu_{\omega_W}$-stable, hence $\mu_{h_d}$-stable by our choice of $\omega_W$. 
\end{proof}

\begin{Cor}\label{cor:FamiliesOfOG10}
    There exists an irreducible quasi projective variety $\cS_d$ with a dominant quasi-finite morphism $\cS_d \to \cK_d$, and a smooth projective morphism $\cM_d \to \cS_d$ whose fibers are HK manifolds of type OG10. 
    Moreover, if $s \in \cS_d$ is general and $(W,h_d)$ is its image, the fiber $\cM_d|_{s}$ parametrizes $\mu_{h_d}$-stable atomic bundles on $W$. 
\end{Cor}

\begin{proof}
Let $\cX_d \to \cK_d$ a universal family.
By abuse of notation, denote by $h_d$ also the relative polarization. 
Let  $ \cM_{\vv_d}(\cX/{\cK_d},h_d) \to \cK_d$ be the relative moduli space Gieseker semistable sheaves with Mukai vector $\vv_d$.
It is projective over $\cK_d$, and by \cref{thm:ModuliOfVB} dominant, so it is surjective.  
By Stein factorization we factor it as a projective morphism $\rho: \cM_{\vv_d}(\cX/{\cK_d},h_d) \to \cS_d$ with connected fibers, and a finite morphism $\cS_d \to \cK_d$.

We replace $\cS_d$ by one irreducible component which dominates $\cK_d$, and contains a point $s \in \cS_d$ such that $\rho^{-1}(s) = M^{\circ}_{\vv_d}(W,h_d)$ as in \cref{thm:ModuliOfVB}.
Such a component must exist, because the locus of $(W,h_d)$ as in \cref{thm:ModuliOfVB} is dense.
By generic flatness, up to taking an open in $\cS_d$ we may assume that $\rho$ is flat and surjective. 
Since $\rho^{-1}(s)$ is smooth, the critical locus $\rm{Crit}(\rho)$ does not dominate $\cS_d$. 
Thus, up to taking an even smaller open in $\cS_d$, the morphism $\rho: \cM_{\vv_d}(\cX/{\cK_d},h_d) \to \cS_d$ has smooth fibers, which is the desired polarized family $\cM_d \to \cS_d$.
The last part of the statement follows in the same way, using the openness of the locally free and of the slope stable loci. 
\end{proof}

\appendix
\section{Functoriality of the compactified Picard}\label{sec:functorialityPicard}
Here we collect some technical results we use in \cref{sec:compactifiedPic} to construct morphisms between relative compactified Picard schemes. 
As in \cref{sec:compactifiedPic}, we follow \cite{altman79,altman80}. 

\begin{Prop}\label{prop:PullbackOfCM}
    Let $X \to S$ and $Y \to S$ be two flat morphisms between Noetherian schemes of finite type, with geometrically integral Gorenstein fibers.
    Let $f : X \to  Y$ be a morphism over $S$, such that for every closed point $s \in S$ the induced morphism $f_s : X_s \to  Y_s$ has finite tor dimension. 
    Then there is a morphism of schemes 
    \[
    f^* : \Pic(Y/S)^- \to \Pic(X/S)^-.
    \]
\end{Prop}

\begin{proof}
    It suffices to see that $f^*$ is well defined at the level of the moduli functors. 
    Let $T \to S$ be any morphism and consider the base change $ f_T : X_T \to Y_T.$
    Both $X_T$ and $Y_T$ are flat over $T$, and for every closed point $t \in T$ the restriction $f_t : X_t \to Y_t$ has finite tor dimension. 
    Let $\cU \in \Coh(Y_T)$ be a $T$-flat family of Cohen--Macaulay sheaves of rank 1.
    We want to prove that $f_T^*(\cU)$ (here, the pullback is not derived) is again $T$-flat with Cohen--Macaulay fibers of rank 1.
    
    We start by proving flatness. 
    Let $t \in T$ be a closed point and consider the square 
    \[\begin{tikzcd}
	   {X_t} & {Y_t} \\
	   X & Y
	   \arrow["{\mu_t}", hook, from=1-1, to=2-1]
	   \arrow["{\nu_t}", hook, from=1-2, to=2-2]
	   \arrow["{f_t}", from=1-1, to=1-2]
	   \arrow["f", from=2-1, to=2-2]
    \end{tikzcd}\]
    By commutativity we have 
    \begin{equation}\label{eq:CommutativityPullbacks}
        \sfL\mu_t^* \circ \sfL f^* = \sfL f_t^* \circ \sfL\nu_t^*.
    \end{equation}
    We examine the composition on the right applied to $\cU$. 
    Since $\cU$ is flat over $T$, we have that $L\nu_t^*(\cU) = \nu_t^*(\cU) = \cU_t$ is a sheaf. 
    Since the restriction $\cU_t$ is Cohen--Macaulay and $f_t$ has finite tor dimension, the pullback $\sfL f_t^*(\cU) = f_t^*(\cU)$ is a sheaf by \cite[Lemma 2.3]{arinkin13}.
    Therefore \eqref{eq:CommutativityPullbacks} implies that $  \sfL \mu_t^* ( \sfL f^*(\cU)) $ is a sheaf for every closed point $t \in T$. 
    We conclude by \cite[Lemma 3.31]{huybrechts2006} that $\sfL f^*(\cU)$ is a $T$-flat sheaf. 
    Notice that, although in the assumptions of  \cite[Lemma 3.31]{huybrechts2006} it is required that $\sfL f^*(\cU) \in \Db(X)$, for the proof it suffices that it is a bounded above complex, as in this case.

    It remains to see that $f^*(\cU)$ is Cohen--Macaulay of rank 1 on every fiber. 
    This follows again by \cite[Lemma 2.3]{arinkin13}, as we are assuming that every fiber is Gorenstein. 
\end{proof}

Recall that $f : X \to Y$ is called local complete intersection if it locally factors as the composition of a smooth morphism and a regular embedding, see \cite[Section 37.60]{stacks}.\footnote{All our schemes are Noetherian, so Koszul-regular regular sequences are regular.}
It follows from the definition that such a morphism has finite tor dimension.

\begin{Lem}\label{lem:LciOnFibers}
Let $S$ be a smooth variety.
Let $X \to S$ and $Y \to S$ two Noetherian schemes of finite type, flat over $S$.
Let $f :X \to Y$ a local complete intersection morphism of $S$-schemes. 
Then, for every $s \in S$, 
\[
f_s : X_s \to Y_s
\]
is a local complete intersection morphism.  
In particular, it has finite tor dimension. 
\end{Lem}

\begin{proof}
The question is local on the source and target, hence we can assume that $f : X \hookrightarrow Y$ is a regular embedding, after possibly renaming $Y$. 
Taking the fiber over a point $s \in S$, we obtain the Cartesian diagram
\[\begin{tikzcd}
	{X_s} & {Y_s} \\
	X & Y
	\arrow[hook, from=1-1, to=2-1]
	\arrow["f", hook, from=2-1, to=2-2]
	\arrow["f_s", hook, from=1-1, to=1-2]
	\arrow[hook, from=1-2, to=2-2]
\end{tikzcd}\]
The vertical arrows are regular embeddings, because $X \to S$ and $Y \to S$ are flat and $S$ is regular.
The bottom arrow is a regular embedding by assumption. 

We conclude since $X_s$ is the fiber product of two regular embeddings and it has the right dimension.
More precisely, let $s \in S$ and $x \in X_s \subset Y_s$, and say that $R \coloneqq \sO_{Y,x}$. 
Let $(g_1,\dots, g_n) \subset R$ and $(h_1,\dots,h_m) \subset R$ be the regular sequences cutting $X$ and $Y_s$, respectively.
The key point is that $X_s \subset X$ is cut by (the restriction to $X$ of) the same equations $(h_1,\dots,h_m)$ cutting $Y_s \subset Y$: the pullback of the equations for $s \in S$.

The sequence $(g_1,\dots,g_n,h_1,\dots,h_m)$ cuts $X_s \subset Y$; thus it is regular in $R$. 
The same is true for the sequence $(h_1,\dots,h_m,g_1,\dots,g_n) \subset R$, because for local rings regularity is independent of the order. 
We conclude by \cite[Lemma 15.30.14]{stacks} that the images $(\overline{g_1},\dots,\overline{g_n})$ are regular in $R/(h_1,\dots,h_m)$, therefore $X_s \subset Y_s$ is a regular embedding. 
\end{proof}


\end{document}